\documentclass[a4paper]{article}
\usepackage[utf8]{inputenc}

\usepackage{inputenc}
\usepackage{graphics}
\usepackage{epsfig}
\usepackage{graphicx}
\usepackage{latexsym}
\usepackage{graphics}
\usepackage{epsfig}
\usepackage{amsmath,amssymb,amsfonts,theorem,color,bm}
\usepackage{multirow}
\usepackage{verbdef}
\usepackage{mathrsfs}
\usepackage{hyperref}
\usepackage{subcaption}
\usepackage{epstopdf}
\usepackage{float}
\newcommand{\bs}{\boldsymbol}

\newcommand{\beqn}{\begin{equation}}
\newcommand{\eeqn}{\end{equation}}

\newtheorem{corollary}{Corollary}[section]
\newtheorem{theorem}{Theorem}[section]

\def\blackbox{\leavevmode\vrule height 5pt width 4pt depth 0pt\relax}
\newenvironment{proof}{\begin{trivlist}
\item[]\hspace{0cm}{\bf Proof:}
\hspace{0cm} }{\hfill $\blackbox$
\end{trivlist}}

\def\R{{\mathbb R}}

\def\N{{\mathbb N}}

\newcommand\restr[2]{{
\left.\kern-\nulldelimiterspace 
#1 
\right|_{#2} 
}}
\usepackage{fix-cm}

\usepackage{amsfonts}

\usepackage{mathtools}
 \usepackage{xcolor}
 \usepackage{listings}
\usepackage{amssymb}
\usepackage{todonotes}
\usepackage{multirow}

\newcommand{\rank}{\mathrm{rank}}
\newcommand{\ns}{N_{\kappa}(\Omega)}

\def\R{{\mathbb R}}

\def\N{{\mathbb N}}

\newcommand{\mK}{\mathsf{K}}
\newcommand{\mA}{\mathsf{A}}

\newcommand{\mI}{\mathsf{I}}

\usepackage[affil-it]{authblk}
\usepackage{hyperref}

\usepackage{soul}

\usepackage{pgfplots} 				
\usepackage{tikz}
\usepackage{tikzscale}
\usepackage{environ}
\newlength\fwidth

\usepackage[linesnumbered]{algorithm2e}

\definecolor{codegreen}{rgb}{0,0.6,0}
\definecolor{codegray}{rgb}{0.5,0.5,0.5}
\definecolor{codepurple}{rgb}{0.58,0,0.82}
\definecolor{backcolour}{rgb}{0.95,0.95,0.92}

\lstdefinestyle{mystyle}{
  backgroundcolor=\color{backcolour},  
  commentstyle=\color{codegreen},
  keywordstyle=\color{magenta},
  numberstyle=\tiny\color{codegray},
  stringstyle=\color{codepurple},
  basicstyle=\ttfamily\footnotesize,
  breakatwhitespace=false,     
  breaklines=true,         
  captionpos=b,          
  keepspaces=true,         
  numbers=left,          
  numbersep=5pt,         
  showspaces=false,        
  showstringspaces=false,
  showtabs=false,         
  tabsize=2
}

\DeclareMathOperator*{\argmax}{argmax}

\title{Data-driven kernel designs for optimized greedy schemes:\\ A machine learning perspective}

\author[1]{Tizian Wenzel \thanks{tizian.wenzel@mathematik.uni-stuttgart.de, corresponding author}}

\author[2]{Francesco Marchetti \thanks{francesco.marchetti@unipd.it}}

\author[3]{Emma Perracchione \thanks{emma.perracchione@polito.it}}

\affil[1]{Institute for Applied Analysis and Numerical Simulation, University of Stuttgart, Germany}
\affil[2]{Dipartimento di Matematica “Tullio Levi-Civita”, Università di Padova, Italy}
\affil[3]{Dipartimento di Scienze Matematiche Giuseppe Luigi Lagrange, Politecnico di Torino, Italy}

\newif\ifplots
\plotstrue

\begin{document}

\maketitle

\begin{abstract}
Thanks to their easy implementation via Radial Basis Functions (RBFs), meshfree kernel methods have been proved to be an effective tool for e.g. scattered data interpolation, PDE collocation, classification and regression tasks.
Their accuracy might depend on a length scale hyperparameter, which is often tuned via cross validation schemes. 
Here we leverage approaches and tools from the machine learning community to introduce two-layered kernel machines, which generalize the classical RBF approaches that rely on a single hyperparameter. 
Indeed, the proposed learning strategy returns a kernel that is optimized not only in the Euclidean directions, but that further incorporates kernel rotations. 
The kernel optimization is shown to be robust by using recently improved calculations of cross validation scores.
Finally, the use of greedy approaches, and specifically of the Vectorial Kernel Orthogonal Greedy Algorithm (VKOGA), allows us to construct an optimized basis that adapts to the data. 
Beyond a rigorous analysis on the convergence of the so-constructed two-Layered (2L)-VKOGA, its benefits are highlighted on both synthesized and real benchmark data sets.
\end{abstract}

\section{Introduction}

Kernel methods \cite{fasshauer2015kernel,wendland2005scattered} are an active field of research due to their wide applicability in several tasks like machine learning, approximation theory or numerical analysis.
In order to increase the accuracy of kernel methods and to avoid instability issues at the same time, 
many researchers already worked on the problem of selecting suitable center points, e.g. via greedy kernel methods, and on the computational issue of finding suitable values for the kernel hyperparameter, namely the shape parameter in Radial Basis Functions (RBFs) literature.

As far as the first item is concerned, assuming that a set of measurements sampled at multivariate scattered points is given, the basic idea of greedy kernel methods \cite{wenzel2022analysis,wirtz2015surrogate} consists in selecting only a smaller subset (called \textit{centers}) out of the large training set of scattered points to build the kernel model.
For this, one starts with an empty set of centers and then adds one more center, step by step, according to some selection criterion \cite{marchi2005near, dutta2021greedy, muller2009komplexitat, schaback2000adaptive}, 
until a predefined amount of centers or some other stopping criterion is met. 
In this context, we will employ the Vectorial Kernel Orthogonal Greedy Algorithm (VKOGA) \cite{santin2019kernel,wirtz2013vectorial}, 
which implements several greedy algorithms in an efficient way. 

Moreover, besides greedy methods, we point out that in the literature there is quite a lot of work on selecting or optimizing a single hyperparameter for RBFs, 
see e.g. \cite[\S 14]{fasshauer2015kernel} or \cite{wendland2005scattered} for a general overview and e.g.\ \cite{driscoll2002interpolation, fornberg2004stable, mccourt2013using} for specific instances that comprise  heuristic approaches, cross-validation methods or schemes based on maximum likelihood estimation.
Such a hyperparameter determines the shape of the basis function, which can be more \emph{picked} or more \emph{flat}. 
This radial scaling works fine as long as the data have similar scales along the Euclidean directions. 
Some improvement in this sense is provided by the so-called anisotropic kernels that need the selection of several scaling parameters, 
precisely as many as the problem dimension (see \cite[\S 3]{fasshauer2015kernel} for a general discussion on the topic). 
Unfortunately, anisotropic kernels in the classical sense do not allow any kernel rotation, and hence we here face the more challenging problem of learning the metric from data, thus allowing the presence of rotations. 
This issue is known as distance metric learning, and several algorithms, as the collapsing classes method, have been developed to solve computational issues e.g.\ in KNN classification (see  \cite{aiolli2014learning, globerson2005metric, goldberger2004neighbourhood, shalev2004online}). 
A different approach to generate anisotropic bases is represented by the so-called variably scaled kernels, 
which demonstrated effectiveness in encoding steep gradients and discontinuities \cite{campi2021learning, marchi2020jumping, Perracchione_2021, rossini2018interpolating}. 

More recent research tries to predict suitable kernel shape parameters with the help of machine learning methods, see e.g.\ \cite{mojarrad2022new}. Another related approach is proposed in \cite{otto2022rffnet}, where random Fourier features are used  for automatic relevance determination of features, however this does not incorporate rotated features. 
A more sophisticated tool, which can be seen as very deep kernels, is provided under the notion of \textit{kernel flows}, see \cite{owhadi2019kernel}. 
There the authors learn a non-parametric family of deep kernels of a given form via incremental data-dependent deformations, thus obtaining very deep (bottomless) kernels. 
Their optimization method is related to ours, however based on a different error criterion.
See also Subsection \ref{subsec:loss_func_for_optim} for some more comments on the relation.
Moreover, \cite{hamzi2021learning} advances by considering cross-validation {criteria} based on maximal Lyapunov exponent or Maximum Mean Discrepancy (MMD). 
In contrast to their approaches, we make use of a two-layered kernel and give a clear interpretation in terms of a hyperparameter optimized kernel instead of using extremely deep (bottomless) kernels. 
Another somehow related topic is the \textit{active subspaces} \cite{constantine2015active,constantine2014active}.
Here the idea is to identify active subspaces, along which the target function changes, and inactive subspaces, along which the target function is invariant. 
However this is done using gradient information, which is usually not available in scattered data approximation.
Furthermore such active subspaces can in principle be also detected with our approach, as shown later in Section \ref{sec:two-layered}.

\noindent In the above setting our main contributions are: 
\begin{enumerate}
\item Developing a two-layered kernel machine for \emph{optimal} data driven kernel designs, i.e. we learn the metric from the samples; 
\item Using such a kernel design to construct a sparse basis via greedy methods.
\end{enumerate}
These two steps allow us to obtain more efficient kernel models for machine learning than previous state-of-the-art greedy methods. For the second step, we employ the VKOGA algorithm \cite{santin2019kernel,wirtz2013vectorial}, 
which implements several greedy algorithms in an efficient way. Despite we here focus on scalar valued outputs instead of vector valued ones, in order to keep the famed acronym VKOGA, 
we denote (with some abuse of notation) our resulting two-layered greedy kernel machine as two-Layered (2L)-VKOGA. 
As far as the first above item is concerned, we go beyond state-of-the art literature by using both kernel scalings and rotations and provide a helpful interpretation based on a Deep Kernel representer theorem as two-layered kernel.
Finally, leveraging recent advancements on the computation of $k$-fold cross validation scores we provide efficient optimization procedures. 
Precisely, the classical hyperparameter is replaced by a more general linear mapping and we prove that this turns out to be equivalent to a two-layered kernel machine.

{We then analyze} both theoretical and computational aspects of the proposed 2L-VKOGA. 
{The convergence analysis shows that we may obtain faster rates of convergence and this depends on the singular values of the optimized first layer linear mapping matrix. 
Numerical evidence stresses that the use of the two-layered optimized kernel possibly significantly outperforms state-of-the art techniques,
and we furthermore provide a practical criterion that allows us to discriminate when and when not we should expect some benefit. Such criterion can be implemented already after the kernel optimization, 
i.e.\ before running the slightly more expensive greedy selection.} Furthermore, we are also able to prove that the complexity of our approach turns out to be way cheaper than a straightforward cross-validation of kernel length scale parameters.

The outline of the paper is as follows. In Section \ref{sec:background} we briefly review the basics of kernel methods and we introduce the main tools needed to develop the 2L-VKOGA scheme. 
Section \ref{sec:two-layered} is entirely focused on the two-layered kernel machines, both from a theoretical and a computationally-oriented viewpoint. 
Numerical experiments with both synthetic and real datasets are carried out in Section \ref{sec:num_experiments}. 
Conclusions and possible further developments are presented in Section \ref{sec:conclusion}.

\section{Background information}
\label{sec:background}

In the following Subsection \ref{subsec:kernel_approx} we review the basics of kernel methods \cite{fasshauer2015kernel, wendland2005scattered}, 
while in the subsequent Subsection \ref{subsec:greedy_kernel} we briefly introduce \textit{greedy} kernel methods \cite{wenzel2022analysis,wirtz2015surrogate}.

\subsection{Kernel approximation}
\label{subsec:kernel_approx}

Let $\Omega \subset \mathbb{R}^d$, $d \geq 1, d \in \mathbb{N}$ be a non-empty set and let us introduce a symmetric kernel function $\kappa: \Omega \times \Omega \longrightarrow \mathbb{R}$. 
For a given set of $N$ scattered data $X_N = \{ \boldsymbol{x}_1, \ldots, \boldsymbol{x}_N \} \subseteq \Omega$, we define the associated kernel matrix $\mK_{N}$ whose entries are given by $(\mK_{N})_{ij}= \kappa(\boldsymbol{x}_i,\boldsymbol{x}_j)$, $i,j=1,\ldots,N$. 
We remark that if the kernel matrix is positive definite for any set of pairwise distinct scattered data, then the kernel is said to be strictly positive definite. 

To each strictly positive definite kernel we can associate a unique Reproducing Kernel Hilbert Space $N_{\kappa} (\Omega)$ (RKHS) equipped with an inner product $\langle \cdot, \cdot \rangle$. 
The RKHS is also known as \textit{native space}, and it contains functions $f:\Omega \longrightarrow \mathbb{R}$ for which $\kappa$ acts as a reproducing kernel, i.e.: 
\begin{itemize}
 \item $\kappa(\cdot, \boldsymbol{x}) \in N_{\kappa}(\Omega)$, $\forall \boldsymbol{x} \in \Omega$,
 \item $f(\boldsymbol{x})=\langle f, \kappa(\cdot,\boldsymbol{x}) \rangle$, $\forall \boldsymbol{x} \in \Omega$, $\forall f \in N_{\kappa}(\Omega)$.
\end{itemize}

Given any set of pairwise distinct interpolation points $X_N \subseteq \Omega$ and an associated set of function values, 
samples of a function $f \in N_{\kappa}(\Omega)$, $F_N = \{f(\boldsymbol{x}_1), \ldots, f(\boldsymbol{x}_N)\}=\{f_1,\ldots,f_N\} \subseteq \mathbb{R}$, 
the well-known kernel representer theorem \cite{kimeldorf1970correspondence,wahba1990spline} 
states that there exists a unique minimum-norm interpolant $s_{X_N} \in N_{\kappa}(\Omega)$ of the form 
\begin{equation} \label{eq:interpolant}
 s_{X_N}(\cdot) = \sum_{i=1}^N \alpha_i \kappa(\cdot,\boldsymbol{x}_i). 
\end{equation}
The coefficients of the kernel-based interpolant are determined by imposing the interpolation conditions $s_{X_N}(\boldsymbol{x}_i)=f_i$ for all $i = 1, \dots, N$, 
thus by solving the linear system 
\begin{equation} \label{eq:system}
 \mK_{N} \boldsymbol{\alpha} = \boldsymbol{f}, 
\end{equation}
where $\boldsymbol{\alpha}=(\alpha_1,\ldots,\alpha_N)^{\intercal}$, and $\boldsymbol{f}=(f_1,\ldots,f_N)^{\intercal}$. \\
A particular class of kernels are so called translational invariant kernels, for which there exists a function $\Phi: \R^d \longrightarrow \R$ such that the kernel can be written as
\begin{align*}
\kappa(\boldsymbol{x}, \boldsymbol{y}) = \Phi(\varepsilon \cdot (\boldsymbol{x}-\boldsymbol{y})),
\end{align*}
whereby we already included a so called \textit{shape} or \textit{length scale} parameter $\varepsilon > 0$. \\
The choice of the shape or length scale parameter $\varepsilon$ from Eq.\ \eqref{eq:rbf_kernels}, 
which affects the concentration of the basis functions around the respective center, is a critical issue. 
Because of the influence of this hyperparameter in the reconstruction process, 
many optimization and searching strategies have been studied for its fine tuning \cite{cavoretto2021search,fornberg2007runge}; 
a review of different techniques is proposed e.g. in \cite[\S 14]{fasshauer2015kernel}. \\
In the following, we will generalize the concept of optimizing the hyperparameter $\varepsilon$, and in doing so, with abuse of notation, we will formally omit the dependence of the kernel on $\varepsilon$.
An important subclass of translational kernels is given by Radial Basis Function (RBF) kernels, for which there exists a univariate radial basis function $\phi: \R_+ \longrightarrow \Omega$, 
which might depend on a positive and real scale parameter $\varepsilon$, such that: 
\begin{align}
\label{eq:rbf_kernels}
(\mK_{N})_{i,j}=\kappa(\boldsymbol{x}_i,\boldsymbol{x}_j)
= \Phi(\varepsilon \cdot (\boldsymbol{x}_i - \boldsymbol{x}_j)) = \phi(\varepsilon \cdot \Vert \boldsymbol{x}_i - \boldsymbol{x}_j \Vert_2)
\end{align}
whereby in general it is also possible to use different distance metrics than the Euclidean one. \\
The function $\Phi(\boldsymbol{x}) \equiv \phi(\Vert \boldsymbol{x} \Vert)$ from Eq.\ \eqref{eq:rbf_kernels} allows to characterize the native space $\ns$ in terms of Sobolev spaces:
Assume that the decay of the Fourier transform of $\Phi: \R^d \longrightarrow \R$ with $\Phi \in L^1(\R^d)$ can be characterized by a decay rate $\tau > d/2$ and constants $c_\Phi, C_\Phi >0$ as
\begin{align}
\label{eq:fourier_decay}
c_\Phi (1+\Vert \boldsymbol{\omega} \Vert_2^2)^{-\tau} \leq \hat{\Phi}(\boldsymbol{\omega}) \leq C_\Phi (1 + \Vert \boldsymbol{\omega} \Vert_2^2)^{-\tau}	\quad \forall \boldsymbol{\omega} \in \R^d.
\end{align}
If additionally the domain $\Omega$ has a Lipschitz boundary,
then the native space $\ns$ can be shown to be norm-equivalent to the Sobolev space $H^\tau(\Omega)$, 
i.e.\ $\ns \asymp H^\tau(\Omega)$ \cite[Corollary 10.48]{wendland2005scattered}.

The interpolant from Eq.\ \eqref{eq:interpolant} can be equivalently defined as the orthogonal projection $\Pi_{V(X_N)}$ of $f$ onto the linear subspace $V(X_N) = \textrm{span} \{\kappa(\cdot, \boldsymbol{x}_i), \boldsymbol{x}_i \in X_N \}$, i.e.,
\begin{equation*} 
  \Pi_{V(X_N)}(f) = \sum_{i=1}^N c_i \kappa(\cdot,\boldsymbol{x}_i). 
\end{equation*}
 
\noindent Classical pointwise error bounds for kernel-based interpolants are of the form %
\begin{equation}\label{eq:bound_pow}
  |f(\boldsymbol{x})-s_{X_N}(\boldsymbol{x})| \leq P_{X_N} \| f - s_{X_N} \|_{N_{\kappa}(\Omega)} \equiv P_{X_N} \| r_N \|_{N_{\kappa}(\Omega)}, \hskip 0.15cm \boldsymbol{x} \in \Omega, \hskip 0.15cm f \in N_{\kappa}(\Omega), 
\end{equation}
where $r_N$ denotes the residual, i.e. $r_N = f- s_{X_N}$, and $P_{X_N}$, known as \emph{power function}, is defined as 
\begin{align}
\label{eq:powerfunc}
  P_{X_N}(\boldsymbol{x}) = \| \kappa(\cdot, \boldsymbol{x})- \Pi_{V(X_N)}(\kappa(\cdot, \boldsymbol{x})) \|_{N_{\kappa}(\Omega)}.
\end{align}
Other error indicators are based on the so-called fill-distance, which is given by
\begin{align}
\label{eq:fill_distance}
 h_{X_N} = h_{\Omega, X_N} = \sup_{ \boldsymbol{x} \in \Omega} \left( \min_{ \boldsymbol{x}_k \in {X_N}} \left\| \boldsymbol{x} - \boldsymbol{x}_k \right\|_2 \right),
\end{align}
and indicates how well $\Omega$ is filled out by data points. 
Then the pointwise error also suffices the following relation
\begin{align}
	|f\left(\boldsymbol{x}\right)-s_{X_N}\left(\boldsymbol{x}\right) | \leq C h^{\tau - d/2}_{X_N} ||f||_{N_{\kappa}(\Omega)}, \quad \boldsymbol{x} \in \Omega,
\label{eq:bound_fill}
\end{align}
for $h_{X_N} \leq h_0$, where $\tau>d/2$ is the rate of the decay of $\hat{\Phi}$ from Eq.\ \eqref{eq:fourier_decay}, i.e.\ depending on the smoothness of the kernel.

As we point out in the next Subsection \ref{subsec:greedy_kernel}, 
such error estimates and in particular Eq.\ \eqref{eq:bound_pow} can be used to define selection criteria for greedy center selection. 

\subsection{Greedy kernel methods} 
\label{subsec:greedy_kernel}

Building the kernel model on all data points $X_N \subseteq \Omega$,
i.e. in the form given by \eqref{eq:interpolant}, might be detrimental in some cases. 
For instance, if the data set is too huge, computing the kernel matrix and solving the linear system is either too costly or even infeasible.
Furthermore, in the case of surrogate modeling \cite{santin2019kernel}, one would like to deal with cheap and quickly evaluable models, thus aiming at small expansion sizes $n \ll N$. 
Due to its small expansion size, the greedy interpolant $s_n$ can be understood as a sparse approximation of $s_{X_N}$. 
An established way for achieve this in the context of surrogate modeling is to select a meaningful subset $X_n \subset X_N$ of the training data $X_N$ via greedy kernel methods \cite{wenzel2022analysis,wirtz2015surrogate}. 
These are iterative schemes that start with an empty set $X_0 = \{ \}$.
Then for $n \geq 1$, at the $n$-th step the set $X_n$ is defined as $X_{n} = X_{n-1} \cup \{\boldsymbol{x}_n\}$ and $\boldsymbol{x}_n$ is so that 
\begin{align*}
\boldsymbol{x}_{n} := \argmax_{\boldsymbol{x} \in X_N \setminus X_{n-1}} \eta^{(n)}(\boldsymbol{x}),
\end{align*}
using some error indicator $\eta^{(n)}: \Omega \longrightarrow \R$. \\
In the kernel literature, the following criteria are frequently used \cite{marchi2005near, muller2009komplexitat, schaback2000adaptive}, 
and they make use of either the residual (see Eq.\ \eqref{eq:bound_pow}) or the power function (see Eq.\ \eqref{eq:powerfunc}) or both:
\begin{enumerate}
\item $P$-greedy: \hspace{7mm} $\eta_P^{(n)}(\boldsymbol{x}) = P_{X_n}(\boldsymbol{x})$,
\item $f$-greedy: \hspace{7.9mm} $\eta_f^{(n)}(\boldsymbol{x}) = |r_n(\boldsymbol{x})|$,
\item $f/P$-greedy: \hspace{3mm} $\eta_{f/P}^{(n)}(\boldsymbol{x}) = |r_n(\boldsymbol{x})|/P_{X_n}(\boldsymbol{x})$.
\end{enumerate}
The convergence rates for the $P$-greedy algorithm were analyzed in \cite{santin2016convergence, wenzel2021novel, wenzel2022stability}.
Based on these works, the $P$-greedy, $f$-greedy and $f/P$-greedy algorithms were recently unified within the scale of so called \textit{$\beta$-greedy algorithms} and also analyzed in terms of their convergence rates \cite{wenzel2022analysis}.
Especially target data dependent algorithms like the $f$-greedy provide a faster rate of convergence, thus usually yielding more accurate (or cheaper) models. 
These faster convergence rates of the $f$-greedy models motivate their use later on in the numerical experiments in Section \ref{sec:num_experiments}. \\
For practical implementation, the algorithms stop as soon as a predefined maximal expansion size is obtained or a predefined accuracy threshold $\eta^{(n)} \leq \tau$ or some stability measure is reached.
An efficient implementation of these greedy kernel algorithms with the selection criteria from above was provided in a matrix free way under the notion VKOGA \cite{santin2019kernel,wirtz2013vectorial}.
For our approach, which will be introduced in the next Section \ref{sec:two-layered}, we build on top of this VKOGA implementation. 

\section{The 2L-VKOGA}
\label{sec:two-layered}

As already mentioned in the Introduction, in order to obtain an efficient and effective kernel model, one can
\begin{enumerate}
\item select suitable center points $\{ \boldsymbol{x}_i \}_{i=1}^n$, for example via greedy kernel methods as elaborated in Subsection \ref{subsec:greedy_kernel},
\item use a suitable kernel $\kappa$ with tuned hyperparameters, as presented in Subsection \ref{subsec:kernel_approx}.
\end{enumerate}
In regards of the first, the best known greedy selection strategies in terms of convergence rates are target-data dependent algorithms, especially the $f$-greedy algorithm \cite{wenzel2022analysis}. 

In the following we do not want to further investigate the greedy selection strategies, but instead focus on optimal kernel design for use in conjunction with a subsequent greedy center selection as implemented in VKOGA.

Precisely, we advance in two ways. First, in Subsection \ref{subsec:hyperparam_optim} we generalize the hyperparameters of Eq.\ \eqref{eq:rbf_kernels} to any arbitrary linear mapping, 
and subsequently show how this setup can be seen as a two-layered kernel machine.
Second, in Subsection \ref{subsec:loss_func_for_optim} we introduce a machine learning inspired strategy for the optimization of the two-layered kernel machines. 
In particular this optimization approach is way more time efficient than cross validating several kernel shape parameters, see also Subsection \ref{subsubsec:computational_complexity}.

Additionally, an analysis of the first layer and a convergence analysis for the greedy selection is discussed in Subsection \ref{subsec:some_analysis}. 
Furthermore, in Subsection \ref{subsec:implementation_2L} we comment on implementation details as well as on the complexity of the introduced 2L-VKOGA.

\subsection{Hyperparameter optimized kernels as two-layered kernel machines} \label{subsec:hyperparam_optim}

As elaborated in Subsection \ref{subsec:kernel_approx}, usually only one single length scale parameter $\varepsilon$ is used within RBF kernels. 
However especially in dimensions $d \gg 1$, different directions within the data might be unequally relevant. 
Furthermore, as those directions do not necessarily need to be aligned with the Euclidean ones, 
it is advisable to also incorporate possible rotations and transformations of the input space into the kernel. 
This can be done by using a $b \times d$ matrix $\mA_{\bs{\theta}}$
\begin{equation*}
  \mA_{\bs{\theta}}=
  \begin{pmatrix}
    \theta_1^1 & \dots & \theta_1^d\\
    \vdots & \ddots & \vdots \\
    \theta_b^1 & \dots & \theta_b^d
  \end{pmatrix},
\end{equation*}
and then considering the kernel 
\begin{equation} 
\label{eq:linear_two_layered_kernel}
  \kappa_{\bs{\theta}}(\boldsymbol{x}, \boldsymbol{y})=\kappa(\mA_{\bs{\theta}} \boldsymbol{x}, \mA_{\bs{\theta}} \boldsymbol{y}).
\end{equation}

For the special choice of $\mA_{\bs{\theta}} = \varepsilon \cdot \mI_d$, where $\mI_d$ is the $d\times d$ identity matrix, we obtain the classical RBF setting, 
while if $\mA_{\bs{\theta}} = {\rm diag}(\varepsilon_1, \ldots, \varepsilon_d)$ we recover the so-called anisotropic kernels. 
As the $b \cdot d$ hyperparameters within the matrix $\mA_{\bs{\theta}}$ can be optimized, 
in the following we generalize the concept of anisotropic kernels by learning the optimal kernel design, i.e., the matrix $\mA_{\boldsymbol{\theta}}$. 
In the following, we will mostly focus on the $d \times d$ case, i.e.\ $b = d$. \\

Now we want to point out that this hyperparameter optimized kernel can be naturally understood as a two-layered kernel according to the deep kernel representer theorem. 
Precisely, according to \cite[Eq.\ (10)]{bohn2019representer}, with slightly modified notation, a deep $L$-layered kernel looks like
\begin{align*}
\mathcal{K}^L(\boldsymbol{x}, \boldsymbol{y}) = K_L(f_{L-1} \circ .. \circ f_1(\boldsymbol{x}), f_{L-1} \circ .. \circ f_1(\boldsymbol{y})),
\end{align*}
with intermediate mappings
\begin{align}
\label{eq:layer_mapping}
f_i(\cdot) = \sum_{j=1}^N \alpha_j^{(i)} K_i(\cdot, f_{i-1} \circ .. \circ f_1(\boldsymbol{x}_j)).
\end{align}
For the special case $L=2$ (thus two-layered), and using a RBF kernel $\kappa$ as outer kernel $K_L$, we obtain
\begin{align}
\label{eq:two_layered_kernel}
\mathcal{K}^2(\boldsymbol{x}, \boldsymbol{y}) = \kappa(f_1(\boldsymbol{x}), f_1(\boldsymbol{y})).
\end{align}
Now we use a linear kernel for the first layer mapping $f_1$ (Eq.\ \eqref{eq:layer_mapping}), 
more precisely we choose a $d \times d$ matrix valued linear kernel as done in \cite[Section 3.2]{wenzel2021universality}: 
\begin{align*}
k_1(\boldsymbol{x}, \boldsymbol{y}) := k_\text{lin}(\boldsymbol{x}, \boldsymbol{y}) \equiv \langle \boldsymbol{x}, \boldsymbol{y} \rangle_{\R^d} \cdot \mI_d.
\end{align*}
Like that we obtain
\begin{align}
\label{eq:kernel_model}
f_1(\cdot) = \sum_{j=1}^N k_\text{lin}(\cdot, \boldsymbol{x}_j) \alpha_j^{(1)}.
\end{align}
Here we can leverage the following Theorem \ref{th:linear_layer} from \cite[Proposition 2]{wenzel2021universality}, which allows us to describe all possible mappings for kernel mappings \eqref{eq:kernel_model}.

\begin{theorem} \label{th:linear_layer}
A linear mapping, i.e.\ a mapping $\R^{d} \longrightarrow \R^d, \boldsymbol{x} \mapsto \mA \boldsymbol{x}$ with $\mA \in \R^{d \times d}$ can be realized as a kernel mapping 
\begin{align*}
s: \R^d \longrightarrow \R^d, 
\boldsymbol{x} \mapsto \sum_{i=1}^N \alpha_i k(\boldsymbol{x}, \boldsymbol{z}_i), ~~ \alpha_i \in \R^d
\end{align*} 
with given centers $\{ \boldsymbol{z}_i \}_{i=1}^N \subset \R^d$ by using a matrix valued linear kernel $k_\text{lin}(\boldsymbol{x}, \boldsymbol{y}) \cdot \mI_b = \langle \boldsymbol{x}, \boldsymbol{y} \rangle_{\R^d} \cdot \mI_d$, 
iff the span of the center points $\boldsymbol{z}_i, i=1, ..., N$ is a superset of the row space of the matrix $\mA$.
\end{theorem}
We remark that if the center matrix $[\boldsymbol{z}_1, ..., \boldsymbol{z}_N] \in \R^{d \times N}$ has rank $d$,
then the span of the center points is always a superset for the row space of any matrix $\mA \in \R^{d \times d}$. 
Then, we can formalize the following corollary.
\begin{corollary}
Consider $\{\boldsymbol{x}_1, ..., \boldsymbol{x}_N\} \subset \R^d$ such that the data matrix $[\boldsymbol{x}_1, ..., \boldsymbol{x}_N] \in \R^{d \times N}$ has rank $d$.
Then the kernel $\kappa_{\bs{\theta}}$ from Eq.\ \eqref{eq:linear_two_layered_kernel} is an instance of a two-layered kernel according to the deep kernel representer \cite[Theorem 1]{bohn2019representer}.
\end{corollary}
\begin{proof}
    Using Theorem \ref{th:linear_layer} applied to Eq.\ \eqref{eq:kernel_model}, we obtain 
\begin{align*}
\sum_{j=1}^N k_\text{lin}(\boldsymbol{x}, \boldsymbol{x}_j) \alpha_j^{(1)} = \mA_{\bs{\theta}} \boldsymbol{x}.
\end{align*}
Thus, the general two-layered kernel from Eq.\ \eqref{eq:two_layered_kernel} specializes to
\begin{align*} 
\mathcal{K}^2(\boldsymbol{x}, \boldsymbol{y}) = \kappa(f_1(\boldsymbol{x}), f_1(\boldsymbol{y})) = \kappa(\mA_{\bs{\theta}}\boldsymbol{x}, \mA_{\bs{\theta}}\boldsymbol{y}),
\end{align*}
which is exactly the hyperparameter tunable kernel $\kappa_{\bs{\theta}}$ from Eq.\ \eqref{eq:linear_two_layered_kernel}. 
\end{proof}

The evaluation of the proposed model is visualized in Figure \ref{fig:2L_machine}, making use of common neural network layout structures: our framework includes two hidden layers, and the first one is used to map the input evaluation point $\bs{x}$ via the learned matrix $\mA_{\boldsymbol{\theta}}$.

\begin{figure}[h]
\centering
\includegraphics[width=0.7\linewidth]{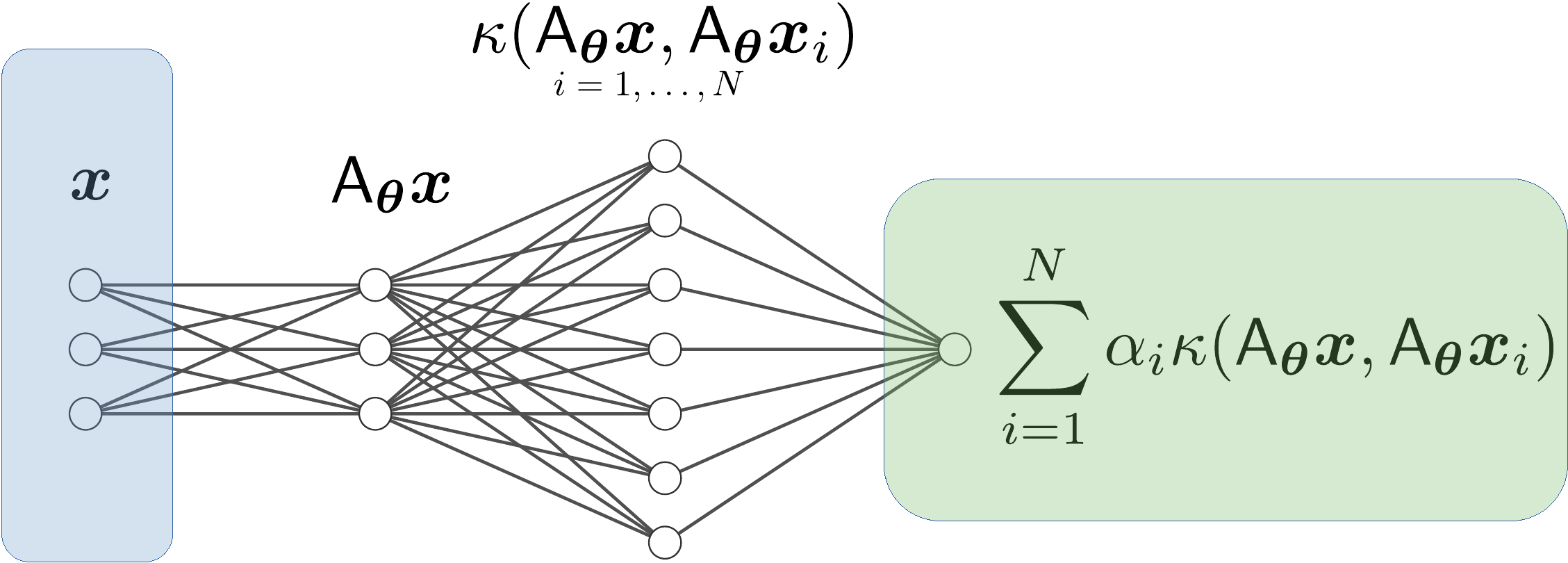}
\caption{Visualisation of the evaluation at $\bs{x}$ of the presented kernel machine. 
The input layer is highlighted in blue, while the output layer is framed in green. 
In this example, $b=d=3$ and $N=7$.}
\label{fig:2L_machine}
\end{figure}

Again we want to emphasize that this two-layered kernel is indeed a generalization of standard shape parameter tuned kernels. 
If we choose $\mA_{\bs{\theta}}$ as a scaled identity matrix, we directly re-obtain the well known shape parameter tuned kernel model because Eq.\ \eqref{eq:linear_two_layered_kernel} boils down to
\begin{align*}
\sum_{j=1}^N \alpha_j^{(2)} \kappa ( \mA_{\bs{\theta}} \boldsymbol{x}, \mA_{\bs{\theta}} \boldsymbol{y}) 
= \sum_{j=1}^N \alpha_j^{(2)} \kappa ( \varepsilon \mI_d \boldsymbol{x}, \varepsilon \mI_d \boldsymbol{y})
= \sum_{j=1}^N \alpha_j^{(2)} \kappa ( \varepsilon \boldsymbol{x}, \varepsilon \boldsymbol{y}).
\end{align*}
This especially also means that the flat limits of kernels (i.e.\ $\varepsilon \rightarrow 0$, see e.g.\ \cite[Chapter 14]{fasshauer2015kernel}) can be realized with our frameworks of two-layered kernels.

\subsection{Theoretical analysis} \label{subsec:some_analysis}

In this Subsection we first of all want to analyze and interpret the meaning of the first kernel layer, allowing to incorporate linear transformations of the input.
Subsequently we provide some preliminary convergence rate analysis for the subsequent greedy selection.

\subsubsection{Analysis of the first layer}
\label{subsubsec:analysis_first_layer}

In order to understand the impact of the first kernel layer on the performance of the overall kernel machine, 
we analyze the matrix $\mA_{\bs{\theta}}$ via its singular value decomposition given by
\begin{align*}
\mA_{\bs{\theta}} = \mathsf{U} \mathsf{\Sigma} \mathsf{V}^\intercal
\end{align*}
with orthogonal matrices $\mathsf{U}, \mathsf{V} \in \R^{d \times d}$ respectively consisting of the left and right singular vectors 
and the diagonal matrix $\mathsf{\Sigma} \in \R^{d \times d}$ with the non-negative singular values on the diagonal.
Denoting the columns of $\mathsf{U}, \mathsf{V}$ by respectively $\boldsymbol{u}_i$ and $\boldsymbol{v}_i, i=1, \dots, d$ we have
\begin{align}
\label{eq:svd_of_A}
\mA_{\bs{\theta}} \mathsf{V} = \mathsf{U} \mathsf{\Sigma} \quad \Leftrightarrow \quad \mA_{\bs{\theta}} \boldsymbol{v}_i = \sigma_i \boldsymbol{u}_i \quad \forall i=1, \dots, d.
\end{align}
For a given input $\boldsymbol{x} \in \R^d$, which can be decomposed as $\boldsymbol{x} = \sum_{i=1}^d \langle \boldsymbol{x}, \boldsymbol{v}_i \rangle_{\R^d} \boldsymbol{v}_i$ we obtain
\begin{align}
\label{eq:svd_calculation}
\mA_{\bs{\theta}} \boldsymbol{x} &= \sum_{i=1}^d \langle \boldsymbol{x}, \boldsymbol{v}_i \rangle_{\R^d} \mA_{\bs{\theta}} \boldsymbol{v}_i = \sum_{i = 1}^d \langle \boldsymbol{x}, \boldsymbol{v}_i \rangle_{\R^d} \sigma_i \boldsymbol{v}_i \notag \\
\Rightarrow \mA_{\bs{\theta}}(\boldsymbol{x} - \tilde{\boldsymbol{x}}) &= \sum_{i=1}^d (\langle \boldsymbol{x}, \boldsymbol{v}_i \rangle_{\R^d} - \langle \tilde{\boldsymbol{x}}, \boldsymbol{v}_i \rangle_{\R^d}) \sigma_i \boldsymbol{v}_i \notag \\
\Rightarrow \Vert \mA_{\bs{\theta}}(\boldsymbol{x} - \tilde{\boldsymbol{x}}) \Vert_{\R^d}^2 &= \sum_{i = 1}^d (\langle \boldsymbol{x}, \boldsymbol{v}_i \rangle_{\R^d} - \langle \tilde{\boldsymbol{x}}, \boldsymbol{v}_i \rangle_{\R^d})^2 \sigma_i^2.
\end{align}
The last line, i.e.\ Eq.\ \eqref{eq:svd_calculation} is of importance, as the RBF kernel $\kappa$ from Eq.\ \eqref{eq:linear_two_layered_kernel} in the second layer of the kernel only requires distances as an input.
Therefore we can see that the matrix $\mathsf{U} \in \R^{d\times d}$ does not matter at all for the 2L-VKOGA model, because it is not seen due to the radiality of the RBF kernel.
In particular it would be possible to set $\mathsf{U} := \mathsf{V}$, such that the matrix $\mA_{\bs{\theta}}$ is even symmetric.
However enforcing symmetry of the matrix $\mA_{\bs{\theta}}$ during the optimization step indeed impedes the performance of the optimization, 
in particular the resulting matrix $\mA_{\bs{\theta}}$ is frequently not as good as if one uses a nonsymmetric optimization.

From Eq.\ \eqref{eq:svd_calculation} we see, that the distance is scaled along the directions provided by the right singular vectors of $\mA_{\bs{\theta}}$ by the corresponding singular value $\sigma_i \geq 0$.
In particular if a singular value $\sigma_i$ is very small or even zero, this means that data along the directions of the corresponding right singular vectors $\boldsymbol{v}_i$ is squeezed or even mapped to the same point.

\subsubsection{Convergence analysis}
\label{subsubsec:conv}

The convergence analysis for (greedy) kernel interpolation is usually done either in the number of (greedily) selected points \cite{santin2016convergence, wenzel2022analysis}, or in terms of the fill distance \cite[Ch.\ 11]{wendland2005scattered}, which was defined in Eq.\ \eqref{eq:fill_distance}.
Both of those approaches focus on the asymptotic rate of the decay of the error.

Given a two-layered kernel $\kappa_{\boldsymbol{\theta}}(\bs{x}, \bs{y}) = \kappa(\mA_{\bs{\theta}} \bs{x}, \mA_{\bs{\theta}} \bs{x})$, 
we can derive convergence rates by considering this kernel as a standard radial basis function kernel acting on the transformed domain $\mA_{\bs{\theta}} \Omega \subset \R^{\rank(\mA)}$.
We will distinguish the two cases $\rank(\mA_{\bs{\theta}}) = d$ and $\rank(\mA_{\bs{\theta}}) < d$.
For the first case, the following Theorem \ref{th:conv_rate_fullrank_A} shows that we obtain at least the same rate of convergence as if one uses the standard kernel $\kappa$ instead of the two-layered kernel $\kappa_{\boldsymbol{\theta}}(\bs{x}, \bs{y})$. 
For the second case, the subsequent Theorem \ref{th:conv_rate_reducedrank_A} shows that the convergence rate is indeed increased, depending on the number of singular values of the matrix $\mA_{\bs{\theta}}$ equal to zero:

\begin{theorem}
\label{th:conv_rate_fullrank_A}
Consider an RBF kernel $\kappa$ that satisfies Eq.\ \eqref{eq:fourier_decay} with $\tau > d/2$ on a bounded Lipschitz domain $\Omega \subset \R^d$.
Consider $f \in \ns$ and the kernel interpolant $s_{X_N}$ using the two-layered kernel $\kappa_{\boldsymbol{\theta}}$ with $\rank(\mA_{\bs{\theta}}) = d$. \\
Then the following pointwise error estimate holds:
\begin{align*}
|f(\bs{x}) - s_{X_N}(\bs{x})| \leq C h_{X_n}^{\tau - d/2}.
\end{align*}
For asymptotically equiddistributed points $h_{X_n} \asymp n^{-1/d}$  it then holds
\begin{align*}
|f(\bs{x}) - s_{X_N}(\bs{x})| \leq C n^{1/2 - \tau/d}.
\end{align*}
\end{theorem}

\begin{proof}
We consider the two-layered kernel $\kappa_{\boldsymbol{\theta}}$ as a standard RBF kernel $\kappa$ applied to the transformed data $\mA_{\bs{\theta}} X_N$ from the transformed domain $\mA_{\bs{\theta}} \Omega$.
As it holds $\rank(\mA_{\bs{\theta}}) = d$, we have $\dim(\mA_{\bs{\theta}} \Omega) = \dim(\Omega) = d$.
By standard Sobolev arguments (as $x \mapsto \mA_{\bs{\theta}} x$ is just a linear full rank transformation) we have $H^\tau(\Omega) \asymp H^\tau(\mA_{\bs{\theta}} \Omega)$, 
in particular $f \circ A_{\boldsymbol{\theta}}^{-1} \in H^\tau(\mA_{\bs{\theta}} \Omega)$.
Furthermore, due to $\rank(\mA_{\bs{\theta}}) = d$ all the singular values of $\mA_{\bs{\theta}}$ are positive, i.e.\ it holds 
$s_{\min}(\mA_{\bs{\theta}}) \Vert x \Vert_2 \leq \Vert \mA_{\bs{\theta}} x \Vert_2 \leq s_{\max}(\mA_{\bs{\theta}}) \Vert x \Vert_2$ or in short $\Vert \cdot \Vert_2 \asymp \Vert \mA_{\bs{\theta}} \cdot \Vert_2$.
Therewith we obtain for the respective fill distances
\begin{align*}
h_{\mA_{\bs{\theta}} \Omega, \mA_{\bs{\theta}} X_N} &\equiv \sup_{\tilde{\bs{x}} \in \mA_{\bs{\theta}} \Omega} \left( \min_{\tilde{\bs{x}}_k \in \mA_{\bs{\theta}} X_k} \Vert \tilde{\bs{x}} - \tilde{\bs{x}}_k \Vert_2 \right) \\
&= \sup_{\bs{x} \in \Omega} \left( \min_{\bs{x}_k \in X_k} \Vert \mA_{\bs{\theta}} \bs{x} - \mA_{\bs{\theta}} \bs{x}_k \Vert_2 \right) \\
&\asymp \sup_{\bs{x} \in \Omega} \left( \min_{\bs{x}_k \in X_k} \Vert \mA_{\bs{\theta}} \bs{x} - \mA_{\bs{\theta}} \bs{x}_k \Vert_2 \right) = h_{\Omega, X_n}.
\end{align*}
Hence we can make use of the error bound of Eq.\ \eqref{eq:bound_fill} to derive the final statement.
\end{proof}

From the previous proof it is obvious to see that we can obtain a faster rate of convergence, as soon as the fill distance decays faster.
An asymptotically faster decay of the fill distance is only possible, if the dimension of the underlying input domain is effectively reduced. 
This is the case iff $\rank(\mA_{\bs{\theta}}) < d$, because then $\dim(\mA_{\bs{\theta}} \Omega) < \dim(\Omega)$. 

In order to avoid technical discussion on radial basis functions and the corresponding Fourier transforms, 
we focus for the following theorem on the class of Matérn kernels, which are also used throughout Section \ref{sec:num_experiments} of numerical experiments.
In their general form they are given as \cite[Section 4.4]{fasshauer2007meshfree}
\begin{align}
\label{eq:matern_kernel}
\Phi(\bs{x}) = \frac{K_{\tau - d/2}(\Vert \bs{x} \Vert) \Vert \bs{x} \Vert^{\tau - d/2}}{2^{\tau - 1} \Gamma(\tau)}, \quad \tau > d/2,
\end{align}
whereby $K_\nu$ is the modified Bessel function of second order.
The corresponding Fourier transform is given as 
\begin{align*}
\hat{\Phi}(\bs{\omega} ) = (1 + \Vert \bs{\omega} \Vert^2)^{-\tau},
\end{align*}
i.e.\ as in Eq.\ \eqref{eq:fourier_decay} but with $c_\Phi = C_\Phi = 1$.

\begin{theorem}
\label{th:conv_rate_reducedrank_A}
Consider a Matérn kernel $\kappa$ of Eq.\ \eqref{eq:matern_kernel} with $\tau > d/2$ on a bounded Lipschitz domain $\Omega \subset \R^d$.
Consider $f \in \ns$ and the kernel interpolant $s_{X_N}$ using the two-layered kernel $\kappa_{\boldsymbol{\theta}}$ such that $\rank(\mA_{\bs{\theta}}) =: d_\text{eff} < d$. \\
Assume that $f \in \ns$ is invariant along the subspace $\mathrm{Null}(\mA_{\bs{\theta}}) \subset \R^d$, i.e.\ $f(\bs{x}) = f(\bs{x}')$ for any $\bs{x}, \bs{x}' \in \Omega$ with $\bs{x} - \bs{x}' \in \mathrm{Null}(A_{\boldsymbol{\theta}})$. For points $X_n \subset \Omega$ such that $h_{\mA_{\boldsymbol{\theta}} \Omega, \mA_{\boldsymbol{\theta}} X_n} \asymp n^{-1/d_\text{eff}}$  it then holds
\begin{align*}
|f(\bs{x}) - s_{X_n}(\bs{x})| \leq C n^{\frac{d}{2d_\text{eff}} - \frac{\tau}{d_\text{eff}}}.
\end{align*}
\end{theorem}
Note that it holds $\frac{d}{2d_\text{eff}} - \frac{\tau}{d_\text{eff}} < \frac{1}{2} - \frac{\tau}{d} < 0$ due to $\tau > d/2$, i.e.\ the decay rate is faster than in Theorem \ref{th:conv_rate_fullrank_A}:
\begin{align*}
\frac{d}{2d_\text{eff}} - \frac{\tau}{d_\text{eff}} - \left( \frac{1}{2} - \frac{\tau}{d} \right) = \frac{d(d-d_\text{eff}) - 2\tau(d - d_\text{eff})}{2dd_\text{eff}} = \frac{(d - 2\tau)(d - d_\text{eff})}{2dd_\text{eff}} < 0.
\end{align*}

\begin{proof}
Define $\cal{N} := \mathrm{Null}(\mA_{\bs{\theta}})$ and consider the orthogonal projector $\Pi_{{\cal N}^\perp}: \R^d \longrightarrow {\cal N}^\perp$.
The mapping $\mA_{\bs{\theta}}^{{\cal N}^\perp}: {\cal N}^\perp \longrightarrow \mathrm{R}(\mA_{\bs{\theta}}), \boldsymbol{x} \mapsto \mA_{\bs{\theta}} \boldsymbol{x}$ is now full rank and thus invertible.
Consider $\boldsymbol{x} \in \Omega$ and decompose $\boldsymbol{x} = \boldsymbol{x}_\parallel + \boldsymbol{x}_\perp$ with $\boldsymbol{x}_\parallel \in {\cal N}, \boldsymbol{x}_{\perp} \in {\cal N}^\perp$.
Using the invariance assumption on $f$ along ${\cal N}$, we have
\begin{align}
\label{eq:calculation}
|(f- s_{X_n})(\boldsymbol{x})| &= |f(\boldsymbol{x}_\parallel + \boldsymbol{x}_\perp) - \sum_{j=1}^n \alpha_j^{(n)} \kappa(\mA_{\bs{\theta}} (\boldsymbol{x}_\parallel + \boldsymbol{x}_\perp), \mA_{\bs{\theta}} (\boldsymbol{x}_{j, \parallel} + \boldsymbol{x}_{j, \perp}))| \notag \\
&= |(f(\boldsymbol{x}_\perp) - \sum_{j=1}^n \alpha_j^{(n)} \kappa(\mA_{\bs{\theta}} \boldsymbol{x}_\perp, \mA_{\bs{\theta}} \boldsymbol{x}_{j, \perp})| \notag \\
&= |(f \circ (\mA_{\bs{\theta}}^{{\cal N}^\perp})^{-1})(\mA_{\bs{\theta}}^{{\cal N}^\perp} \boldsymbol{x}_\perp) - \sum_{j=1}^n \alpha_j^{(n)} \kappa(\mA_{\bs{\theta}} \boldsymbol{x}_\perp, \mA_{\bs{\theta}} \boldsymbol{x}_{j, \perp})|.
\end{align}

As $\dim(\mA_{\bs{\theta}} \Omega) \equiv d_\text{eff}$, the native space $N_{\kappa}(\mA_{\bs{\theta}} \Omega)$ is now norm-equivalent to the Sobolev space $H^{\tau'}(\mA_{\bs{\theta}} \Omega)$ of smaller smoothness $\tau' = \tau - \frac{d - d_\text{eff}}{2} < \tau$ see Eq.\ \eqref{eq:matern_kernel}.
Therefore we obtain $f \circ (\mA_{\bs{\theta}}^{{\cal N}^\perp})^{-1} \in H^{\tau'}(\mA_{\bs{\theta}} \Omega) \asymp N_{\kappa}(\mA_{\bs{\theta}} \Omega)$. \\
Furthermore $\sum_{j=1}^n \alpha_j^{(n)} \kappa(\cdot, \mA_{\bs{\theta}} \boldsymbol{x}_{j, \perp}) \in H^{\tau'}(\mA_{\bs{\theta}} \Omega)$ and due to the kernel interpolation condition it holds
\begin{align*}
\sum_{j=1}^n \alpha_j^{(n)} \kappa(\boldsymbol{y}_i, \mA_{\bs{\theta}} \boldsymbol{x}_{j, \perp}) = (f \circ (\mA_{\bs{\theta}}^{{\cal N}^\perp})^{-1})(\boldsymbol{y}_i)
\end{align*}
for all $\boldsymbol{y}_i \in \{\mA_{\bs{\theta}} \boldsymbol{x}_i ~ | ~ \boldsymbol{x}_i \in X_n\}$.
Therefore we can leverage Eq.\ \eqref{eq:bound_fill} to bound the error as
\begin{align*}
|(f \circ (\mA_{\bs{\theta}}^{{\cal N}^\perp})^{-1})(\boldsymbol{y}) - \sum_{j=1}^n \alpha_j^{(n)} \kappa(\boldsymbol{y}, \mA_{\bs{\theta}} \boldsymbol{x}_{j, \perp})| &< C h_{\mA_{\bs{\theta}} \Omega, \mA_{\bs{\theta}} X_n}^{\tau' - d_\text{eff} / 2}, \qquad \boldsymbol{y} \in \{\mA_{\bs{\theta}} \boldsymbol{x} ~ | ~ \boldsymbol{x} \in \Omega \} \\
\Leftrightarrow \quad |(f \circ (\mA_{\bs{\theta}}^{{\cal N}^\perp})^{-1})(\mA_{\bs{\theta}}^{{\cal N}^\perp} \boldsymbol{x}_\perp) - \sum_{j=1}^n \alpha_j^{(n)} \kappa(\mA_{\bs{\theta}} \boldsymbol{x}_\perp, \mA_{\bs{\theta}} \boldsymbol{x}_{j, \perp})| &< C h_{\mA_{\bs{\theta}} \Omega, \mA_{\bs{\theta}} X_n}^{\tau' - d_\text{eff} / 2}, \qquad \boldsymbol{x} \in \Omega,
\end{align*}
such than in conjunction with Eq.\ \eqref{eq:calculation} we obtain
\begin{align*}
|(f - s_{X_n})(\boldsymbol{x})| \leq C h_{\mA_{\bs{\theta}} \Omega, \mA_{\bs{\theta}} X_n}^{\tau' - d_\text{eff} / 2}, \qquad \boldsymbol{x} \in \Omega.
\end{align*}
Using finally $h_{\mA_{\boldsymbol{\theta}} \Omega, \mA_{\boldsymbol{\theta}} X_n} \asymp n^{-1/d_\text{eff}}$ which is possible due to $\dim(\mA_{\bs{\theta}} \Omega) \equiv d_\text{eff}$ we obtain the desired statement:
\begin{align*}
-\frac{1}{d_\text{eff}} \cdot \left( \tau' - d_\text{eff} / 2 \right) = -\frac{\tau - \frac{d - d_\text{eff}}{2}}{d_\text{eff}} + \frac{1}{2} = -\frac{\tau}{d_\text{eff}} + \frac{d}{2d_\text{eff}}.
\end{align*}
\end{proof}

Theorem \ref{th:conv_rate_reducedrank_A} can be leveraged in the following way. 
Given $f \in \ns$, which is invariant in some directions, then the two-layered kernel $\kappa_{\boldsymbol{\theta}}$ can be chosen such that the null space $\mathrm{Null}(\mA_{\bs{\theta}})$ of the matrix $\mA_{\bs{\theta}}$ coincides with this invariant subspace.
Then, as $\rank(\mA_{\bs{\theta}}) \equiv d_\text{eff} < d$, the fill distance $h_{\mA_{\boldsymbol{\theta}} \Omega, \mA_{\boldsymbol{\theta}} X_n}$ can decay as $h_{\mA_{\boldsymbol{\theta}} \Omega, \mA_{\boldsymbol{\theta}} X_n} \asymp n^{-1/d_\text{eff}}$ for suitable chosen points $X_n$ (namely such that $\mA_{\boldsymbol{\theta}} X_n$ is asymptotically equiditsant within  $\mA_{\boldsymbol{\theta}} \Omega$).
As remarked above, this provides a faster convergence rate, thus providing a benefit of using the two-layered kernel. 
Section \ref{sec:num_experiments} shows that our used optimization approach is indeed capable of (approximately) finding those invariant directions, i.e.\ inactive subspaces of the considered function $f \in \ns$.

In order to avoid too many technical details, we do not show convergence results for the $f$-greedy algorithm using the two-layered kernel $\kappa_{\boldsymbol{\theta}}(\bs{x}, \bs{y})$, albeit we use it later on in Section \ref{sec:num_experiments}.
However we remark that roughly speaking the convergence analysis of the $f$-greedy algorithm from \cite{wenzel2022analysis} is based on convergence rates for the $P$-greedy algorithm, 
i.e.\ those rates from Theorem \ref{th:conv_rate_fullrank_A} and \ref{th:conv_rate_reducedrank_A}.
The analysis from \cite{wenzel2022analysis} shows an additional convergence rate of $\log(n) n^{-1/2}$ for the $f$-greedy algorithm by making use of its target data dependent selection criterion.
Therefore we expect that the $f$-greedy algorithm provides improved convergence rates as in Theorem \ref{th:conv_rate_fullrank_A} respectively \ref{th:conv_rate_reducedrank_A}, namely by the additional factor of $\log(n) n^{-1/2}$.

As it can also be seen in the numerical experiments in Section \ref{sec:num_experiments}, we want to point out that our 2L-VKOGA approach does not only necessarily provide benefits in terms of the asymptotic convergence rate, but also in terms of the preasymptotic factor, especially if singular values of the final optimized matrix $\mA_{\bs{\theta}}$ are not exactly zero, but close to.
A more detailed analysis of this preasymptotic regime is left for future research.

\subsection{Loss function for optimization} \label{subsec:loss_func_for_optim}

Learning the kernel $\kappa_{\bs{\theta}}$ is equivalent to \textit{optimizing} the matrix $\mA_{\boldsymbol{\theta}}$, 
thus $\kappa_{\bs{\theta}}$ is a parametric model that depends on $b \times d$ parameters belonging to some space $\Theta\subseteq \mathbb{R}^{b\times d}$. 
We aim at optimizing the kernel by minimizing a loss function that depends on the input and target data:
\begin{align} \label{eq:loss1}
  \min_{\mA_{\bs{\theta}}\in\Theta}\ell(X_{N}, F_N, \mA_{\bs{\theta}}).
\end{align}
For the actual optimization we will make use of well known optimization strategies from the machine learning community \cite{goodfellow2016deep},
in particular gradient based optimization and the use of mini-batches; more details on the optimization and implementation will be given in the next Subsection \ref{subsec:implementation_2L}. 
In the following, we will focus on the structure of the loss function and on its use in conjunction with the optimization via mini-batches.

The loss consists in the cross validation (CV) error, which is an established criterion in the scientific community to assess the effectiveness of a model \cite{golub1979generalized}. 
However, instead of directly evaluating the CV error on $X_N$ with respect to $F_N$ using the kernel $\kappa_{\bs{\theta}}$, we evaluate the $k$-fold cross validation error on so-called \textit{mini-batches}. 
This is necessary and even beneficial, as the evaluation of $k$-fold cross validation scores is time consuming on large datasets and full batch learning is known to be detrimental for generalization in machine learning tasks \cite{goodfellow2016deep}.
Mini-batches are randomly drawn subsets $(X_\text{batch}, F_\text{batch})$ of input data with corresponding target values of size $n_\text{batch} \ll N$ from the large data set $X_N$ with corresponding target data $F_N$. 
A commonly used value, which is later on employed in the numerical experiments, is $n_\text{batch} = 64$. 

As mentioned in the introduction, this strategy is related to the optimization approaches in \cite{hamzi2021learning, owhadi2019kernel}, where also mini-batches are used.
However there the optimization criterion was based on the premise that a kernel must be good, \lq\lq if the number of points used to interpolate the data can be halved without significant loss in accuracy\rq\rq, see \cite[Eq.\ (6)]{hamzi2021learning}. 
In contrast we use for every mini-batch the $k$-fold cross validation error by taking advantage of an efficient implementation proposed by Rippa for the case $k=n_\text{batch}$ (Leave-One-Out CV (LOOCV)) \cite{rippa1999algorithm} and then extended to the general $1< k < n_\text{batch}$ framework in \cite{marchetti2021extension}. 
In the following, we briefly outline this strategy to which we refer to as Extended Rippa's Algorithm (ERA).

Let $k \in \N$, $1 < k \leq n_\text{batch}$, be the number of folds used for the $k$-fold validation scheme and suppose, for simplicity, that $p = n_\text{batch}/k \in \mathbb{N}$. 
Then, for each fold, let us split the minibatch $X_\text{batch}$ into a training set $T_{n_\text{batch}-p}$ of cardinality $n_\text{batch}-p$ and validation set $V_{p}$, 
so that $X_{n_\text{batch}}=T_{n_\text{batch}-p} \cup V_{p}$ and $T_{n_\text{batch}-p} \cap V_{p} = \emptyset$. 
Let us denote by $\boldsymbol{r} = (r_1,\ldots, r_p)^{\intercal}$, $r_i \in \{1,\ldots, n_\text{batch}\}$ the vector of distinct validation indices for a given fold, 
i.e. $V_p=\{\boldsymbol{x}_{r_i}, i=1,\ldots,p\}$. 
Then, we are interested in computing the residual vector $\boldsymbol{e}_{\boldsymbol{r}}=(e_1,\ldots,e_{p})$ whose components are $e_i=|s_{T_{n_\text{batch}-p}}(\boldsymbol{x}_{r_i})-f(\boldsymbol{x}_{r_i})|$, 
$\boldsymbol{x}_{r_i} \in V_p$, being $s_{T_{n_\text{batch}-p}}$ the interpolant constructed upon the training set $T_{n_\text{batch}-p}$. 
A standard application of the $k$-fold CV scheme would require the inversion of $k$ different $(n_\text{batch}-p) \times (n_\text{batch}-p)$ linear systems of the form \eqref{eq:system}, 
leading to a complexity cost of about ${\cal O}(n_\text{batch}^3 k)$. Fortunately, letting
$$
  \mK^{\boldsymbol{\theta}}_{n_\text{batch}}=
  \begin{pmatrix}
  \kappa_{\boldsymbol{\theta}}(\bs{x}_1,\bs{x}_1) & \dots & \kappa_{\boldsymbol{\theta}}(\bs{x}_1,\bs{x}_{n_\text{batch}})\\
  \vdots & \ddots & \vdots \\
  \kappa_{\boldsymbol{\theta}}(\bs{x}_{n_\text{batch}},\bs{x}_1) & \dots & \kappa_{\boldsymbol{\theta}}(\bs{x}_{n_\text{batch}},\bs{x}_{n_\text{batch}})
  \end{pmatrix},
$$
in \cite{marchetti2021extension}, 
the author proved that $\boldsymbol{e}_{\boldsymbol{r}}$ is the unique solution of the linear system
\begin{equation}\label{eq:rippa_loss}
(\mK^{\boldsymbol{\theta}}_{n_\text{batch}})^{-1}_{V_{p}} \boldsymbol{e}_{\boldsymbol{r}} = \boldsymbol{c}_{\boldsymbol{r}},
\end{equation}
where $(\mK^{\boldsymbol{\theta}}_{n_\text{batch}})^{-1}_{V_{p}}$ is the submatrix of the $n_\text{batch}\times n_\text{batch}$ inverse kernel matrix $(\mK^{\boldsymbol{\theta}}_{n_\text{batch}})^{-1}$, 
built on $\kappa_{\bs{\theta}}$, which is defined by restricting to the validation indices, 
and $\boldsymbol{c}_{\boldsymbol{r}}$ are the components of the solution vector from Eq.\ \eqref{eq:system} for which $i \in \boldsymbol{r}$. 
Hence, such an implementation requires a total complexity cost which is about ${\cal O}(n_\text{batch}^3)+ {\cal O} (n_\text{batch}^3/k^2)$ for the construction of the \textit{complete} error vector $\boldsymbol{e}=\ell(X_{N},F_N,\kappa_{\bs{\theta}})$,
which includes the validation errors computed on each element of the batch.
Finally, we point out that any norm of the vector $\boldsymbol{e}$ can be used to have an a priori error estimate, because the Rippa's scheme is independent of the used norm in the end. 
In the following we will consider the squared two-norm, and we will apply Tikhonov regularization with a parameter $\lambda$ to Eq.\ \eqref{eq:rippa_loss}. \\
Thus finally we are optimizing the kernel $\kappa_{\bs{\theta}}$ by minimizing the following loss function, which is a refinement of Eq.\ \eqref{eq:loss1}:
\begin{align} \label{eq:loss2}
  \min_{\mA_{\bs{\theta}}\in\Theta} \ell_\lambda(X_{N}, F_N, \mA_{\bs{\theta}}, k).
\end{align}
The ERA may be further speeded up by allowing a stochastic approximation, which however affects the exactness of the scheme \cite{ling2022stochastic}. 
While ERA, as the Rippa's scheme, was originally designed for the tuning of the shape parameter, 
we employ it as a more general \textit{error indicator} as recently done in \cite{cavoretto2022adaptive, marchetti2022efficient}. 
In the next subsection we provide more details on the optimization and implementation. 

\subsection{Optimization, implementation and computational complexity}
\label{subsec:implementation_2L}

This section is devoted to the optimization and implementation of the two layered kernel and the investigation of its computational complexity.

\subsubsection{Optimization and implementation}

In Subsection \ref{subsec:loss_func_for_optim} we elaborated on the (family of) loss functions which we employ for the optimization of the kernel $\kappa_{\bs{\theta}}$.
Here we give some more details on the actual optimization procedure. 
Based on the computed loss values, we employ an iterative gradient based optimization using the adaptive \textit{Adam optimizer} \cite{kingma2014adam} (instead of plain stochastic gradient descent (SGD)) and \textit{early stopping}\cite{goodfellow2016deep} on the accumulated loss values during one epoch.
The implementation is done in Python leveraging the deep learning framework {\tt pytorch} \cite{paszke2019pytorch} and especially making use of the auto-differentiation for computing the gradients.
The code is implemented as an extension of the VKOGA software package, which is described in \cite{santin2019kernel}.

The algorithm used for the optimization of the two-layered kernel is depicted in Algorithm \ref{alg:kernel_optimization} for the case of plain stochastic gradient descent.
For the use of Adam optimizer, the weights are more sophistically updated and we waived to delve into details concerning this.
The crucial step of the optimization, namely the computation of the gradients of the loss with respect to the matrix $\mA_{\bs{\theta}}$, is conveniently handled by pytorch via automatic differentiation.
To improve the numerical stability of the algorithm, a Tikhonov regularization is added to the kernel matrix as elaborated in Subsection \ref{subsec:loss_func_for_optim}. 
The early stopping criterion in line 11 stops the optimization if the loss does not decay further, thus avoiding unnecessary further optimization steps. 

\SetKwComment{Comment}{/* }{ */}

\RestyleAlgo{ruled}
\begin{algorithm}[hbt!]
\caption{SGD optimization of the kernel $\kappa_{\bs{\theta}} \equiv \kappa(\mA_{\bs{\theta}} \boldsymbol{x}, \mA_{\bs{\theta}} \boldsymbol{y})$.}\label{alg:kernel_optimization}
\SetKwInOut{Input}{Input}
\SetKwInOut{Output}{Output}
\Input{Data $(X_N, F_N)$, base kernel $\kappa$, $k$-fold parameter $k$, learning rate $\mu$, regularization parameter $\lambda$}
\KwResult{Optimized matrix $\mA_{\bs{\theta}}$}
$\mA_{\bs{\theta}} \gets \textrm{diag}(1, ..., 1)$ \Comment*[r]{Initialization of $\mA_{\bs{\theta}}$}
~ \\
\For{$n_\text{epoch} = 1, ..., \max_\text{epoch}$}{
	$L_\mathrm{epoch} \gets 0$\;
	\textsc{shuffle} $(X_N, F_N)$\;

	\For{$n_\text{iter} = 1, ..., \max_\text{iter}$}{
	($X_\text{batch}, F_\text{batch}) \gets \textsc{get}\_\textsc{batch}((X_N, F_N))$\;
	$L = \ell_\lambda(X_\text{batch}, F_\text{batch}, \mA_{\bs{\theta}}, k)$	 \Comment*[r]{Using Eq.\ \eqref{eq:loss2}}
	
	$\mA_{\bs{\theta}} \gets \mA_{\bs{\theta}} - \mu \cdot \frac{\partial L}{\partial \mA_{\bs{\theta}}}$ \Comment*[r]{Gradient descent update}
	
	$L_\mathrm{epoch} \gets L_\mathrm{epoch} + L$
	}
	~ \\
	\textsc{early}\_\textsc{stopping}($L_\mathrm{epoch}$)
}
\end{algorithm}

\subsubsection{Computational complexity}
\label{subsubsec:computational_complexity}

In this section, we briefly analyze the complexity of our proposed 2L-VKOGA approach against a standard cross validation approach.
Furthermore we comment on the speed up of using Rippa's and extended Rippa's scheme as described in Subsection \ref{subsec:loss_func_for_optim} for the kernel optimization:

As our 2L-VKOGA approach is based on a gradient descent optimization using up to $\text{max}_\text{epoch}$ epochs with each $N / n_\text{batch}$ iterations using small $n_\mathrm{batch} \times n_\text{batch}$, it requires 
\begin{align*}
{\cal O} (n_\text{epoch}  \cdot N / n_\text{batch} \cdot n_\text{batch}^3 ) = {\cal O} (n_\text{epoch}  \cdot N \cdot n_\text{batch}^2),
\end{align*}
operations for the calculation of cross validation errors on small matrices. 
The subsequent run of VKOGA is typically of order ${\cal O} (n_\text{vkoga}^2 \cdot N)$, 
thus the overall complexity of the 2L-VKOGA is
\begin{align*}
{\cal O} (n_\text{epoch}  \cdot N \cdot n_\text{batch}^2) + {\cal O} (n_\text{vkoga}^2 \cdot N).
\end{align*}
For typical values of $n_\text{batch}, n_\text{epoch}, n_\text{vkoga}$ such as $n_\text{batch} = 64, n_\text{epoch} = 10, n_\text{vkoga} = 1000$, which were used for the numerical experiments in Section \ref{sec:num_experiments}, it holds $n_\text{epoch}  \cdot N \cdot n_\text{batch}^2 < n_\text{vkoga}^2 \cdot N$, 
i.e.\ there is only a small computational overhead for the kernel optimization before running VKOGA. 

This is in contrast to a straightforward shape parameter cross validation using VKOGA. Indeed, full cross validation for the matrix $\mA_{\boldsymbol{\theta}}$ is infeasible, because the number of cross validation runs scales exponentially in the number of parameters, 
i.e.\ ${n_\text{CV}}^{d^2}$ (curse of dimensionality).
Also, a full cross validation for a classical anisotropic kernel, i.e.\ just for the diagonal of $\mA_{\boldsymbol{\theta}}$, is infeasible, as this requires $n_\text{CV}^{d^2}$ runs.
Thus, a cross validation of $n_\text{CV}$ shape parameters takes the effort
\begin{align*}
{\cal O} (n_\text{CV} \cdot n_\text{vkoga}^2 \cdot N),
\end{align*}
which elucidates the additional factor $n_\text{CV}$.
Hence we conclude that our approach is favorable and cheaper than the classical CV implementation, in particular for large $N$. 

Especially for large space dimension $d$ only a cross validated shape parameter is likely inferior to a whole optimized matrix (in terms of the explored parameter space $\Theta \subset \R^{d \times d}$).

\section{Numerical experiments}
\label{sec:num_experiments}

In this section we provide three different kinds of numerical experiments. 
First, in Subsection \ref{subsec:num_ex_func_approx} we use our two-layered approach for the efficient approximation of given functions.
Second, in Subsection \ref{subsec:num_ex_ml_data} we show the applicability of two-layered kernels for sparse surrogate modeling on real world machine learning datasets.
Finally, in Subsection \ref{subsec:num_stability} we investigate the optimization of the first layer, i.e.\ the matrix $\mA_{\boldsymbol{\theta}}$, with help of the extended Rippa's formula.

\subsection{Function approximation on the unit cube} \label{subsec:num_ex_func_approx}

As a first class of example we highlight the benefits of the proposed 2L-VKOGA for function approximation.
For this, we picked functions where different behaviours can be seen. 
We consider the unit cube $\Omega = [0, 1]^d$ for $d=5, 6, 7$, discretized with $N=50000$ uniformly randomly selected points. The corresponding target data is given as the evaluation of the function
\begin{align*}
  f_d(\boldsymbol{x}) = \left\{
\begin{array}{ll}
{\rm e}^{-4 \left( \sum_{j=1}^5 \boldsymbol{x}_i - 0.5 \right)^2}, & \textrm{for $d=5$}, \\
{\rm e}^{-4 \sum_{j=1}^5 (\boldsymbol{x}_i - 0.5)^2} + 2 |\boldsymbol{x}_1 - 0.5|, & \textrm{for $d=6$}, \\
{\rm e}^{-\sum_{j=1}^7 (\boldsymbol{x}_i - 0.5)^2} + {\rm e}^{-9 \sum_{j=1}^2 (\boldsymbol{x}_i - 0.3)^2}, & \, \textrm{for $d=7$}. 
\end{array}
\right.
\end{align*}
While the first function $f_5$ clearly possesses an active subspace along the direction $(1, 1, 1, 1, 1)^\intercal \in \R^5$, this does not hold for $f_6$ and $f_7$. 
However for $f_6$ the  kink exists only in the $x_1$ direction, while for $f_7$ the second bump depends only on the first two variables.

For the approximation we use as a base kernel the Matérn kernel $\kappa(\boldsymbol{x}, \boldsymbol{y})=\exp(-\Vert \boldsymbol{x} - \boldsymbol{y} \Vert / \sqrt{d})$, but we remark that the results are qualitatively the same when using other Matérn kernels. 

We compare the approximation given by the $f$-greedy (up to $250$ centers) of the 2L-VKOGA  introduced in Section \ref{sec:two-layered} with a standard hyperparameter tuned kernel method.
\begin{itemize}
\item The two-layered kernel is optimized for 25 epochs using a batch size of 64 and the Adam optimizer with an initial learning rate of $5 \cdot 10^{-3}$. 
A regularization of $10^{-5}$ was added to stabilize the numerical calculation of Rippa's formula in Eq.\ \eqref{eq:rippa_loss}.
\item The standard Matérn kernel was used with $10$ logarithmically equally spaced shape parameters $\varepsilon$ between $0.05$ and $10$, i.e.\ $\kappa(\boldsymbol{x}, \boldsymbol{y})=\exp(-\varepsilon \cdot \Vert \boldsymbol{x} - \boldsymbol{y} \Vert / \sqrt{d})$.
\end{itemize}
The results are visualized in Figure \ref{fig:optimized_vs_standard_kernel} and furthermore listed in Table \ref{tab:tabular_results}.

In the left column of Figure \ref{fig:optimized_vs_standard_kernel} we note that the 2L-VKOGA performs better than any of the hyperparameter tuned kernels. 
There seem to be two cases:
\begin{itemize}
\item For $f_5$ (top), the convergence rate of the two-layered kernel is significantly faster than the convergence rates of the standard kernels. Just in the beginning the convergence rate seems to decrease. 
The faster convergence rate can be explained by the final optimized matrix $\mA_{\boldsymbol{\theta}}$, which has eigenvalues
\begin{align*}
\lambda_1 = 2.6623, \lambda_2 = 1.5540 \cdot 10^{-2}, \lambda_3 = 1.8592 \cdot 10^{-3}, \\ 
\lambda_4 = 2.4576 \cdot 10^{-4}, \lambda_5 = -4.1398 \cdot 10^{-3},
\end{align*}
i.e.\ one major eigenvalue and 4 more eigenvalues which are significantly smaller. 
The eigenvector associated to the largest eigenvalue is given by 
\begin{align*}
\boldsymbol{v}_{\lambda_1} = (0.4453, 0.4418, 0.4364, 0.4461, 0.4480)^\intercal,
\end{align*}
which is close to (a multiple) of $(1, 1, 1, 1, 1)^\intercal$ and thus quite well aligned with the active subspace direction of $f_5$.
\item For both $f_6$ and $f_7$ the convergence rate seems to be the same, however the prefactor is smaller, i.e.\ the two-layered approach is consistently better by some factor.
Neither $f_6$ nor $f_7$ have an active subspace, nevertheless still there exist \lq\lq more important directions" for the approximation, which are found by the optimization of the matrix $\mA_{\boldsymbol{\theta}}$ of the first layer.
\end{itemize}

In all the cases the optimized two-layered kernel is better suited for the approximation of the given target function than any kernel with a tuned single hyperparameter in view of the preasymptotic range. These results on function approximation raise also the demand for further theoretical work, 
to better understand the faster convergence rates with approximation theory results. 

\begin{table}[]
\centering 
\begin{tabular}{l|lll}
 & $f_5$ & $f_6$ & $f_7$ \\ \hline
Kernel optimization & 8.936s & 9.173s & 9.817s \\
average VKOGA runtime & 8.087s & 8.329s & 8.787s \\ \hline
2L-VKOGA kernel MSE & $4.351 \cdot 10^{-8}$ & $5.553 \cdot 10^{-4}$ & $2.730 \cdot 10^{-3}$ \\
standard kernel MSE & $9.009 \cdot 10^{-3}$ & $4.174 \cdot 10^{-3}$ & $6.261 \cdot 10^{-3}$ \\ \hline
\end{tabular}
\caption{Optimization and VKOGA runtime for $f_5, f_6$ and $f_7$. 
The optimization of the two-layered kernel takes approximately as much time as the runtime of VKOGA.
The optimized two-layered kernel is several orders of magnitude more accurate, in terms of Mean Squared Error (MSE), than a standard kernel model when using the same expansion size.}
\label{tab:tabular_results}
\end{table}

\begin{figure}[h]
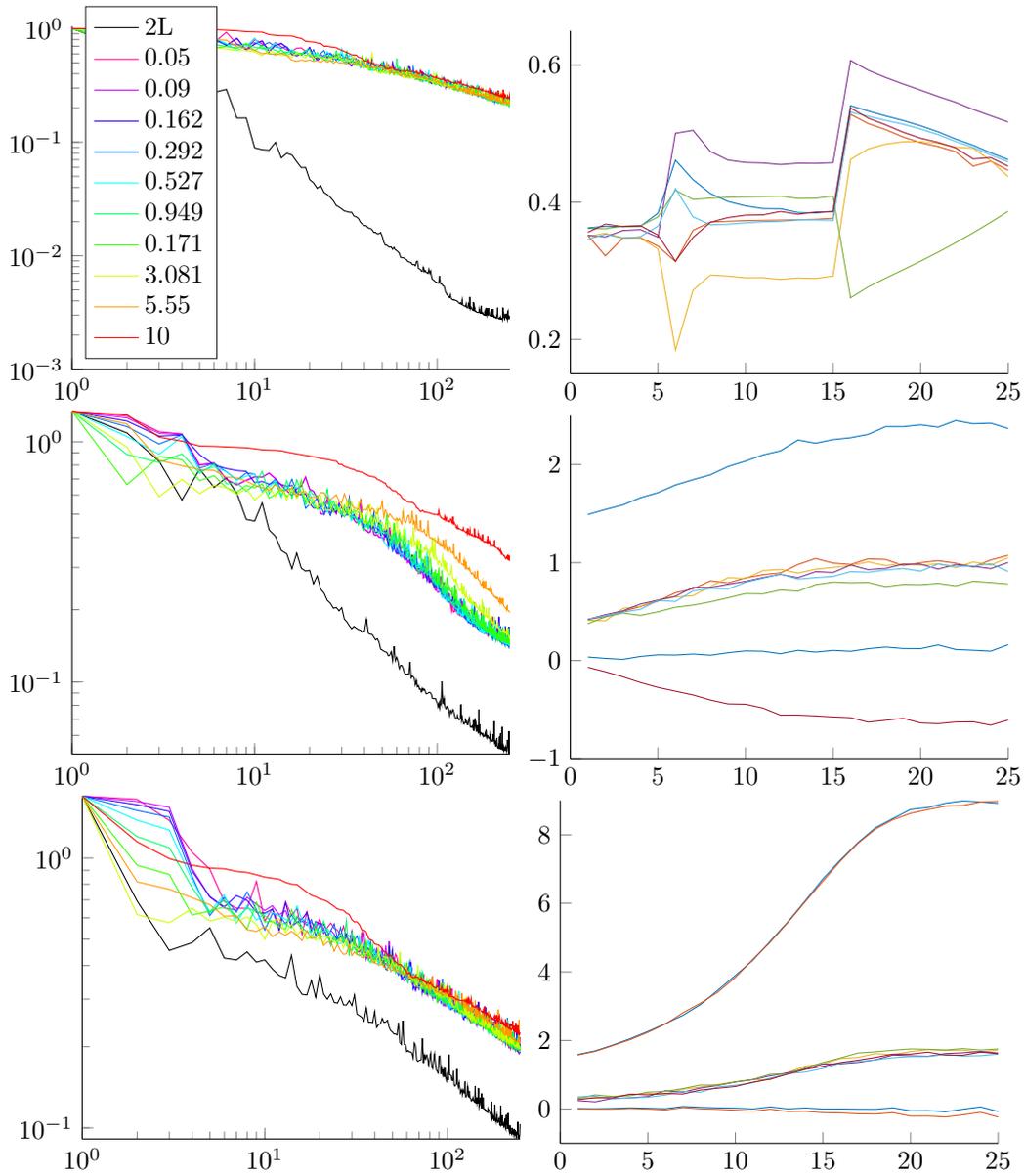

\centering
\setlength\fwidth{.4\textwidth}
\ifplots
\input{Figures/train_hist_5d_faster_conv.tex}
%
%
\definecolor{mycolor1}{rgb}{0.00000,0.44700,0.74100}%
\definecolor{mycolor2}{rgb}{0.85000,0.32500,0.09800}%
\definecolor{mycolor3}{rgb}{0.92900,0.69400,0.12500}%
\definecolor{mycolor4}{rgb}{0.49400,0.18400,0.55600}%
\definecolor{mycolor5}{rgb}{0.46600,0.67400,0.18800}%
\definecolor{mycolor6}{rgb}{0.30100,0.74500,0.93300}%
\definecolor{mycolor7}{rgb}{0.63500,0.07800,0.18400}%
\begin{tikzpicture}

\begin{axis}[%
width=0.951\fwidth,
height=0.75\fwidth,
at={(0\fwidth,0\fwidth)},
scale only axis,
xmin=0,
xmax=25,
ymin=0.15,
ymax=0.65,
axis background/.style={fill=white},
axis x line*=bottom,
axis y line*=left,
legend style={legend cell align=left, align=left, draw=white!15!black}
]
\addplot [color=mycolor1]
  table[row sep=crcr]{%
1	0.362655282020569\\
2	0.364706724882126\\
3	0.364430725574493\\
4	0.364994287490845\\
5	0.384068608283997\\
6	0.46139732003212\\
7	0.432983130216599\\
8	0.412709355354309\\
9	0.401021003723145\\
10	0.394628077745438\\
11	0.390743136405945\\
12	0.390512615442276\\
13	0.38499915599823\\
14	0.384406864643097\\
15	0.386616438627243\\
16	0.540865540504456\\
17	0.533002734184265\\
18	0.525799691677094\\
19	0.519174873828888\\
20	0.51148384809494\\
21	0.502671778202057\\
22	0.492100983858109\\
23	0.4831523001194\\
24	0.47197088599205\\
25	0.462144881486893\\
};

\addplot [color=mycolor2]
  table[row sep=crcr]{%
1	0.351481020450592\\
2	0.321811705827713\\
3	0.347966492176056\\
4	0.348376929759979\\
5	0.336082100868225\\
6	0.31351113319397\\
7	0.359089553356171\\
8	0.37109112739563\\
9	0.372170060873032\\
10	0.372968196868896\\
11	0.373501151800156\\
12	0.373836517333984\\
13	0.374346196651459\\
14	0.374955505132675\\
15	0.376527488231659\\
16	0.52809864282608\\
17	0.514723181724548\\
18	0.505873203277588\\
19	0.495151460170746\\
20	0.486797302961349\\
21	0.481331557035446\\
22	0.473583072423935\\
23	0.452539503574371\\
24	0.459836840629578\\
25	0.446562945842743\\
};

\addplot [color=mycolor3]
  table[row sep=crcr]{%
1	0.349860489368439\\
2	0.354729115962982\\
3	0.348320037126541\\
4	0.34759709239006\\
5	0.331737726926804\\
6	0.184484541416168\\
7	0.271687895059586\\
8	0.293929487466812\\
9	0.292332679033279\\
10	0.289820611476898\\
11	0.290006428956985\\
12	0.287420213222504\\
13	0.289591163396835\\
14	0.288786709308624\\
15	0.29243278503418\\
16	0.462393283843994\\
17	0.47765576839447\\
18	0.484786629676819\\
19	0.488123685121536\\
20	0.488247066736221\\
21	0.485864877700806\\
22	0.480445921421051\\
23	0.478716969490051\\
24	0.460697799921036\\
25	0.437296211719513\\
};

\addplot [color=mycolor4]
  table[row sep=crcr]{%
1	0.351775974035263\\
2	0.349418491125107\\
3	0.358799904584885\\
4	0.360082864761353\\
5	0.349016398191452\\
6	0.500571072101593\\
7	0.504836320877075\\
8	0.473915129899979\\
9	0.461427718400955\\
10	0.457987606525421\\
11	0.457268059253693\\
12	0.454930126667023\\
13	0.456943452358246\\
14	0.456633180379868\\
15	0.457717686891556\\
16	0.606808602809906\\
17	0.592526495456696\\
18	0.581923007965088\\
19	0.572426497936249\\
20	0.563450336456299\\
21	0.554298937320709\\
22	0.545792579650879\\
23	0.535281717777252\\
24	0.52603667974472\\
25	0.51693195104599\\
};

\addplot [color=mycolor5]
  table[row sep=crcr]{%
1	0.361704379320145\\
2	0.361222892999649\\
3	0.365241557359695\\
4	0.365516901016235\\
5	0.378543555736542\\
6	0.417604327201843\\
7	0.404027551412582\\
8	0.405912905931473\\
9	0.407130986452103\\
10	0.407763719558716\\
11	0.408106952905655\\
12	0.408915430307388\\
13	0.405462086200714\\
14	0.405697196722031\\
15	0.408981680870056\\
16	0.260593920946121\\
17	0.276436984539032\\
18	0.289296001195908\\
19	0.301558613777161\\
20	0.314065277576447\\
21	0.327183187007904\\
22	0.340993136167526\\
23	0.35551905632019\\
24	0.370824664831161\\
25	0.386883795261383\\
};

\addplot [color=mycolor6]
  table[row sep=crcr]{%
1	0.346108794212341\\
2	0.353192538022995\\
3	0.347451508045197\\
4	0.349807947874069\\
5	0.365123897790909\\
6	0.419736295938492\\
7	0.378744035959244\\
8	0.36710712313652\\
9	0.367961525917053\\
10	0.369845628738403\\
11	0.371403813362122\\
12	0.372010409832001\\
13	0.374144941568375\\
14	0.374138385057449\\
15	0.373316526412964\\
16	0.531305432319641\\
17	0.52547949552536\\
18	0.519359469413757\\
19	0.513756632804871\\
20	0.50682258605957\\
21	0.498697936534882\\
22	0.488596230745316\\
23	0.479963898658752\\
24	0.468845725059509\\
25	0.459006816148758\\
};

\addplot [color=mycolor7]
  table[row sep=crcr]{%
1	0.356273829936981\\
2	0.368066012859344\\
3	0.364414989948273\\
4	0.365861535072327\\
5	0.352102816104889\\
6	0.31405445933342\\
7	0.348750054836273\\
8	0.370755761861801\\
9	0.377640008926392\\
10	0.381097346544266\\
11	0.381546467542648\\
12	0.386548072099686\\
13	0.382442086935043\\
14	0.385687619447708\\
15	0.386356830596924\\
16	0.537415385246277\\
17	0.522657215595245\\
18	0.512900650501251\\
19	0.502118527889252\\
20	0.493376135826111\\
21	0.487341552972794\\
22	0.479268014431\\
23	0.463087320327759\\
24	0.465028345584869\\
25	0.452613741159439\\
};

\end{axis}
\end{tikzpicture}%
\input{Figures/train_hist_6d_kink.tex}
%
%
\definecolor{mycolor1}{rgb}{0.00000,0.44700,0.74100}%
\definecolor{mycolor2}{rgb}{0.85000,0.32500,0.09800}%
\definecolor{mycolor3}{rgb}{0.92900,0.69400,0.12500}%
\definecolor{mycolor4}{rgb}{0.49400,0.18400,0.55600}%
\definecolor{mycolor5}{rgb}{0.46600,0.67400,0.18800}%
\definecolor{mycolor6}{rgb}{0.30100,0.74500,0.93300}%
\definecolor{mycolor7}{rgb}{0.63500,0.07800,0.18400}%
\begin{tikzpicture}

\begin{axis}[%
width=0.951\fwidth,
height=0.75\fwidth,
at={(0\fwidth,0\fwidth)},
scale only axis,
xmin=0,
xmax=25,
ymin=-1,
ymax=2.5,
axis background/.style={fill=white},
axis x line*=bottom,
axis y line*=left,
legend style={legend cell align=left, align=left, draw=white!15!black}
]
\addplot [color=mycolor1]
  table[row sep=crcr]{%
1	1.49071657657623\\
2	1.54135954380035\\
3	1.58800768852234\\
4	1.66242504119873\\
5	1.71403169631958\\
6	1.79184830188751\\
7	1.84611105918884\\
8	1.89600384235382\\
9	1.97966158390045\\
10	2.0351996421814\\
11	2.09920239448547\\
12	2.14128994941711\\
13	2.25133895874023\\
14	2.21839666366577\\
15	2.25512337684631\\
16	2.27417826652527\\
17	2.30847334861755\\
18	2.38738012313843\\
19	2.38703274726868\\
20	2.40621757507324\\
21	2.38212156295776\\
22	2.44952440261841\\
23	2.41645836830139\\
24	2.42311382293701\\
25	2.36823844909668\\
};

\addplot [color=mycolor2]
  table[row sep=crcr]{%
1	0.416855812072754\\
2	0.44920489192009\\
3	0.500184774398804\\
4	0.550527215003967\\
5	0.605735778808594\\
6	0.691612243652344\\
7	0.730441153049469\\
8	0.814212143421173\\
9	0.791399359703064\\
10	0.843945503234863\\
11	0.876512587070465\\
12	0.894899308681488\\
13	0.981036245822906\\
14	1.04176425933838\\
15	0.995210886001587\\
16	0.981163322925568\\
17	1.03811430931091\\
18	1.03188478946686\\
19	0.971059381961823\\
20	0.999592542648315\\
21	1.02007710933685\\
22	0.993998289108276\\
23	0.955316781997681\\
24	1.02685117721558\\
25	1.07369446754456\\
};

\addplot [color=mycolor3]
  table[row sep=crcr]{%
1	0.40976271033287\\
2	0.405505031347275\\
3	0.531993448734283\\
4	0.552497327327728\\
5	0.620084404945374\\
6	0.647971928119659\\
7	0.661277711391449\\
8	0.75530081987381\\
9	0.845808327198029\\
10	0.837840974330902\\
11	0.919377088546753\\
12	0.931109189987183\\
13	0.893386960029602\\
14	0.932862222194672\\
15	0.95091849565506\\
16	0.971520721912384\\
17	1.007737159729\\
18	0.969889879226685\\
19	0.986599385738373\\
20	0.976288557052612\\
21	0.987463057041168\\
22	0.949829041957855\\
23	1.00640022754669\\
24	0.975843727588654\\
25	1.04983329772949\\
};

\addplot [color=mycolor4]
  table[row sep=crcr]{%
1	0.420461356639862\\
2	0.467831403017044\\
3	0.507381558418274\\
4	0.576356887817383\\
5	0.615510940551758\\
6	0.655916690826416\\
7	0.746101498603821\\
8	0.745819091796875\\
9	0.778564870357513\\
10	0.808434188365936\\
11	0.847869098186493\\
12	0.879725396633148\\
13	0.846398413181305\\
14	0.907800257205963\\
15	0.899262189865112\\
16	0.975942969322205\\
17	0.928428471088409\\
18	0.9411301612854\\
19	0.984637796878815\\
20	0.983996152877808\\
21	0.935151338577271\\
22	0.980562925338745\\
23	0.964151203632355\\
24	0.939296543598175\\
25	1.0017341375351\\
};

\addplot [color=mycolor5]
  table[row sep=crcr]{%
1	0.376306235790253\\
2	0.445020496845245\\
3	0.484659403562546\\
4	0.461056530475616\\
5	0.497749716043472\\
6	0.5437873005867\\
7	0.565505385398865\\
8	0.598800539970398\\
9	0.639890611171722\\
10	0.683091044425964\\
11	0.679995179176331\\
12	0.719017624855042\\
13	0.707432568073273\\
14	0.77479088306427\\
15	0.801317512989044\\
16	0.793516516685486\\
17	0.797144949436188\\
18	0.755079567432404\\
19	0.777164041996002\\
20	0.773775696754456\\
21	0.788084506988525\\
22	0.762811064720154\\
23	0.809251725673676\\
24	0.796435713768005\\
25	0.781101107597351\\
};

\addplot [color=mycolor6]
  table[row sep=crcr]{%
1	0.423256039619446\\
2	0.456554651260376\\
3	0.49164479970932\\
4	0.523161888122559\\
5	0.613095700740814\\
6	0.602579593658447\\
7	0.709830939769745\\
8	0.734346628189087\\
9	0.730390608310699\\
10	0.798387825489044\\
11	0.840139925479889\\
12	0.882631182670593\\
13	0.831124782562256\\
14	0.846956491470337\\
15	0.859144806861877\\
16	0.908178925514221\\
17	0.901362776756287\\
18	0.924547791481018\\
19	0.940509080886841\\
20	0.912784993648529\\
21	0.987651288509369\\
22	0.960791289806366\\
23	0.965281784534454\\
24	0.989750385284424\\
25	0.909563302993774\\
};

\addplot [color=mycolor7]
  table[row sep=crcr]{%
1	-0.0688645169138908\\
2	-0.114971078932285\\
3	-0.166440606117249\\
4	-0.2253348082304\\
5	-0.274571239948273\\
6	-0.313361346721649\\
7	-0.35138338804245\\
8	-0.403505474328995\\
9	-0.442892372608185\\
10	-0.446044921875\\
11	-0.485320925712585\\
12	-0.557541072368622\\
13	-0.556771874427795\\
14	-0.564958333969116\\
15	-0.575411200523376\\
16	-0.583025455474854\\
17	-0.628840386867523\\
18	-0.609229743480682\\
19	-0.589344680309296\\
20	-0.637377858161926\\
21	-0.644401729106903\\
22	-0.627228438854218\\
23	-0.625053703784943\\
24	-0.659671485424042\\
25	-0.607025682926178\\
};

\addplot [color=mycolor1]
  table[row sep=crcr]{%
1	0.0357886590063572\\
2	0.0218284390866756\\
3	0.0118480380624533\\
4	0.0418303906917572\\
5	0.0583440475165844\\
6	0.0558630377054214\\
7	0.0665247067809105\\
8	0.0541851036250591\\
9	0.0797440335154533\\
10	0.0985146686434746\\
11	0.0946167409420013\\
12	0.0698122084140778\\
13	0.104495525360107\\
14	0.0857371613383293\\
15	0.103562690317631\\
16	0.0939729958772659\\
17	0.120627284049988\\
18	0.137385278940201\\
19	0.122490659356117\\
20	0.121126614511013\\
21	0.158981055021286\\
22	0.112626820802689\\
23	0.104700237512589\\
24	0.0953338518738747\\
25	0.161414340138435\\
};

\end{axis}
\end{tikzpicture}%
\input{Figures/train_hist_7d_semiactive.tex}
%
%
\definecolor{mycolor1}{rgb}{0.00000,0.44700,0.74100}%
\definecolor{mycolor2}{rgb}{0.85000,0.32500,0.09800}%
\definecolor{mycolor3}{rgb}{0.92900,0.69400,0.12500}%
\definecolor{mycolor4}{rgb}{0.49400,0.18400,0.55600}%
\definecolor{mycolor5}{rgb}{0.46600,0.67400,0.18800}%
\definecolor{mycolor6}{rgb}{0.30100,0.74500,0.93300}%
\definecolor{mycolor7}{rgb}{0.63500,0.07800,0.18400}%
\begin{tikzpicture}

\begin{axis}[%
width=0.951\fwidth,
height=0.75\fwidth,
at={(0\fwidth,0\fwidth)},
scale only axis,
xmin=0,
xmax=25,
ymin=-1,
ymax=9,
axis background/.style={fill=white},
axis x line*=bottom,
axis y line*=left,
legend style={legend cell align=left, align=left, draw=white!15!black}
]
\addplot [color=mycolor1]
  table[row sep=crcr]{%
1	1.59130847454071\\
2	1.69828474521637\\
3	1.86628437042236\\
4	2.0532763004303\\
5	2.26192235946655\\
6	2.49161148071289\\
7	2.73118996620178\\
8	3.0440092086792\\
9	3.46334457397461\\
10	3.89661526679993\\
11	4.33873558044434\\
12	4.89925718307495\\
13	5.4603009223938\\
14	6.07887268066406\\
15	6.71519231796265\\
16	7.26400327682495\\
17	7.77366638183594\\
18	8.20570087432861\\
19	8.4756498336792\\
20	8.73793697357178\\
21	8.79599285125732\\
22	8.9225492477417\\
23	8.98991870880127\\
24	8.96211338043213\\
25	8.91784381866455\\
};

\addplot [color=mycolor2]
  table[row sep=crcr]{%
1	1.57589960098267\\
2	1.68363237380981\\
3	1.85089755058289\\
4	2.02451372146606\\
5	2.23433041572571\\
6	2.47651934623718\\
7	2.80492186546326\\
8	3.08471536636353\\
9	3.40728521347046\\
10	3.83969879150391\\
11	4.3660717010498\\
12	4.87067079544067\\
13	5.44347476959229\\
14	6.05154371261597\\
15	6.64870834350586\\
16	7.23305892944336\\
17	7.7541241645813\\
18	8.16721439361572\\
19	8.44116878509521\\
20	8.6252555847168\\
21	8.73724365234375\\
22	8.83656978607178\\
23	8.85562801361084\\
24	8.96045684814453\\
25	8.97795581817627\\
};

\addplot [color=mycolor3]
  table[row sep=crcr]{%
1	0.330161452293396\\
2	0.322624802589417\\
3	0.365474343299866\\
4	0.413821130990982\\
5	0.457177698612213\\
6	0.43763467669487\\
7	0.579929888248444\\
8	0.642691493034363\\
9	0.716446161270142\\
10	0.794654071331024\\
11	0.863819718360901\\
12	0.859086155891418\\
13	1.06790959835052\\
14	1.25555849075317\\
15	1.31393444538116\\
16	1.47409892082214\\
17	1.50608968734741\\
18	1.61573803424835\\
19	1.58561325073242\\
20	1.67559540271759\\
21	1.73336565494537\\
22	1.70767092704773\\
23	1.72057163715363\\
24	1.65822815895081\\
25	1.72411775588989\\
};

\addplot [color=mycolor4]
  table[row sep=crcr]{%
1	0.247312128543854\\
2	0.210770085453987\\
3	0.314618408679962\\
4	0.326151758432388\\
5	0.365959167480469\\
6	0.534582376480103\\
7	0.493880718946457\\
8	0.614031374454498\\
9	0.689180314540863\\
10	0.788229882717133\\
11	0.857597947120667\\
12	0.970601499080658\\
13	1.0581157207489\\
14	1.17222511768341\\
15	1.31358516216278\\
16	1.31845605373383\\
17	1.35912036895752\\
18	1.43737041950226\\
19	1.57528030872345\\
20	1.55251109600067\\
21	1.5402979850769\\
22	1.61911427974701\\
23	1.64163863658905\\
24	1.6817489862442\\
25	1.63408899307251\\
};

\addplot [color=mycolor5]
  table[row sep=crcr]{%
1	0.312533348798752\\
2	0.414349317550659\\
3	0.366733908653259\\
4	0.390038430690765\\
5	0.494095742702484\\
6	0.516449332237244\\
7	0.596977829933167\\
8	0.702353656291962\\
9	0.696377635002136\\
10	0.791922032833099\\
11	0.864043951034546\\
12	1.01505184173584\\
13	1.02706265449524\\
14	1.20461976528168\\
15	1.35722458362579\\
16	1.4903975725174\\
17	1.63506305217743\\
18	1.66923725605011\\
19	1.71561515331268\\
20	1.7533483505249\\
21	1.74301362037659\\
22	1.72350454330444\\
23	1.76359415054321\\
24	1.72075128555298\\
25	1.75323975086212\\
};

\addplot [color=mycolor6]
  table[row sep=crcr]{%
1	0.357113718986511\\
2	0.373195379972458\\
3	0.298520505428314\\
4	0.327450096607208\\
5	0.345057994127274\\
6	0.408230811357498\\
7	0.521695375442505\\
8	0.491372644901276\\
9	0.644180357456207\\
10	0.693395435810089\\
11	0.764357447624207\\
12	0.902453601360321\\
13	1.03181123733521\\
14	1.07804322242737\\
15	1.18243741989136\\
16	1.33699476718903\\
17	1.38441872596741\\
18	1.43331789970398\\
19	1.49786794185638\\
20	1.52967536449432\\
21	1.53074634075165\\
22	1.59150445461273\\
23	1.54613971710205\\
24	1.55666065216064\\
25	1.60135018825531\\
};

\addplot [color=mycolor7]
  table[row sep=crcr]{%
1	0.275579869747162\\
2	0.331184059381485\\
3	0.330758541822433\\
4	0.434656649827957\\
5	0.416766494512558\\
6	0.456673890352249\\
7	0.439315557479858\\
8	0.555349946022034\\
9	0.608504235744476\\
10	0.658279597759247\\
11	0.782848238945007\\
12	0.8717360496521\\
13	1.02108132839203\\
14	1.15555119514465\\
15	1.24538326263428\\
16	1.35480034351349\\
17	1.42429506778717\\
18	1.50664854049683\\
19	1.60087835788727\\
20	1.60105347633362\\
21	1.65952754020691\\
22	1.57446432113647\\
23	1.56993198394775\\
24	1.66124379634857\\
25	1.61635148525238\\
};

\addplot [color=mycolor1]
  table[row sep=crcr]{%
1	0.016448512673378\\
2	0.0106406807899475\\
3	0.0195757504552603\\
4	0.0370286367833614\\
5	0.024498550221324\\
6	0.0112396497279406\\
7	0.0761632919311523\\
8	0.0523805841803551\\
9	0.03540064021945\\
10	0.02915009111166\\
11	0.0034450099337846\\
12	0.0636796578764915\\
13	0.00495172571390867\\
14	0.0324491150677204\\
15	0.00173487816937268\\
16	0.00371058122254908\\
17	-0.00418328586965799\\
18	-0.00636858819052577\\
19	0.0366100668907166\\
20	-0.0488493703305721\\
21	-0.0451893657445908\\
22	-0.0791141837835312\\
23	-0.00301972217857838\\
24	0.0635900795459747\\
25	-0.0706094205379486\\
};

\addplot [color=mycolor2]
  table[row sep=crcr]{%
1	0.00689380103722215\\
2	-0.000567731214687228\\
3	-0.00120143289677799\\
4	0.0183633994311094\\
5	0.00177051383070648\\
6	-0.025550065562129\\
7	0.0448913425207138\\
8	0.00858901627361774\\
9	-0.00639494461938739\\
10	-0.0297076497226954\\
11	-0.0600821487605572\\
12	-0.00170811312273145\\
13	-0.0712547078728676\\
14	-0.0638572350144386\\
15	-0.103517979383469\\
16	-0.111354678869247\\
17	-0.136982038617134\\
18	-0.14463484287262\\
19	-0.104495726525784\\
20	-0.20498152077198\\
21	-0.202005580067635\\
22	-0.229747459292412\\
23	-0.168196305632591\\
24	-0.0937421917915344\\
25	-0.230514124035835\\
};

\end{axis}
\end{tikzpicture}%
\else

\fi
\caption{Top: 5D experiment, middle: 6D experiment, bottom: 7D experiment. 
Left: Visualization of the $\Vert \cdot \Vert_{L^\infty(\Omega)}$ error ($y$-axis) in the number of greedily selected points ($x$-axis) for both the standard kernel and the 2L-VKOGA.
Right: Visualization of the change of several matrix entries of the matrix $\mA_{\boldsymbol{\theta}}$ ($y$-axis) during the optimization epochs ($x$-axis).}
\label{fig:optimized_vs_standard_kernel}
\end{figure}

\subsection{Machine learning data sets} \label{subsec:num_ex_ml_data}

In order to show the usability of our 2L-VKOGA method also on real world datasets, 
we compare our method to standard hyperparameter tuned approaches on 12 out of 15 regression datasets which were used in \cite{holzmuller2022framework}. 
We excluded the three datasets: \textit{methane}, \textit{poker} and \textit{protein} because those datasets are not suitable for sparse kernel models. 
These datasets have input dimensionality ranging from 2 to 379 and sampling sizes between 8153 and 300000, 
see Table \ref{tab:tabular_datasets}. For more details on the used datasets, we refer to \cite[Table E.1, E.2]{holzmuller2022framework}.

We note that the motivation of surrogate modeling might not be intrinsically given on those datasets, however they still provide meaningful benchmark to compare our method with standard cross validated kernel models.

We used the same setup for our comparison as in Subsection \ref{subsec:num_ex_func_approx}, 
with the small modification that we used a regularization of $10^{-3}$ for the stabilization of the numerical calculation of Rippa's formula and a regularization of $10^{-4}$ within the greedy approximation.
These higher regularizations are due to the fact that these real world datasets are possibly noisy.
Note that these hyperparameters were used out of the box without any fine tuning or adaption to any specific datasets, which shows the general applicability of the approach.
The experiments were rerun five times for different training test splits in order to mitigate the randomness of the split. We note that we did not rerun the (partly randomized, due to batch selection) kernel optimization several times within a given training test split. This is discussed in more detail in Subsection \ref{subsec:num_stability}, which highlights the robustness and stability of our optimization procedure.

The results for the different datasets are visualized in the Figures \ref{fig:ml_datasets_results_1} and \ref{fig:ml_datasets_results_2}, where the decay of the test MSE error is displayed for one out of the five reruns. 
We waived to include the error bars, because they provided very limited further insights but made the visualizations of the results more difficult. \\ 
The results for the 12 datasets are ordered from \textit{better} to \textit{worse} (from the 2L-VKOGA point of view) by using the mean relative improvement of the 2L-VKOGA approach in comparison to the best hyperparameter tuned kernel as a criterion.
\begin{itemize}
\item For the six datasets \textit{ct}, \textit{sgemm}, \textit{wecs}, \textit{mlr\_knn\_rng}, \textit{fried} and \textit{kegg\_undir\_uci} one can observe that the 2L-VKOGA approach is quite consistently better than any hyperparameter tuned kernel. 
Especially for e.g.\ \textit{ct} or \textit{sgemm}, the 2L-VKOGA approach with around 50 centers can already reach the same accuracy as the best hyperparameter tuned kernel model with 1000 centers.
We note that in some datasets (e.g.\ in \textit{fried}) one can observe an overall saturation of the accuracy, for any method. 
However the 2L-VKOGA approach achieves this saturation already with a way smaller expansion size. 

One should note that this does not necessarily imply that the 2L-VKOGA approach is cheaper to evaluate, because the linear mapping of the first layer is possibly costly, especially e.g.\ for the 379 dimensional \textit{ct} dataset. 
Though for the 14 dimensional \textit{sgemm} dataset we can for sure expect a benefit.
\item For the six datasets \textit{diamonds}, \textit{sarcos}, \textit{stock}, \textit{road\_network}, \textit{online\_video} and \textit{query\_agg\_count}, the 2L-VKOGA approach is not consistently better:
One can observe alternating performances e.g.\ for the \textit{diamonds}, \textit{sarcos} and \textit{online\_video} dataset, where the 2L-VKOGA approach provides better models for small expansion sizes, but no longer for large expansion sizes.
For \textit{stock} and \textit{road\_network} the 2L-VKOGA approach is (asymptotically) on par with the hyperparameter tuned method.
Only for the \textit{query\_agg\_count} the 2L-VKOGA approach seems to be consistently inferior to well chosen hyperparameter tuned kernels.
A possible remedy in view of the alternating behaviour might be larger or even varying batch sizes during training or other optimization objectives. 
We leave these points for future research, while in the following we try to give some explanations about the fact that the  2L-VKOGA is not necessarily better than a well chosen hyperparameter tuned model.
\end{itemize}
We can again leverage the eigenvalues of the optimized first layer matrix $\mA_{\bs{\theta}}$:
A standard kernel can be seen as using the identity matrix and therefore having all eigenvalues equal to one, meaning that all directions in the Euclidean space are of equal importance.
A matrix $\mA_{\bs{\theta}}$ with only a few large singular values might, on the other hand side,  indicate that there are more important directions in the dataset. Therefore we use the \textit{cumulative power}
\begin{align}
\label{eq:singularvalue_ratio}
n \mapsto \frac{\sum_{i=1}^n |s_i(\mA_{\bs{\theta}})|}{\sum_{i=1}^d |s_i(\mA_{\bs{\theta}})|}, \quad 1 \leq n \leq d
\end{align}
as a criterion (whereby the singular values $s_i(\mA_{\bs{\theta}})$ are ordered from the largest to the smallest one according to their absolute value) that allows us to understand how much power is clustered in the top singular values.
The behaviour of the quantity in Eq.\ \eqref{eq:singularvalue_ratio} is depicted in Figure \ref{fig:overview_singularvalues}, whereby $n/d$ is used for the $x$-axis. 
This allows us to compare datasets with different dimension $d$. 
One can see that the 2L-VKOGA models are better, if the quantity Eq.\ \eqref{eq:singularvalue_ratio} increases quickly, 
i.e.\ if a lot of power is capture by a few large singular values, which is the case for example for \textit{ct} or \textit{sgemm}. 
On the other hand, a slow increase (e.g.\ for \textit{query\_agg\_count} or \textit{online\_video}) is linked to the fact that we have no benefit in using the 2L-VKOGA approach. 
Therefore the decay of the absolute values of the singular values $s_i(\mA_{\boldsymbol{\theta}})$, or especially the existence of particular large singular values, can be used as a criterion to assess whether the kernel optimization is likely to improve the kernel model or not.

Based on this observation it might make sense to use and optimize a $b \times d$ matrix with $b < d$ instead of the full $d \times d$ matrix.
 We leave this idea to future research. 
 Furthermore another natural idea due to using the MSE as a test metric would be to use a least square approximation with the selected $f$-greedy centers.
However, as this would be applied to both the classical and 2L-VKOGA model, it would not change the comparison a lot.

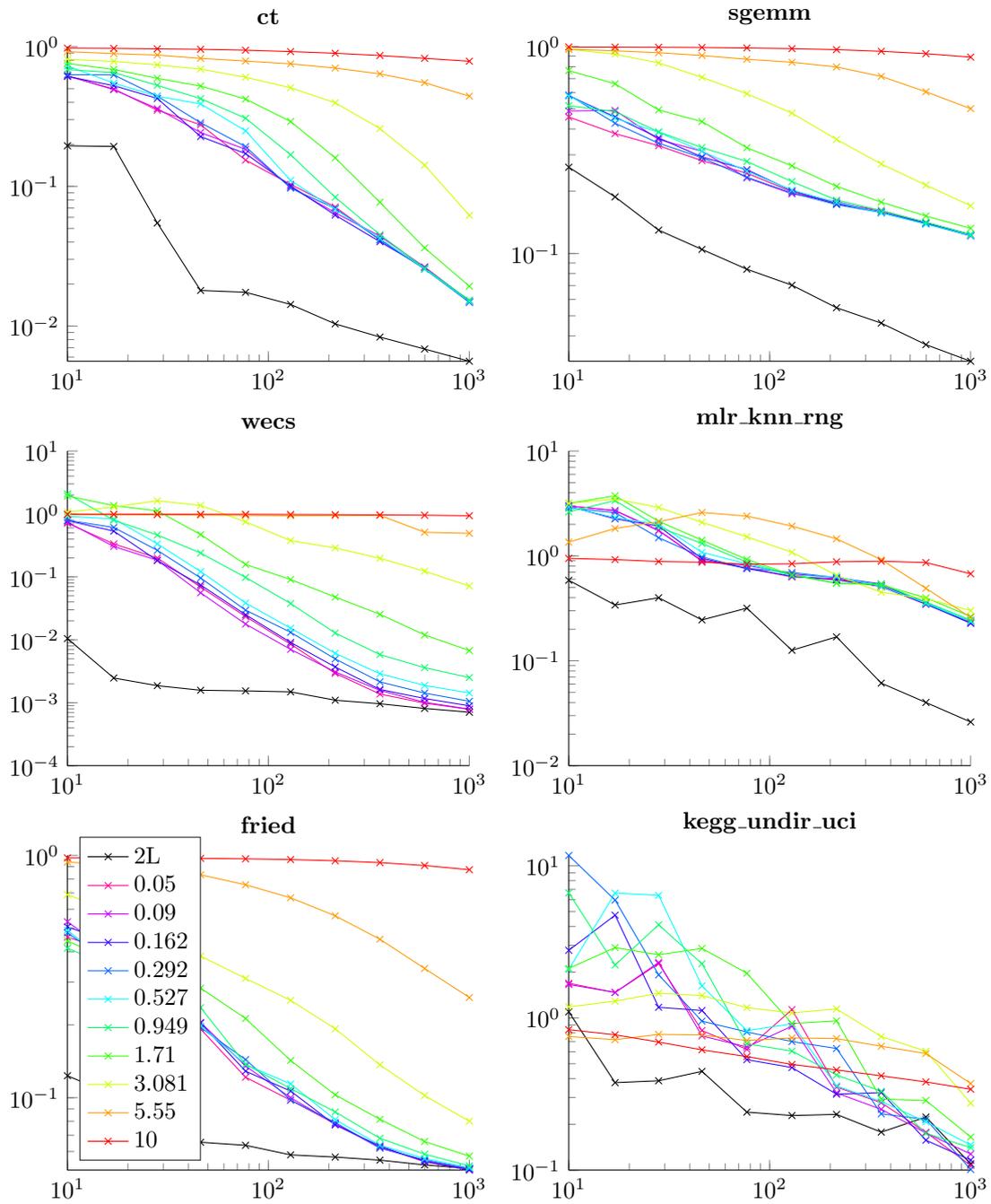
\begin{figure}[h]
\centering
\setlength\fwidth{.4\textwidth}
\ifplots
%
%
\definecolor{mycolor1}{rgb}{1.00000,0.00000,0.60000}%
\definecolor{mycolor2}{rgb}{0.80000,0.00000,1.00000}%
\definecolor{mycolor3}{rgb}{0.00000,0.40000,1.00000}%
\definecolor{mycolor4}{rgb}{0.00000,1.00000,1.00000}%
\definecolor{mycolor5}{rgb}{0.20000,1.00000,0.00000}%
\definecolor{mycolor6}{rgb}{0.80000,1.00000,0.00000}%
\definecolor{mycolor7}{rgb}{1.00000,0.60000,0.00000}%
\begin{tikzpicture}

\begin{axis}[%
width=0.951\fwidth,
height=0.75\fwidth,
at={(0\fwidth,0\fwidth)},
scale only axis,
xmode=log,
xmin=10,
xmax=1000,
xminorticks=true,
ymode=log,
ymin=0.0056016652981051,
ymax=1,
yminorticks=true,
axis background/.style={fill=white},
title style={font=\bfseries},
title={ct},
axis x line*=bottom,
axis y line*=left
]
\addplot [color=black, mark=x, mark options={solid, black}, forget plot]
  table[row sep=crcr]{%
10	0.195101138680005\\
17	0.192974055605198\\
28	0.0546845312261571\\
46	0.0179768405892351\\
77	0.017444545916417\\
129	0.0142710050100011\\
215	0.0103631408551339\\
359	0.00834543203674544\\
599	0.00685034222701795\\
1000	0.0056016652981051\\
};
\addplot [color=mycolor1, mark=x, mark options={solid, mycolor1}, forget plot]
  table[row sep=crcr]{%
10	0.615457176884671\\
17	0.497199575105351\\
28	0.35062485614416\\
46	0.275853462995589\\
77	0.154155661654851\\
129	0.103238491688894\\
215	0.0711018083227623\\
359	0.0415001515612093\\
599	0.0262754589286846\\
1000	0.0149879789895131\\
};
\addplot [color=mycolor2, mark=x, mark options={solid, mycolor2}, forget plot]
  table[row sep=crcr]{%
10	0.619484943131836\\
17	0.491606895501476\\
28	0.360148013295966\\
46	0.242536894247434\\
77	0.184044985310702\\
129	0.0984441561344756\\
215	0.0646459775760587\\
359	0.0440131960484144\\
599	0.0264683902661778\\
1000	0.0151421695273276\\
};
\addplot [color=blue!75!mycolor2, mark=x, mark options={solid, blue!75!mycolor2}, forget plot]
  table[row sep=crcr]{%
10	0.611973050748944\\
17	0.525609810002947\\
28	0.425638125222309\\
46	0.227033252485265\\
77	0.171559181733331\\
129	0.100771358451223\\
215	0.0622236184101038\\
359	0.0401330599437309\\
599	0.0260285004771206\\
1000	0.0148236056356429\\
};
\addplot [color=mycolor3, mark=x, mark options={solid, mycolor3}, forget plot]
  table[row sep=crcr]{%
10	0.628588347939125\\
17	0.627447803912656\\
28	0.442186738361674\\
46	0.285210825477752\\
77	0.192440765931287\\
129	0.0970788305603453\\
215	0.0694362883929151\\
359	0.0412802291994059\\
599	0.0255424442310989\\
1000	0.0147981096907319\\
};
\addplot [color=mycolor4, mark=x, mark options={solid, mycolor4}, forget plot]
  table[row sep=crcr]{%
10	0.72623634154211\\
17	0.548354301830035\\
28	0.441872841582521\\
46	0.389227607754806\\
77	0.249406997219424\\
129	0.110159515185566\\
215	0.0682944188627945\\
359	0.0423625894320687\\
599	0.0256691136551473\\
1000	0.0149756891037615\\
};
\addplot [color=green!60!mycolor4, mark=x, mark options={solid, green!60!mycolor4}, forget plot]
  table[row sep=crcr]{%
10	0.685442945870278\\
17	0.652180542893373\\
28	0.528130460756437\\
46	0.424514552764893\\
77	0.307150218671679\\
129	0.168322669211189\\
215	0.083532446016727\\
359	0.0450408853246325\\
599	0.0259806607578818\\
1000	0.0153535444826857\\
};
\addplot [color=mycolor5, mark=x, mark options={solid, mycolor5}, forget plot]
  table[row sep=crcr]{%
10	0.757493186066474\\
17	0.686967527335166\\
28	0.593599185372387\\
46	0.521927222073709\\
77	0.421403204835489\\
129	0.291136685160508\\
215	0.159905411098921\\
359	0.0768345613285375\\
599	0.0362628382099751\\
1000	0.0192809277070129\\
};
\addplot [color=mycolor6, mark=x, mark options={solid, mycolor6}, forget plot]
  table[row sep=crcr]{%
10	0.811278843916377\\
17	0.783550519514363\\
28	0.740354441269578\\
46	0.691005036157857\\
77	0.603413837526068\\
129	0.507709017917158\\
215	0.396373924846253\\
359	0.259206193026127\\
599	0.141903376893622\\
1000	0.0619595357481452\\
};
\addplot [color=mycolor7, mark=x, mark options={solid, mycolor7}, forget plot]
  table[row sep=crcr]{%
10	0.92094691105173\\
17	0.891104388985904\\
28	0.870679143441685\\
46	0.823785055110179\\
77	0.788357950104485\\
129	0.752947390308242\\
215	0.702677047926845\\
359	0.638192687696596\\
599	0.551693214730915\\
1000	0.442143552717374\\
};
\addplot [color=red, mark=x, mark options={solid, red}, forget plot]
  table[row sep=crcr]{%
10	0.97535060666241\\
17	0.971906960001754\\
28	0.964676399071394\\
46	0.956497824585202\\
77	0.942910324035708\\
129	0.921510007648014\\
215	0.896431887454758\\
359	0.86354528892476\\
599	0.824099645979002\\
1000	0.785808784751035\\
};
\end{axis}
\end{tikzpicture}%
%
%
\definecolor{mycolor1}{rgb}{1.00000,0.00000,0.60000}%
\definecolor{mycolor2}{rgb}{0.80000,0.00000,1.00000}%
\definecolor{mycolor3}{rgb}{0.00000,0.40000,1.00000}%
\definecolor{mycolor4}{rgb}{0.00000,1.00000,1.00000}%
\definecolor{mycolor5}{rgb}{0.20000,1.00000,0.00000}%
\definecolor{mycolor6}{rgb}{0.80000,1.00000,0.00000}%
\definecolor{mycolor7}{rgb}{1.00000,0.60000,0.00000}%
\begin{tikzpicture}

\begin{axis}[%
width=0.951\fwidth,
height=0.75\fwidth,
at={(0\fwidth,0\fwidth)},
scale only axis,
xmode=log,
xmin=10,
xmax=1000,
xminorticks=true,
ymode=log,
ymin=0.030173160561517,
ymax=1,
yminorticks=true,
axis background/.style={fill=white},
title style={font=\bfseries},
title={sgemm},
axis x line*=bottom,
axis y line*=left
]
\addplot [color=black, mark=x, mark options={solid, black}, forget plot]
  table[row sep=crcr]{%
10	0.261278417405402\\
17	0.187844592528063\\
28	0.129721374690587\\
46	0.10469606089587\\
77	0.083953946061946\\
129	0.0703116340150106\\
215	0.0546978255978039\\
359	0.0461010645421884\\
599	0.036294796833887\\
1000	0.030173160561517\\
};
\addplot [color=mycolor1, mark=x, mark options={solid, mycolor1}, forget plot]
  table[row sep=crcr]{%
10	0.457142332856729\\
17	0.379812049583636\\
28	0.331966733634856\\
46	0.281077931876366\\
77	0.244445334956676\\
129	0.19897922589057\\
215	0.174650031219036\\
359	0.161498662503282\\
599	0.140517626196868\\
1000	0.122631541663682\\
};
\addplot [color=mycolor2, mark=x, mark options={solid, mycolor2}, forget plot]
  table[row sep=crcr]{%
10	0.488352988785453\\
17	0.491828528026912\\
28	0.355462103824396\\
46	0.312284976340961\\
77	0.232114941102937\\
129	0.194415365462842\\
215	0.178056057984848\\
359	0.157556864687116\\
599	0.139165642300233\\
1000	0.121567999064952\\
};
\addplot [color=blue!75!mycolor2, mark=x, mark options={solid, blue!75!mycolor2}, forget plot]
  table[row sep=crcr]{%
10	0.579253044222296\\
17	0.457166193875765\\
28	0.36218511649481\\
46	0.292992985585466\\
77	0.254181819698247\\
129	0.201132187831554\\
215	0.17391489958089\\
359	0.158774721799313\\
599	0.13998497646654\\
1000	0.122206636880612\\
};
\addplot [color=mycolor3, mark=x, mark options={solid, mycolor3}, forget plot]
  table[row sep=crcr]{%
10	0.583410250658511\\
17	0.426268014925041\\
28	0.341275858032448\\
46	0.290651982030808\\
77	0.233652690945669\\
129	0.196226352861733\\
215	0.172301696316061\\
359	0.157653073391317\\
599	0.139448879998042\\
1000	0.122629671200124\\
};
\addplot [color=mycolor4, mark=x, mark options={solid, mycolor4}, forget plot]
  table[row sep=crcr]{%
10	0.576822624733058\\
17	0.450490820001704\\
28	0.385572722093788\\
46	0.311950637819321\\
77	0.250215140334625\\
129	0.202619356386897\\
215	0.176520701362385\\
359	0.157285518159535\\
599	0.138617026094902\\
1000	0.122038938053738\\
};
\addplot [color=green!60!mycolor4, mark=x, mark options={solid, green!60!mycolor4}, forget plot]
  table[row sep=crcr]{%
10	0.518291901880005\\
17	0.485166787808584\\
28	0.387508531961611\\
46	0.325280469470852\\
77	0.278556264405508\\
129	0.222756101741391\\
215	0.181215288573656\\
359	0.161258932032405\\
599	0.141813165500589\\
1000	0.123959546777945\\
};
\addplot [color=mycolor5, mark=x, mark options={solid, mycolor5}, forget plot]
  table[row sep=crcr]{%
10	0.7659352592394\\
17	0.661449501755443\\
28	0.494375583556211\\
46	0.434261560179317\\
77	0.324546943058358\\
129	0.265203983802148\\
215	0.211143305236272\\
359	0.177234005260493\\
599	0.151802207503769\\
1000	0.132346863828378\\
};
\addplot [color=mycolor6, mark=x, mark options={solid, mycolor6}, forget plot]
  table[row sep=crcr]{%
10	0.970938850019895\\
17	0.918993637037231\\
28	0.833957505199117\\
46	0.710740152376965\\
77	0.591973164755887\\
129	0.476815842616487\\
215	0.356180099400769\\
359	0.2708528717332\\
599	0.214072900985346\\
1000	0.17015429914437\\
};
\addplot [color=mycolor7, mark=x, mark options={solid, mycolor7}, forget plot]
  table[row sep=crcr]{%
10	0.971196039208915\\
17	0.954337144315112\\
28	0.93439513208613\\
46	0.906125259804762\\
77	0.868948827271256\\
129	0.84065886883419\\
215	0.797603517254769\\
359	0.71806441069213\\
599	0.605113543950972\\
1000	0.501349412171998\\
};
\addplot [color=red, mark=x, mark options={solid, red}, forget plot]
  table[row sep=crcr]{%
10	0.996196287037811\\
17	0.994950510778207\\
28	0.993125359745759\\
46	0.990568519129521\\
77	0.985295499359937\\
129	0.977960684597122\\
215	0.966894223250957\\
359	0.949229836687218\\
599	0.924543566755193\\
1000	0.889806409628056\\
};
\end{axis}
\end{tikzpicture}%
%
%
\definecolor{mycolor1}{rgb}{1.00000,0.00000,0.60000}%
\definecolor{mycolor2}{rgb}{0.80000,0.00000,1.00000}%
\definecolor{mycolor3}{rgb}{0.00000,0.40000,1.00000}%
\definecolor{mycolor4}{rgb}{0.00000,1.00000,1.00000}%
\definecolor{mycolor5}{rgb}{0.20000,1.00000,0.00000}%
\definecolor{mycolor6}{rgb}{0.80000,1.00000,0.00000}%
\definecolor{mycolor7}{rgb}{1.00000,0.60000,0.00000}%
\begin{tikzpicture}

\begin{axis}[%
width=0.951\fwidth,
height=0.75\fwidth,
at={(0\fwidth,0\fwidth)},
scale only axis,
xmode=log,
xmin=10,
xmax=1000,
xminorticks=true,
ymode=log,
ymin=0.0001,
ymax=10,
yminorticks=true,
axis background/.style={fill=white},
title style={font=\bfseries},
title={wecs},
axis x line*=bottom,
axis y line*=left
]
\addplot [color=black, mark=x, mark options={solid, black}, forget plot]
  table[row sep=crcr]{%
10	0.0104601998749228\\
17	0.00247209833302417\\
28	0.00186834249467555\\
46	0.00157957263006422\\
77	0.00154354405044039\\
129	0.00148915062928706\\
215	0.00109751271493493\\
359	0.000964179950295484\\
599	0.000812398916284386\\
1000	0.000707399125631462\\
};
\addplot [color=mycolor1, mark=x, mark options={solid, mycolor1}, forget plot]
  table[row sep=crcr]{%
10	0.713863619110766\\
17	0.333774225215472\\
28	0.199688267416076\\
46	0.0686980337325759\\
77	0.0233939029863264\\
129	0.0085162335980186\\
215	0.00293130963955966\\
359	0.00137449655361388\\
599	0.000986024610325891\\
1000	0.00078466783878277\\
};
\addplot [color=mycolor2, mark=x, mark options={solid, mycolor2}, forget plot]
  table[row sep=crcr]{%
10	0.73921134574813\\
17	0.304389123707513\\
28	0.183719548793792\\
46	0.0558019133665416\\
77	0.0178000784182765\\
129	0.00704404636384116\\
215	0.00309209613382931\\
359	0.00157352405426957\\
599	0.00101645806522483\\
1000	0.000793727742662307\\
};
\addplot [color=blue!75!mycolor2, mark=x, mark options={solid, blue!75!mycolor2}, forget plot]
  table[row sep=crcr]{%
10	0.783779164503439\\
17	0.534328842056109\\
28	0.181712047679808\\
46	0.0741953379027907\\
77	0.02505034155555\\
129	0.00909231717978515\\
215	0.00370309696398597\\
359	0.00162750326993664\\
599	0.0011694183637708\\
1000	0.00089484172684881\\
};
\addplot [color=mycolor3, mark=x, mark options={solid, mycolor3}, forget plot]
  table[row sep=crcr]{%
10	0.792251104711349\\
17	0.610506503562019\\
28	0.264013747986078\\
46	0.0960969735558413\\
77	0.0296007558357707\\
129	0.0131596536101939\\
215	0.00496923498795625\\
359	0.00214752759779892\\
599	0.00142594468870718\\
1000	0.0010558540723523\\
};
\addplot [color=mycolor4, mark=x, mark options={solid, mycolor4}, forget plot]
  table[row sep=crcr]{%
10	0.917422985862155\\
17	0.833716642012463\\
28	0.335329678358198\\
46	0.121498652628665\\
77	0.0385077919859029\\
129	0.0152644114006022\\
215	0.00610725350480947\\
359	0.00289341958217275\\
599	0.00188496325776677\\
1000	0.00143152865217195\\
};
\addplot [color=green!60!mycolor4, mark=x, mark options={solid, green!60!mycolor4}, forget plot]
  table[row sep=crcr]{%
10	2.08678643559388\\
17	0.795736313083803\\
28	0.463376095923909\\
46	0.239163754155587\\
77	0.0980335264775869\\
129	0.037728543854723\\
215	0.0128250269291206\\
359	0.00582296239742523\\
599	0.0036048388287869\\
1000	0.00251513462780042\\
};
\addplot [color=mycolor5, mark=x, mark options={solid, mycolor5}, forget plot]
  table[row sep=crcr]{%
10	1.93078783787813\\
17	1.35930181693048\\
28	1.12115955942897\\
46	0.470547678924283\\
77	0.157484661822404\\
129	0.090807674758313\\
215	0.0477080043365499\\
359	0.0253975249736295\\
599	0.0119503348583657\\
1000	0.00674449637193438\\
};
\addplot [color=mycolor6, mark=x, mark options={solid, mycolor6}, forget plot]
  table[row sep=crcr]{%
10	1.0717946514449\\
17	1.29334698546561\\
28	1.62257716148998\\
46	1.36579318147847\\
77	0.749627851529588\\
129	0.376896531898572\\
215	0.288347197007742\\
359	0.197058536917907\\
599	0.123035422931177\\
1000	0.0714919579324502\\
};
\addplot [color=mycolor7, mark=x, mark options={solid, mycolor7}, forget plot]
  table[row sep=crcr]{%
10	0.981350365754882\\
17	0.97476247782857\\
28	0.96798017888807\\
46	0.958619295040331\\
77	0.947688558106078\\
129	0.937505095113218\\
215	0.944176824249409\\
359	0.949489833455849\\
599	0.511749516805813\\
1000	0.491782261383961\\
};
\addplot [color=red, mark=x, mark options={solid, red}, forget plot]
  table[row sep=crcr]{%
10	0.989132236330273\\
17	0.987862832876483\\
28	0.987806226282253\\
46	0.986801041601315\\
77	0.985811618345062\\
129	0.980425715376734\\
215	0.973598119060593\\
359	0.964752335728915\\
599	0.955338611950684\\
1000	0.937797178429518\\
};
\end{axis}
\end{tikzpicture}%
%
%
\definecolor{mycolor1}{rgb}{1.00000,0.00000,0.60000}%
\definecolor{mycolor2}{rgb}{0.80000,0.00000,1.00000}%
\definecolor{mycolor3}{rgb}{0.00000,0.40000,1.00000}%
\definecolor{mycolor4}{rgb}{0.00000,1.00000,1.00000}%
\definecolor{mycolor5}{rgb}{0.20000,1.00000,0.00000}%
\definecolor{mycolor6}{rgb}{0.80000,1.00000,0.00000}%
\definecolor{mycolor7}{rgb}{1.00000,0.60000,0.00000}%
\begin{tikzpicture}

\begin{axis}[%
width=0.951\fwidth,
height=0.75\fwidth,
at={(0\fwidth,0\fwidth)},
scale only axis,
xmode=log,
xmin=10,
xmax=1000,
xminorticks=true,
ymode=log,
ymin=0.01,
ymax=10,
yminorticks=true,
axis background/.style={fill=white},
title style={font=\bfseries},
title={$\text{mlr}\_\text{k}\text{nn}\_\text{r}\text{ng}$},
axis x line*=bottom,
axis y line*=left
]
\addplot [color=black, mark=x, mark options={solid, black}, forget plot]
  table[row sep=crcr]{%
10	0.585468093020371\\
17	0.34021743561354\\
28	0.398802334295848\\
46	0.24529204887694\\
77	0.31770920780425\\
129	0.125922740127205\\
215	0.169045831654712\\
359	0.0612268347668686\\
599	0.0400473522239765\\
1000	0.0261096119747559\\
};
\addplot [color=mycolor1, mark=x, mark options={solid, mycolor1}, forget plot]
  table[row sep=crcr]{%
10	2.99582967184645\\
17	2.69852886107284\\
28	1.73930877841018\\
46	0.907008334559497\\
77	0.772538424160243\\
129	0.629059891621948\\
215	0.606685667386263\\
359	0.506900612318557\\
599	0.346188583172771\\
1000	0.238302468215608\\
};
\addplot [color=mycolor2, mark=x, mark options={solid, mycolor2}, forget plot]
  table[row sep=crcr]{%
10	2.98650872801966\\
17	2.70902059999458\\
28	1.75851734156213\\
46	0.886974474382062\\
77	0.812667686462622\\
129	0.67124004746803\\
215	0.618807581203847\\
359	0.518812871406778\\
599	0.361426655890826\\
1000	0.228260722867154\\
};
\addplot [color=blue!75!mycolor2, mark=x, mark options={solid, blue!75!mycolor2}, forget plot]
  table[row sep=crcr]{%
10	2.96855512313593\\
17	2.25897088007034\\
28	1.919828406807\\
46	0.944213311497103\\
77	0.755676826468451\\
129	0.639808949009771\\
215	0.589859216004562\\
359	0.513055019323441\\
599	0.346444153649601\\
1000	0.227846396300345\\
};
\addplot [color=mycolor3, mark=x, mark options={solid, mycolor3}, forget plot]
  table[row sep=crcr]{%
10	2.86838921418268\\
17	2.5881401160132\\
28	1.49137715011724\\
46	0.97602816366021\\
77	0.761184968098831\\
129	0.694226832345362\\
215	0.620967814609658\\
359	0.539770186494163\\
599	0.349898301032488\\
1000	0.234654031057947\\
};
\addplot [color=mycolor4, mark=x, mark options={solid, mycolor4}, forget plot]
  table[row sep=crcr]{%
10	2.93722815081663\\
17	2.31447295585899\\
28	2.00458976934317\\
46	1.08162028689064\\
77	0.838945555820106\\
129	0.646883822199307\\
215	0.634415050752179\\
359	0.509258507025189\\
599	0.354278648739493\\
1000	0.237092369843518\\
};
\addplot [color=green!60!mycolor4, mark=x, mark options={solid, green!60!mycolor4}, forget plot]
  table[row sep=crcr]{%
10	2.63044161113441\\
17	3.36990316216252\\
28	1.88071612333828\\
46	1.31354405988534\\
77	0.870740602819736\\
129	0.653586763428093\\
215	0.547672240866754\\
359	0.5396599428757\\
599	0.362721732263433\\
1000	0.244519685981481\\
};
\addplot [color=mycolor5, mark=x, mark options={solid, mycolor5}, forget plot]
  table[row sep=crcr]{%
10	3.17289885433613\\
17	3.73971432817087\\
28	2.13914267188872\\
46	1.41238632720545\\
77	0.924487883819875\\
129	0.658779125494484\\
215	0.551313600392898\\
359	0.528737425946503\\
599	0.398703542006789\\
1000	0.264190760310094\\
};
\addplot [color=mycolor6, mark=x, mark options={solid, mycolor6}, forget plot]
  table[row sep=crcr]{%
10	3.15979348304782\\
17	3.4972534046795\\
28	2.89662663832168\\
46	2.09114902773306\\
77	1.51961117901155\\
129	1.08228962473332\\
215	0.6518073575393\\
359	0.449533257136468\\
599	0.391328101247291\\
1000	0.30064986864363\\
};
\addplot [color=mycolor7, mark=x, mark options={solid, mycolor7}, forget plot]
  table[row sep=crcr]{%
10	1.35348434530827\\
17	1.82635771372126\\
28	2.11636198234503\\
46	2.59351480507274\\
77	2.39826439063637\\
129	1.92299389714528\\
215	1.45545447383481\\
359	0.921634897507636\\
599	0.490714538143831\\
1000	0.258723271658992\\
};
\addplot [color=red, mark=x, mark options={solid, red}, forget plot]
  table[row sep=crcr]{%
10	0.94648221157044\\
17	0.921926507865797\\
28	0.884138100992697\\
46	0.870353525095685\\
77	0.834072215822443\\
129	0.841255820129254\\
215	0.878582874547198\\
359	0.889620487107806\\
599	0.864056105022313\\
1000	0.674867707050866\\
};
\end{axis}
\end{tikzpicture}%
%
%
\definecolor{mycolor1}{rgb}{1.00000,0.00000,0.60000}%
\definecolor{mycolor2}{rgb}{0.80000,0.00000,1.00000}%
\definecolor{mycolor3}{rgb}{0.00000,0.40000,1.00000}%
\definecolor{mycolor4}{rgb}{0.00000,1.00000,1.00000}%
\definecolor{mycolor5}{rgb}{0.20000,1.00000,0.00000}%
\definecolor{mycolor6}{rgb}{0.80000,1.00000,0.00000}%
\definecolor{mycolor7}{rgb}{1.00000,0.60000,0.00000}%
\begin{tikzpicture}

\begin{axis}[%
width=0.951\fwidth,
height=0.75\fwidth,
at={(0\fwidth,0\fwidth)},
scale only axis,
xmode=log,
xmin=10,
xmax=1000,
xminorticks=true,
ymode=log,
ymin=0.0502711840190287,
ymax=1,
yminorticks=true,
axis background/.style={fill=white},
title style={font=\bfseries},
title={fried},
axis x line*=bottom,
axis y line*=left,
legend style={at={(0.03,0.03)}, anchor=south west, legend cell align=left, align=left, draw=white!15!black}
]
\addplot [color=black, mark=x, mark options={solid, black}]
  table[row sep=crcr]{%
10	0.123256999098451\\
17	0.1026343870421\\
28	0.0785339079340689\\
46	0.0655046779660202\\
77	0.0636372541380158\\
129	0.0581071677106308\\
215	0.056896271988317\\
359	0.0552687837361821\\
599	0.0528187918189329\\
1000	0.0510442682418493\\
};
\addlegendentry{2L}

\addplot [color=mycolor1, mark=x, mark options={solid, mycolor1}]
  table[row sep=crcr]{%
10	0.459851231160245\\
17	0.401189804495147\\
28	0.289590432526174\\
46	0.192222523678153\\
77	0.121554391234783\\
129	0.0977013281997719\\
215	0.0780151479092142\\
359	0.0626879377520444\\
599	0.0545119116818215\\
1000	0.0505436502899783\\
};
\addlegendentry{0.05}

\addplot [color=mycolor2, mark=x, mark options={solid, mycolor2}]
  table[row sep=crcr]{%
10	0.53243277694105\\
17	0.388790274412241\\
28	0.279421632660551\\
46	0.204670488655427\\
77	0.135191568114614\\
129	0.100088864859164\\
215	0.0780652092155441\\
359	0.0620205969392287\\
599	0.0550082875926627\\
1000	0.0507847622102568\\
};
\addlegendentry{0.09}

\addplot [color=blue!75!mycolor2, mark=x, mark options={solid, blue!75!mycolor2}]
  table[row sep=crcr]{%
10	0.507373804870481\\
17	0.433825362972868\\
28	0.274841849962106\\
46	0.203360261629846\\
77	0.128745698127938\\
129	0.106581666339669\\
215	0.0772061483796658\\
359	0.0632533158454169\\
599	0.054530707680676\\
1000	0.0502711840190287\\
};
\addlegendentry{0.162}

\addplot [color=mycolor3, mark=x, mark options={solid, mycolor3}]
  table[row sep=crcr]{%
10	0.479334782744159\\
17	0.350763302184583\\
28	0.238593630294995\\
46	0.194912886084757\\
77	0.143748361285626\\
129	0.0977950417888985\\
215	0.0790472315221189\\
359	0.0620748828726598\\
599	0.0554280557707513\\
1000	0.0510035220602206\\
};
\addlegendentry{0.292}

\addplot [color=mycolor4, mark=x, mark options={solid, mycolor4}]
  table[row sep=crcr]{%
10	0.491423242136123\\
17	0.342705888866183\\
28	0.2392832385623\\
46	0.198120929538177\\
77	0.137114038652655\\
129	0.114054608397574\\
215	0.0812119527182585\\
359	0.0637600774980381\\
599	0.0561299341111318\\
1000	0.0514703851640659\\
};
\addlegendentry{0.527}

\addplot [color=green!60!mycolor4, mark=x, mark options={solid, green!60!mycolor4}]
  table[row sep=crcr]{%
10	0.414215741102707\\
17	0.34469638515273\\
28	0.303616204198316\\
46	0.235461866810236\\
77	0.137437019756915\\
129	0.109127900758578\\
215	0.0874407868712443\\
359	0.0680531913719911\\
599	0.0585128349762015\\
1000	0.0523562558486871\\
};
\addlegendentry{0.949}

\addplot [color=mycolor5, mark=x, mark options={solid, mycolor5}]
  table[row sep=crcr]{%
10	0.443746745159593\\
17	0.366240613532769\\
28	0.327917995559079\\
46	0.283330760169934\\
77	0.212331718396431\\
129	0.142498753677317\\
215	0.102794423334714\\
359	0.0815041096773523\\
599	0.0660007436835723\\
1000	0.0574070188712727\\
};
\addlegendentry{1.71}

\addplot [color=mycolor6, mark=x, mark options={solid, mycolor6}]
  table[row sep=crcr]{%
10	0.689180368151724\\
17	0.580905978702494\\
28	0.473187723271218\\
46	0.384469724348498\\
77	0.310802871324883\\
129	0.252420162922629\\
215	0.192624594788015\\
359	0.136694103663319\\
599	0.102116502876677\\
1000	0.0799861288456674\\
};
\addlegendentry{3.081}

\addplot [color=mycolor7, mark=x, mark options={solid, mycolor7}]
  table[row sep=crcr]{%
10	0.93855117312127\\
17	0.917932772082348\\
28	0.878162178316414\\
46	0.831779885853382\\
77	0.757159415565924\\
129	0.668622696242856\\
215	0.563535733553504\\
359	0.450953529527939\\
599	0.341604048396497\\
1000	0.259179925126221\\
};
\addlegendentry{5.55}

\addplot [color=red, mark=x, mark options={solid, red}]
  table[row sep=crcr]{%
10	0.977535894885656\\
17	0.976729327320884\\
28	0.975334960161446\\
46	0.972853940588521\\
77	0.968460895583259\\
129	0.961882422249853\\
215	0.950786595979683\\
359	0.934390105427365\\
599	0.908926215252056\\
1000	0.873329114586907\\
};
\addlegendentry{10}

\end{axis}
\end{tikzpicture}%
%
%
\definecolor{mycolor1}{rgb}{1.00000,0.00000,0.60000}%
\definecolor{mycolor2}{rgb}{0.80000,0.00000,1.00000}%
\definecolor{mycolor3}{rgb}{0.00000,0.40000,1.00000}%
\definecolor{mycolor4}{rgb}{0.00000,1.00000,1.00000}%
\definecolor{mycolor5}{rgb}{0.20000,1.00000,0.00000}%
\definecolor{mycolor6}{rgb}{0.80000,1.00000,0.00000}%
\definecolor{mycolor7}{rgb}{1.00000,0.60000,0.00000}%
\begin{tikzpicture}

\begin{axis}[%
width=0.951\fwidth,
height=0.75\fwidth,
at={(0\fwidth,0\fwidth)},
scale only axis,
xmode=log,
xmin=10,
xmax=1000,
xminorticks=true,
ymode=log,
ymin=0.1,
ymax=11.6963507381542,
yminorticks=true,
axis background/.style={fill=white},
title style={font=\bfseries},
title={$\text{kegg}\_\text{u}\text{ndir}\_\text{u}\text{ci}$},
axis x line*=bottom,
axis y line*=left
]
\addplot [color=black, mark=x, mark options={solid, black}, forget plot]
  table[row sep=crcr]{%
10	1.09740925439129\\
17	0.37576253072719\\
28	0.386738322571205\\
46	0.445653282259259\\
77	0.24059856378753\\
129	0.228456881072254\\
215	0.232429079335593\\
359	0.17773042698704\\
599	0.223297636483726\\
1000	0.10988181794422\\
};
\addplot [color=mycolor1, mark=x, mark options={solid, mycolor1}, forget plot]
  table[row sep=crcr]{%
10	1.69154367939659\\
17	1.469218732947\\
28	2.2748664618939\\
46	0.826828283256273\\
77	0.622353177553582\\
129	1.13493877323086\\
215	0.352026761212496\\
359	0.277578266607443\\
599	0.177383355232517\\
1000	0.107311088017861\\
};
\addplot [color=mycolor2, mark=x, mark options={solid, mycolor2}, forget plot]
  table[row sep=crcr]{%
10	1.65932567692637\\
17	1.47239150857317\\
28	2.31728840709532\\
46	0.761891427127695\\
77	0.645335533034608\\
129	0.883160159439012\\
215	0.321607332030641\\
359	0.251083050967246\\
599	0.175423468557221\\
1000	0.127649765023752\\
};
\addplot [color=blue!75!mycolor2, mark=x, mark options={solid, blue!75!mycolor2}, forget plot]
  table[row sep=crcr]{%
10	2.78311692956291\\
17	4.73828286450209\\
28	1.17532253492861\\
46	1.11950185353689\\
77	0.532015229463012\\
129	0.472347097971256\\
215	0.315597190524556\\
359	0.323297180533752\\
599	0.157176796277853\\
1000	0.11638931373738\\
};
\addplot [color=mycolor3, mark=x, mark options={solid, mycolor3}, forget plot]
  table[row sep=crcr]{%
10	11.6963507381542\\
17	5.95272488922334\\
28	1.91970789154868\\
46	0.953116970395082\\
77	0.80962047253072\\
129	0.699827740742511\\
215	0.630357548211277\\
359	0.234372431000693\\
599	0.21570909382165\\
1000	0.100959437847085\\
};
\addplot [color=mycolor4, mark=x, mark options={solid, mycolor4}, forget plot]
  table[row sep=crcr]{%
10	2.08462021196314\\
17	6.61838829392211\\
28	6.39489817687023\\
46	1.62398893076642\\
77	0.825740874771883\\
129	0.915899987614043\\
215	0.35746579939562\\
359	0.282698522710781\\
599	0.207892866264634\\
1000	0.146215152462534\\
};
\addplot [color=green!60!mycolor4, mark=x, mark options={solid, green!60!mycolor4}, forget plot]
  table[row sep=crcr]{%
10	6.64583972615359\\
17	2.22808651956753\\
28	4.09804951332721\\
46	2.27034223044446\\
77	0.676645440585398\\
129	0.606295032098996\\
215	0.422280299887453\\
359	0.329778635319303\\
599	0.174720278236496\\
1000	0.140502704070484\\
};
\addplot [color=mycolor5, mark=x, mark options={solid, mycolor5}, forget plot]
  table[row sep=crcr]{%
10	2.1156354499797\\
17	2.89917432213915\\
28	2.60163039969399\\
46	2.86478523400473\\
77	1.96723028257364\\
129	0.919557574416848\\
215	0.95768972074625\\
359	0.293126603792016\\
599	0.286892120924447\\
1000	0.16508351922976\\
};
\addplot [color=mycolor6, mark=x, mark options={solid, mycolor6}, forget plot]
  table[row sep=crcr]{%
10	1.18189346923217\\
17	1.29049833034776\\
28	1.45356921467568\\
46	1.4026787135195\\
77	1.16914166320067\\
129	1.07936611319671\\
215	1.14669508992306\\
359	0.755367432293295\\
599	0.604527349237838\\
1000	0.275669413629641\\
};
\addplot [color=mycolor7, mark=x, mark options={solid, mycolor7}, forget plot]
  table[row sep=crcr]{%
10	0.755873386293018\\
17	0.720085312190134\\
28	0.780220464709385\\
46	0.773583198405047\\
77	0.711292481348803\\
129	0.739426809918998\\
215	0.732405253373259\\
359	0.652329886587669\\
599	0.583815806310321\\
1000	0.370841166884805\\
};
\addplot [color=red, mark=x, mark options={solid, red}, forget plot]
  table[row sep=crcr]{%
10	0.83820249888641\\
17	0.775610895575268\\
28	0.694805711201392\\
46	0.617352471644866\\
77	0.55769110686224\\
129	0.495318452281601\\
215	0.455575677058383\\
359	0.416717843712617\\
599	0.380103559934244\\
1000	0.340142190333747\\
};
\end{axis}
\end{tikzpicture}%
\else
\fi
\caption{Visualization of the test MSE ($y$-axis) over the number of greedily selected centers ($x$-axis) for the first six datasets. 
The black line shows the 2L-VKOGA, while the colored lines show the use of standard kernels with different length scale parameters $\varepsilon$.}
\label{fig:ml_datasets_results_1}
\end{figure}

\begin{figure}[h]
\centering
\setlength\fwidth{.4\textwidth}
\ifplots
%
%
\definecolor{mycolor1}{rgb}{1.00000,0.00000,0.60000}%
\definecolor{mycolor2}{rgb}{0.80000,0.00000,1.00000}%
\definecolor{mycolor3}{rgb}{0.00000,0.40000,1.00000}%
\definecolor{mycolor4}{rgb}{0.00000,1.00000,1.00000}%
\definecolor{mycolor5}{rgb}{0.20000,1.00000,0.00000}%
\definecolor{mycolor6}{rgb}{0.80000,1.00000,0.00000}%
\definecolor{mycolor7}{rgb}{1.00000,0.60000,0.00000}%
\begin{tikzpicture}

\begin{axis}[%
width=0.951\fwidth,
height=0.75\fwidth,
at={(0\fwidth,0\fwidth)},
scale only axis,
xmode=log,
xmin=10,
xmax=1000,
xminorticks=true,
ymode=log,
ymin=0.0450116305727643,
ymax=10.7118795589394,
yminorticks=true,
axis background/.style={fill=white},
title style={font=\bfseries},
title={$\text{online}\_\text{v}\text{ideo}$},
axis x line*=bottom,
axis y line*=left
]
\addplot [color=black, mark=x, mark options={solid, black}, forget plot]
  table[row sep=crcr]{%
10	0.176413125454326\\
17	0.189937656137508\\
28	0.386537685072048\\
46	0.211208986009364\\
77	0.300229575141678\\
129	0.457951428806252\\
215	0.295918525061362\\
359	0.374452538620403\\
599	0.205199742920705\\
1000	0.0983270956887856\\
};
\addplot [color=mycolor1, mark=x, mark options={solid, mycolor1}, forget plot]
  table[row sep=crcr]{%
10	6.15402703238756\\
17	3.05365355038645\\
28	2.33413265029603\\
46	1.65481341114079\\
77	1.06368284256166\\
129	0.674046275638423\\
215	0.338942729913826\\
359	0.14316694470002\\
599	0.0870930102368362\\
1000	0.0458530426915516\\
};
\addplot [color=mycolor2, mark=x, mark options={solid, mycolor2}, forget plot]
  table[row sep=crcr]{%
10	1.42450256790559\\
17	3.23764743651072\\
28	2.34913012903541\\
46	2.07863018091569\\
77	1.32503422276776\\
129	0.475524523407528\\
215	0.353285747424675\\
359	0.156945781480944\\
599	0.0926862608324339\\
1000	0.0450116305727643\\
};
\addplot [color=blue!75!mycolor2, mark=x, mark options={solid, blue!75!mycolor2}, forget plot]
  table[row sep=crcr]{%
10	8.19732584003393\\
17	2.61633610095293\\
28	1.82859082201734\\
46	1.67702394973176\\
77	1.3321833086957\\
129	0.570711570828431\\
215	0.295619646365536\\
359	0.150615489851029\\
599	0.0892107366678652\\
1000	0.0488390198287914\\
};
\addplot [color=mycolor3, mark=x, mark options={solid, mycolor3}, forget plot]
  table[row sep=crcr]{%
10	2.89798561272257\\
17	1.98355898045555\\
28	1.09742264393288\\
46	1.79681422832659\\
77	1.04493363031087\\
129	0.565538158618477\\
215	0.308149806715524\\
359	0.146442946665273\\
599	0.095583659250617\\
1000	0.0482084451241182\\
};
\addplot [color=mycolor4, mark=x, mark options={solid, mycolor4}, forget plot]
  table[row sep=crcr]{%
10	8.77728061277286\\
17	3.7949845093264\\
28	1.61152379497006\\
46	1.40604998307207\\
77	0.866058309345852\\
129	0.475632188139658\\
215	0.335144658639348\\
359	0.140715286833052\\
599	0.0994630083238966\\
1000	0.0468337035775679\\
};
\addplot [color=green!60!mycolor4, mark=x, mark options={solid, green!60!mycolor4}, forget plot]
  table[row sep=crcr]{%
10	6.4737055032042\\
17	2.83904589875906\\
28	1.23359587320565\\
46	0.787413724045028\\
77	1.00959702948746\\
129	0.633197525294283\\
215	0.450166180943131\\
359	0.168327981522176\\
599	0.0951329606928533\\
1000	0.0500374730771615\\
};
\addplot [color=mycolor5, mark=x, mark options={solid, mycolor5}, forget plot]
  table[row sep=crcr]{%
10	10.7118795589394\\
17	5.57334259422025\\
28	3.08312965339445\\
46	1.42994238754747\\
77	0.335842019172929\\
129	0.311179524495181\\
215	0.253127095702642\\
359	0.183990596174999\\
599	0.0978053205549424\\
1000	0.054405167905977\\
};
\addplot [color=mycolor6, mark=x, mark options={solid, mycolor6}, forget plot]
  table[row sep=crcr]{%
10	1.71829195250594\\
17	2.82467497297519\\
28	2.65079593720346\\
46	2.03136117891128\\
77	1.44466746272182\\
129	0.889909870227541\\
215	0.420503279489255\\
359	0.174946524930757\\
599	0.0802025537894473\\
1000	0.0536562573885547\\
};
\addplot [color=mycolor7, mark=x, mark options={solid, mycolor7}, forget plot]
  table[row sep=crcr]{%
10	0.670495367704181\\
17	0.612054610989302\\
28	0.54188848303903\\
46	0.508753462709656\\
77	0.526164123505443\\
129	0.510038361211008\\
215	0.509924367698627\\
359	0.457189552897346\\
599	0.349497936466217\\
1000	0.251134027515205\\
};
\addplot [color=red, mark=x, mark options={solid, red}, forget plot]
  table[row sep=crcr]{%
10	0.895903417283368\\
17	0.871192999311085\\
28	0.820537528608872\\
46	0.741651145136117\\
77	0.668073319103475\\
129	0.566813308969948\\
215	0.453861782983667\\
359	0.358544080690355\\
599	0.287588736561568\\
1000	0.251571769995054\\
};
\end{axis}
\end{tikzpicture}%
%
%
\definecolor{mycolor1}{rgb}{1.00000,0.00000,0.60000}%
\definecolor{mycolor2}{rgb}{0.80000,0.00000,1.00000}%
\definecolor{mycolor3}{rgb}{0.00000,0.40000,1.00000}%
\definecolor{mycolor4}{rgb}{0.00000,1.00000,1.00000}%
\definecolor{mycolor5}{rgb}{0.20000,1.00000,0.00000}%
\definecolor{mycolor6}{rgb}{0.80000,1.00000,0.00000}%
\definecolor{mycolor7}{rgb}{1.00000,0.60000,0.00000}%
\begin{tikzpicture}

\begin{axis}[%
width=0.951\fwidth,
height=0.75\fwidth,
at={(0\fwidth,0\fwidth)},
scale only axis,
xmode=log,
xmin=10,
xmax=1000,
xminorticks=true,
ymode=log,
ymin=0.0367835213144311,
ymax=2.21833084822089,
yminorticks=true,
axis background/.style={fill=white},
title style={font=\bfseries},
title={diamonds},
axis x line*=bottom,
axis y line*=left
]
\addplot [color=black, mark=x, mark options={solid, black}, forget plot]
  table[row sep=crcr]{%
10	0.240573370125126\\
17	0.12200713858717\\
28	0.0855254072098582\\
46	0.082740295848803\\
77	0.0799559315541836\\
129	0.0719671818550044\\
215	0.0762213016290421\\
359	0.158225135396179\\
599	0.0367835213144311\\
1000	0.0407064871239236\\
};
\addplot [color=mycolor1, mark=x, mark options={solid, mycolor1}, forget plot]
  table[row sep=crcr]{%
10	1.43351451381905\\
17	0.571198663710978\\
28	0.29082662160912\\
46	0.164623085360588\\
77	0.132400787906072\\
129	0.10895052022852\\
215	0.136313599111721\\
359	0.104780657480614\\
599	0.063190597520119\\
1000	0.0454262990181216\\
};
\addplot [color=mycolor2, mark=x, mark options={solid, mycolor2}, forget plot]
  table[row sep=crcr]{%
10	0.802324483690249\\
17	0.559855207541786\\
28	0.361236986017587\\
46	0.13704334843603\\
77	0.110601924840667\\
129	0.115964426564007\\
215	0.0980596436177043\\
359	0.079803034212733\\
599	0.0598531617172287\\
1000	0.0461235039913569\\
};
\addplot [color=blue!75!mycolor2, mark=x, mark options={solid, blue!75!mycolor2}, forget plot]
  table[row sep=crcr]{%
10	0.982876935563361\\
17	0.493497528349938\\
28	0.261589888743255\\
46	0.129006901307168\\
77	0.120962370720412\\
129	0.116680539822402\\
215	0.121255645629766\\
359	0.0841377996373424\\
599	0.0637679831611155\\
1000	0.0438760376219014\\
};
\addplot [color=mycolor3, mark=x, mark options={solid, mycolor3}, forget plot]
  table[row sep=crcr]{%
10	0.934586264394735\\
17	0.591772358293344\\
28	0.338753015616303\\
46	0.106812827724822\\
77	0.105450142300742\\
129	0.118891422698744\\
215	0.103432784419832\\
359	0.0808372929247578\\
599	0.0647145260801564\\
1000	0.0465863498449148\\
};
\addplot [color=mycolor4, mark=x, mark options={solid, mycolor4}, forget plot]
  table[row sep=crcr]{%
10	1.06635977984382\\
17	0.708868833973856\\
28	0.495936018134841\\
46	0.294916802145839\\
77	0.109012538670174\\
129	0.121808498742244\\
215	0.103221836483878\\
359	0.0692872968217196\\
599	0.0678198495762568\\
1000	0.0403809205569338\\
};
\addplot [color=green!60!mycolor4, mark=x, mark options={solid, green!60!mycolor4}, forget plot]
  table[row sep=crcr]{%
10	1.6733879282915\\
17	0.891490514238813\\
28	0.617485041388555\\
46	0.288061566547997\\
77	0.180531024070435\\
129	0.120355268852493\\
215	0.100521977579266\\
359	0.0790121921045466\\
599	0.0565223863928407\\
1000	0.0385233777656855\\
};
\addplot [color=mycolor5, mark=x, mark options={solid, mycolor5}, forget plot]
  table[row sep=crcr]{%
10	2.21833084822089\\
17	1.84033869353249\\
28	1.14295506261073\\
46	1.00542803566877\\
77	0.525194039063972\\
129	0.229850173227322\\
215	0.131026635364418\\
359	0.108375298529188\\
599	0.0702411291377215\\
1000	0.0482899112289138\\
};
\addplot [color=mycolor6, mark=x, mark options={solid, mycolor6}, forget plot]
  table[row sep=crcr]{%
10	0.905582656145723\\
17	0.981028934223239\\
28	1.26338653885676\\
46	1.78841041747551\\
77	1.32080118389534\\
129	0.783858397347821\\
215	0.489162915268373\\
359	0.281737917064452\\
599	0.155678171865414\\
1000	0.0873353377662929\\
};
\addplot [color=mycolor7, mark=x, mark options={solid, mycolor7}, forget plot]
  table[row sep=crcr]{%
10	0.931706800100816\\
17	0.888950478524699\\
28	0.848384598976185\\
46	0.772622728326995\\
77	0.683898521447348\\
129	0.603291681027421\\
215	0.592799334258232\\
359	0.661340415677796\\
599	0.545773187412953\\
1000	0.310010914475444\\
};
\addplot [color=red, mark=x, mark options={solid, red}, forget plot]
  table[row sep=crcr]{%
10	0.981566195956742\\
17	0.965553153975901\\
28	0.954714274545941\\
46	0.931089294122619\\
77	0.885695545042823\\
129	0.836703438558452\\
215	0.774735122717196\\
359	0.682320830227741\\
599	0.564723494445937\\
1000	0.464003599428947\\
};
\end{axis}
\end{tikzpicture}%
%
%
\definecolor{mycolor1}{rgb}{1.00000,0.00000,0.60000}%
\definecolor{mycolor2}{rgb}{0.80000,0.00000,1.00000}%
\definecolor{mycolor3}{rgb}{0.00000,0.40000,1.00000}%
\definecolor{mycolor4}{rgb}{0.00000,1.00000,1.00000}%
\definecolor{mycolor5}{rgb}{0.20000,1.00000,0.00000}%
\definecolor{mycolor6}{rgb}{0.80000,1.00000,0.00000}%
\definecolor{mycolor7}{rgb}{1.00000,0.60000,0.00000}%
\begin{tikzpicture}

\begin{axis}[%
width=0.951\fwidth,
height=0.75\fwidth,
at={(0\fwidth,0\fwidth)},
scale only axis,
xmode=log,
xmin=10,
xmax=1000,
xminorticks=true,
ymode=log,
ymin=0.30441056462173,
ymax=0.993015763611574,
yminorticks=true,
axis background/.style={fill=white},
title style={font=\bfseries},
title={stock},
axis x line*=bottom,
axis y line*=left
]
\addplot [color=black, mark=x, mark options={solid, black}, forget plot]
  table[row sep=crcr]{%
10	0.71140177613704\\
17	0.878751192367983\\
28	0.993015763611574\\
46	0.746367286044738\\
77	0.787298320480422\\
129	0.569453780397645\\
215	0.487678917657691\\
359	0.374931702456788\\
599	0.360400032024093\\
1000	0.30644050382806\\
};
\addplot [color=mycolor1, mark=x, mark options={solid, mycolor1}, forget plot]
  table[row sep=crcr]{%
10	0.793662713331918\\
17	0.596775602743769\\
28	0.618829365434077\\
46	0.607310362920604\\
77	0.634993445475134\\
129	0.607702662297172\\
215	0.4731405464957\\
359	0.410228716945137\\
599	0.359252837376395\\
1000	0.315276681214166\\
};
\addplot [color=mycolor2, mark=x, mark options={solid, mycolor2}, forget plot]
  table[row sep=crcr]{%
10	0.749168640325551\\
17	0.70445673726983\\
28	0.674966936501188\\
46	0.58713416802542\\
77	0.516036109078575\\
129	0.486140901685594\\
215	0.461129442402471\\
359	0.405145980892752\\
599	0.359171533908265\\
1000	0.315598313328545\\
};
\addplot [color=blue!75!mycolor2, mark=x, mark options={solid, blue!75!mycolor2}, forget plot]
  table[row sep=crcr]{%
10	0.740499504364895\\
17	0.478670325600025\\
28	0.541708861852376\\
46	0.570307378133853\\
77	0.658042492134326\\
129	0.522791193216951\\
215	0.432382227775517\\
359	0.370775777652826\\
599	0.336149835793984\\
1000	0.311380329105747\\
};
\addplot [color=mycolor3, mark=x, mark options={solid, mycolor3}, forget plot]
  table[row sep=crcr]{%
10	0.740452245211808\\
17	0.48265622938729\\
28	0.671123176449563\\
46	0.519265876364698\\
77	0.565795506522175\\
129	0.484913171828662\\
215	0.484746096472925\\
359	0.382360977924056\\
599	0.33874413221654\\
1000	0.312479349125787\\
};
\addplot [color=mycolor4, mark=x, mark options={solid, mycolor4}, forget plot]
  table[row sep=crcr]{%
10	0.620316634925227\\
17	0.608504492418745\\
28	0.64018460464611\\
46	0.674599897300886\\
77	0.552412179046407\\
129	0.535462858045749\\
215	0.459510197454519\\
359	0.386729349104887\\
599	0.338831686065455\\
1000	0.30441056462173\\
};
\addplot [color=green!60!mycolor4, mark=x, mark options={solid, green!60!mycolor4}, forget plot]
  table[row sep=crcr]{%
10	0.68467066289771\\
17	0.733252745634975\\
28	0.584847412883017\\
46	0.596504381514368\\
77	0.568903220835936\\
129	0.488587516493889\\
215	0.436881408192994\\
359	0.387422334508024\\
599	0.339925778490391\\
1000	0.317741037366741\\
};
\addplot [color=mycolor5, mark=x, mark options={solid, mycolor5}, forget plot]
  table[row sep=crcr]{%
10	0.616981524597844\\
17	0.687149374393312\\
28	0.596401664694872\\
46	0.534538397705547\\
77	0.581286347769694\\
129	0.522159517392671\\
215	0.468107720991698\\
359	0.408805030839278\\
599	0.347450002543178\\
1000	0.306577236767886\\
};
\addplot [color=mycolor6, mark=x, mark options={solid, mycolor6}, forget plot]
  table[row sep=crcr]{%
10	0.785707574116528\\
17	0.705873265663493\\
28	0.623856534825795\\
46	0.548744116607408\\
77	0.533212012152653\\
129	0.532558302235261\\
215	0.471121501978407\\
359	0.396688461956252\\
599	0.352598474006852\\
1000	0.31717106483846\\
};
\addplot [color=mycolor7, mark=x, mark options={solid, mycolor7}, forget plot]
  table[row sep=crcr]{%
10	0.868720357153526\\
17	0.814338573180101\\
28	0.788878770253226\\
46	0.698910384082192\\
77	0.67528647759038\\
129	0.599304366904591\\
215	0.551429647077603\\
359	0.502045613095294\\
599	0.435625802053078\\
1000	0.369451152201916\\
};
\addplot [color=red, mark=x, mark options={solid, red}, forget plot]
  table[row sep=crcr]{%
10	0.981196404119784\\
17	0.975625880132653\\
28	0.960215225234792\\
46	0.924824117657844\\
77	0.901319591944016\\
129	0.868497348323627\\
215	0.81679228637978\\
359	0.756944260166649\\
599	0.694156661838619\\
1000	0.612733106732451\\
};
\end{axis}
\end{tikzpicture}%
%
%
\definecolor{mycolor1}{rgb}{1.00000,0.00000,0.60000}%
\definecolor{mycolor2}{rgb}{0.80000,0.00000,1.00000}%
\definecolor{mycolor3}{rgb}{0.00000,0.40000,1.00000}%
\definecolor{mycolor4}{rgb}{0.00000,1.00000,1.00000}%
\definecolor{mycolor5}{rgb}{0.20000,1.00000,0.00000}%
\definecolor{mycolor6}{rgb}{0.80000,1.00000,0.00000}%
\definecolor{mycolor7}{rgb}{1.00000,0.60000,0.00000}%
\begin{tikzpicture}

\begin{axis}[%
width=0.951\fwidth,
height=0.75\fwidth,
at={(0\fwidth,0\fwidth)},
scale only axis,
xmode=log,
xmin=10,
xmax=1000,
xminorticks=true,
ymode=log,
ymin=0.01,
ymax=1,
yminorticks=true,
axis background/.style={fill=white},
title style={font=\bfseries},
title={sarcos},
axis x line*=bottom,
axis y line*=left
]
\addplot [color=black, mark=x, mark options={solid, black}, forget plot]
  table[row sep=crcr]{%
10	0.26828652741768\\
17	0.233701710941026\\
28	0.113079922244814\\
46	0.122216390679377\\
77	0.0857903381128407\\
129	0.0764583381294211\\
215	0.065445417290986\\
359	0.0531955009958917\\
599	0.0425786255872009\\
1000	0.0346904484442486\\
};
\addplot [color=mycolor1, mark=x, mark options={solid, mycolor1}, forget plot]
  table[row sep=crcr]{%
10	0.358227459777781\\
17	0.193155817630969\\
28	0.160227006999473\\
46	0.110565745485305\\
77	0.0673878706311983\\
129	0.0598919409629234\\
215	0.0529002751104808\\
359	0.0362945952243126\\
599	0.0315655773705287\\
1000	0.0233400407032891\\
};
\addplot [color=mycolor2, mark=x, mark options={solid, mycolor2}, forget plot]
  table[row sep=crcr]{%
10	0.342266831661334\\
17	0.225725954514857\\
28	0.163667118083908\\
46	0.120634174469073\\
77	0.0805215482426015\\
129	0.0665724459732887\\
215	0.0503958203626872\\
359	0.041389261557818\\
599	0.0310105542298956\\
1000	0.0230772307187235\\
};
\addplot [color=blue!75!mycolor2, mark=x, mark options={solid, blue!75!mycolor2}, forget plot]
  table[row sep=crcr]{%
10	0.318200374041252\\
17	0.183011366455074\\
28	0.164317563655086\\
46	0.11102226379261\\
77	0.0674013283566983\\
129	0.0688893249817086\\
215	0.0498219583667738\\
359	0.0386202608984622\\
599	0.0301251767980139\\
1000	0.0232134289325531\\
};
\addplot [color=mycolor3, mark=x, mark options={solid, mycolor3}, forget plot]
  table[row sep=crcr]{%
10	0.379935440524501\\
17	0.19321602772212\\
28	0.199593454116217\\
46	0.126650662099369\\
77	0.09280885889805\\
129	0.0639569166660652\\
215	0.0494976701467517\\
359	0.040563150653321\\
599	0.0306760284327617\\
1000	0.0232481926552267\\
};
\addplot [color=mycolor4, mark=x, mark options={solid, mycolor4}, forget plot]
  table[row sep=crcr]{%
10	0.527713354377064\\
17	0.262045777898058\\
28	0.152685222694342\\
46	0.14395207922977\\
77	0.0834899307202362\\
129	0.0599297007676577\\
215	0.0534045944489278\\
359	0.0413987643924726\\
599	0.029477005411156\\
1000	0.023888049950378\\
};
\addplot [color=green!60!mycolor4, mark=x, mark options={solid, green!60!mycolor4}, forget plot]
  table[row sep=crcr]{%
10	0.427216669955222\\
17	0.331181466874258\\
28	0.196107825934088\\
46	0.132885907080492\\
77	0.103542049345023\\
129	0.0761671421995425\\
215	0.0522554042273126\\
359	0.0430926862424703\\
599	0.0317379331259131\\
1000	0.0240525112106086\\
};
\addplot [color=mycolor5, mark=x, mark options={solid, mycolor5}, forget plot]
  table[row sep=crcr]{%
10	0.671858329936798\\
17	0.382437965701036\\
28	0.317873048831573\\
46	0.188325674130999\\
77	0.136907746571497\\
129	0.0870615993360098\\
215	0.0634506493613974\\
359	0.0448148894676304\\
599	0.0350951951430085\\
1000	0.026291425958144\\
};
\addplot [color=mycolor6, mark=x, mark options={solid, mycolor6}, forget plot]
  table[row sep=crcr]{%
10	0.697255864750033\\
17	0.58526841095426\\
28	0.508364619275951\\
46	0.435737807344042\\
77	0.293679019019978\\
129	0.185375916744006\\
215	0.124274605882666\\
359	0.0748611341271503\\
599	0.0533609873082292\\
1000	0.0366767356268157\\
};
\addplot [color=mycolor7, mark=x, mark options={solid, mycolor7}, forget plot]
  table[row sep=crcr]{%
10	0.904392669075579\\
17	0.855643604688827\\
28	0.804109406223936\\
46	0.746295182916051\\
77	0.668102754535315\\
129	0.563130652382337\\
215	0.457136014434669\\
359	0.319432730545432\\
599	0.21900043448853\\
1000	0.127730586386118\\
};
\addplot [color=red, mark=x, mark options={solid, red}, forget plot]
  table[row sep=crcr]{%
10	0.976257566799321\\
17	0.966566495875892\\
28	0.953641847357894\\
46	0.937725348474341\\
77	0.915483547923745\\
129	0.877998896469234\\
215	0.840070276893802\\
359	0.773367865095099\\
599	0.700990808361421\\
1000	0.602856632611724\\
};
\end{axis}
\end{tikzpicture}%
%
%
\definecolor{mycolor1}{rgb}{1.00000,0.00000,0.60000}%
\definecolor{mycolor2}{rgb}{0.80000,0.00000,1.00000}%
\definecolor{mycolor3}{rgb}{0.00000,0.40000,1.00000}%
\definecolor{mycolor4}{rgb}{0.00000,1.00000,1.00000}%
\definecolor{mycolor5}{rgb}{0.20000,1.00000,0.00000}%
\definecolor{mycolor6}{rgb}{0.80000,1.00000,0.00000}%
\definecolor{mycolor7}{rgb}{1.00000,0.60000,0.00000}%
\begin{tikzpicture}

\begin{axis}[%
width=0.951\fwidth,
height=0.75\fwidth,
at={(0\fwidth,0\fwidth)},
scale only axis,
xmode=log,
xmin=10,
xmax=1000,
xminorticks=true,
ymode=log,
ymin=0.001,
ymax=1,
yminorticks=true,
axis background/.style={fill=white},
title style={font=\bfseries},
title={$\text{query}\_\text{a}\text{gg}\_\text{c}\text{ount}$},
axis x line*=bottom,
axis y line*=left
]
\addplot [color=black, mark=x, mark options={solid, black}, forget plot]
  table[row sep=crcr]{%
10	0.278920527501315\\
17	0.0913773522353813\\
28	0.0728482411229075\\
46	0.0380674891256041\\
77	0.0282757364307657\\
129	0.0186567836275219\\
215	0.0167324587918382\\
359	0.00933544426249954\\
599	0.00645306034842195\\
1000	0.00478009180954343\\
};
\addplot [color=mycolor1, mark=x, mark options={solid, mycolor1}, forget plot]
  table[row sep=crcr]{%
10	0.0967684756800368\\
17	0.13050061410019\\
28	0.0506068152587258\\
46	0.0254708034312771\\
77	0.0135316545816748\\
129	0.00875238443285514\\
215	0.0066793464918297\\
359	0.00486787001129551\\
599	0.00323106839198579\\
1000	0.00237150624825376\\
};
\addplot [color=mycolor2, mark=x, mark options={solid, mycolor2}, forget plot]
  table[row sep=crcr]{%
10	0.119830835991982\\
17	0.108115346896181\\
28	0.0503950460952773\\
46	0.0225782683369944\\
77	0.0159560879672142\\
129	0.00951303838914771\\
215	0.00669621006538111\\
359	0.00449459917236226\\
599	0.00330611617622218\\
1000	0.00244600263242521\\
};
\addplot [color=blue!75!mycolor2, mark=x, mark options={solid, blue!75!mycolor2}, forget plot]
  table[row sep=crcr]{%
10	0.138180693476498\\
17	0.0868208948009182\\
28	0.0475486367691238\\
46	0.0241843790893562\\
77	0.0119541378244635\\
129	0.00937820866247179\\
215	0.00676699881167107\\
359	0.00446459797072013\\
599	0.00337409437964302\\
1000	0.00245885375385847\\
};
\addplot [color=mycolor3, mark=x, mark options={solid, mycolor3}, forget plot]
  table[row sep=crcr]{%
10	0.103509752376338\\
17	0.0817824824016097\\
28	0.0627666913846203\\
46	0.0236674133694107\\
77	0.0174473708310282\\
129	0.00810460408051823\\
215	0.00621342876285095\\
359	0.00495549267041423\\
599	0.00312804979614788\\
1000	0.00239169266487798\\
};
\addplot [color=mycolor4, mark=x, mark options={solid, mycolor4}, forget plot]
  table[row sep=crcr]{%
10	0.0785421161843669\\
17	0.111389361030947\\
28	0.0608031629175977\\
46	0.0182354163458091\\
77	0.0124467075843315\\
129	0.00891638052758707\\
215	0.00573596831484961\\
359	0.0042744711964003\\
599	0.00337218810086686\\
1000	0.00250478668973289\\
};
\addplot [color=green!60!mycolor4, mark=x, mark options={solid, green!60!mycolor4}, forget plot]
  table[row sep=crcr]{%
10	0.211906039759258\\
17	0.123313390481886\\
28	0.0751509290074638\\
46	0.0323783565007068\\
77	0.0210361992260647\\
129	0.00939636625366074\\
215	0.00676373442515024\\
359	0.00476548305193828\\
599	0.00341236067255259\\
1000	0.00229426885945621\\
};
\addplot [color=mycolor5, mark=x, mark options={solid, mycolor5}, forget plot]
  table[row sep=crcr]{%
10	0.422125550540821\\
17	0.34537935698395\\
28	0.135811902415278\\
46	0.0726985478158326\\
77	0.0242802426827879\\
129	0.0135982846869018\\
215	0.0070855351377214\\
359	0.00496566164759942\\
599	0.00350509322416787\\
1000	0.00241363920112356\\
};
\addplot [color=mycolor6, mark=x, mark options={solid, mycolor6}, forget plot]
  table[row sep=crcr]{%
10	0.434130110651414\\
17	0.417197802784004\\
28	0.258777420725263\\
46	0.147702126876368\\
77	0.0576598197773609\\
129	0.0258172545327029\\
215	0.0129182989693198\\
359	0.00646713610080864\\
599	0.00414139692585789\\
1000	0.00282229760925802\\
};
\addplot [color=mycolor7, mark=x, mark options={solid, mycolor7}, forget plot]
  table[row sep=crcr]{%
10	0.546604308326962\\
17	0.462662425118106\\
28	0.413463230737193\\
46	0.368903509398961\\
77	0.23675900259775\\
129	0.126913915251947\\
215	0.0536493450029557\\
359	0.0216019170624971\\
599	0.00942664120920444\\
1000	0.00465227745391109\\
};
\addplot [color=red, mark=x, mark options={solid, red}, forget plot]
  table[row sep=crcr]{%
10	0.840699651046388\\
17	0.785743724580451\\
28	0.725749110382909\\
46	0.646851407490164\\
77	0.560225153245957\\
129	0.463671290574043\\
215	0.373803606267993\\
359	0.256446467422843\\
599	0.148046109152471\\
1000	0.0636096008473246\\
};
\end{axis}
\end{tikzpicture}%
%
%
\definecolor{mycolor1}{rgb}{1.00000,0.00000,0.60000}%
\definecolor{mycolor2}{rgb}{0.80000,0.00000,1.00000}%
\definecolor{mycolor3}{rgb}{0.00000,0.40000,1.00000}%
\definecolor{mycolor4}{rgb}{0.00000,1.00000,1.00000}%
\definecolor{mycolor5}{rgb}{0.20000,1.00000,0.00000}%
\definecolor{mycolor6}{rgb}{0.80000,1.00000,0.00000}%
\definecolor{mycolor7}{rgb}{1.00000,0.60000,0.00000}%
\begin{tikzpicture}

\begin{axis}[%
width=0.951\fwidth,
height=0.75\fwidth,
at={(0\fwidth,0\fwidth)},
scale only axis,
xmode=log,
xmin=10,
xmax=1000,
xminorticks=true,
ymode=log,
ymin=0.152667373432623,
ymax=3.35425226316319,
yminorticks=true,
axis background/.style={fill=white},
title style={font=\bfseries},
title={$\text{road}\_\text{n}\text{etwork}$},
axis x line*=bottom,
axis y line*=left
]
\addplot [color=black, mark=x, mark options={solid, black}, forget plot]
  table[row sep=crcr]{%
10	1.39010319500544\\
17	1.04557943661251\\
28	1.01754711362335\\
46	0.841744546670118\\
77	0.701103214197335\\
129	0.550095869851396\\
215	0.450731561647163\\
359	0.336192064011339\\
599	0.274882262485855\\
1000	0.185132327553611\\
};
\addplot [color=mycolor1, mark=x, mark options={solid, mycolor1}, forget plot]
  table[row sep=crcr]{%
10	2.88355940514954\\
17	3.35425226316319\\
28	1.09646087596263\\
46	1.27386210699161\\
77	1.01180364455189\\
129	0.740100450779944\\
215	0.498586442345621\\
359	0.353358604810697\\
599	0.253590151641001\\
1000	0.163060416508126\\
};
\addplot [color=mycolor2, mark=x, mark options={solid, mycolor2}, forget plot]
  table[row sep=crcr]{%
10	2.84662118772645\\
17	3.22510171901736\\
28	1.08490221487791\\
46	1.25425334239737\\
77	1.01071398470451\\
129	0.739212665817281\\
215	0.498709063951946\\
359	0.353393102486987\\
599	0.253669331477999\\
1000	0.163073768860639\\
};
\addplot [color=blue!75!mycolor2, mark=x, mark options={solid, blue!75!mycolor2}, forget plot]
  table[row sep=crcr]{%
10	2.47063940707557\\
17	2.01685789589099\\
28	1.21886322060983\\
46	1.05754939121395\\
77	0.793131590946824\\
129	0.767432663740669\\
215	0.567507011656511\\
359	0.36292216582627\\
599	0.253109363463471\\
1000	0.164640679101027\\
};
\addplot [color=mycolor3, mark=x, mark options={solid, mycolor3}, forget plot]
  table[row sep=crcr]{%
10	0.961861348300562\\
17	1.96215306921222\\
28	1.35335061830524\\
46	1.72140298806149\\
77	0.982965680383045\\
129	0.605086192836443\\
215	0.534161934855437\\
359	0.370317014449522\\
599	0.25959982119371\\
1000	0.153317444865653\\
};
\addplot [color=mycolor4, mark=x, mark options={solid, mycolor4}, forget plot]
  table[row sep=crcr]{%
10	1.34843479452793\\
17	1.87764906222235\\
28	1.1737825153425\\
46	1.01847130565033\\
77	0.886848926028881\\
129	0.676475982460393\\
215	0.51627872472494\\
359	0.36502704088248\\
599	0.258564986071592\\
1000	0.167673950934646\\
};
\addplot [color=green!60!mycolor4, mark=x, mark options={solid, green!60!mycolor4}, forget plot]
  table[row sep=crcr]{%
10	1.78316176723938\\
17	1.95577783117133\\
28	1.01145139273853\\
46	1.02199731101423\\
77	0.994659814648515\\
129	0.740654117787873\\
215	0.47815776272211\\
359	0.377216164950377\\
599	0.269517578176772\\
1000	0.162801633034408\\
};
\addplot [color=mycolor5, mark=x, mark options={solid, mycolor5}, forget plot]
  table[row sep=crcr]{%
10	2.15384669532063\\
17	1.54119650891179\\
28	1.53669229359549\\
46	0.997045330347858\\
77	0.910960331949531\\
129	0.741896202251458\\
215	0.460220542393935\\
359	0.415746087932765\\
599	0.263169168958436\\
1000	0.160238733020358\\
};
\addplot [color=mycolor6, mark=x, mark options={solid, mycolor6}, forget plot]
  table[row sep=crcr]{%
10	0.826037724913831\\
17	1.40254759293493\\
28	1.21601230904096\\
46	0.73065055278101\\
77	0.929390173039141\\
129	0.678429964293566\\
215	0.461148445866556\\
359	0.345735511336748\\
599	0.239575108709374\\
1000	0.152667373432623\\
};
\addplot [color=mycolor7, mark=x, mark options={solid, mycolor7}, forget plot]
  table[row sep=crcr]{%
10	0.859511414450988\\
17	0.891196971838094\\
28	1.27651143645104\\
46	0.808033868982211\\
77	0.704805514487124\\
129	0.545490236969442\\
215	0.514684383576006\\
359	0.355265554016244\\
599	0.273175507073664\\
1000	0.190558501598708\\
};
\addplot [color=red, mark=x, mark options={solid, red}, forget plot]
  table[row sep=crcr]{%
10	0.862579818193516\\
17	0.797779722174206\\
28	0.817930715426352\\
46	0.828917926290937\\
77	0.706160675645608\\
129	0.625444416149221\\
215	0.483565917006231\\
359	0.383922036145352\\
599	0.293622361345053\\
1000	0.228473224584309\\
};
\end{axis}
\end{tikzpicture}%
\else
\fi
\caption{Visualization of the test MSE ($y$-axis) over the number of greedily selected centers ($x$-axis) for the last six datasets. 
The black line shows the 2L-VKOGA, while the colored lines show the use of standard kernels with different length scale parameters $\varepsilon$.}
\label{fig:ml_datasets_results_2}
\end{figure}
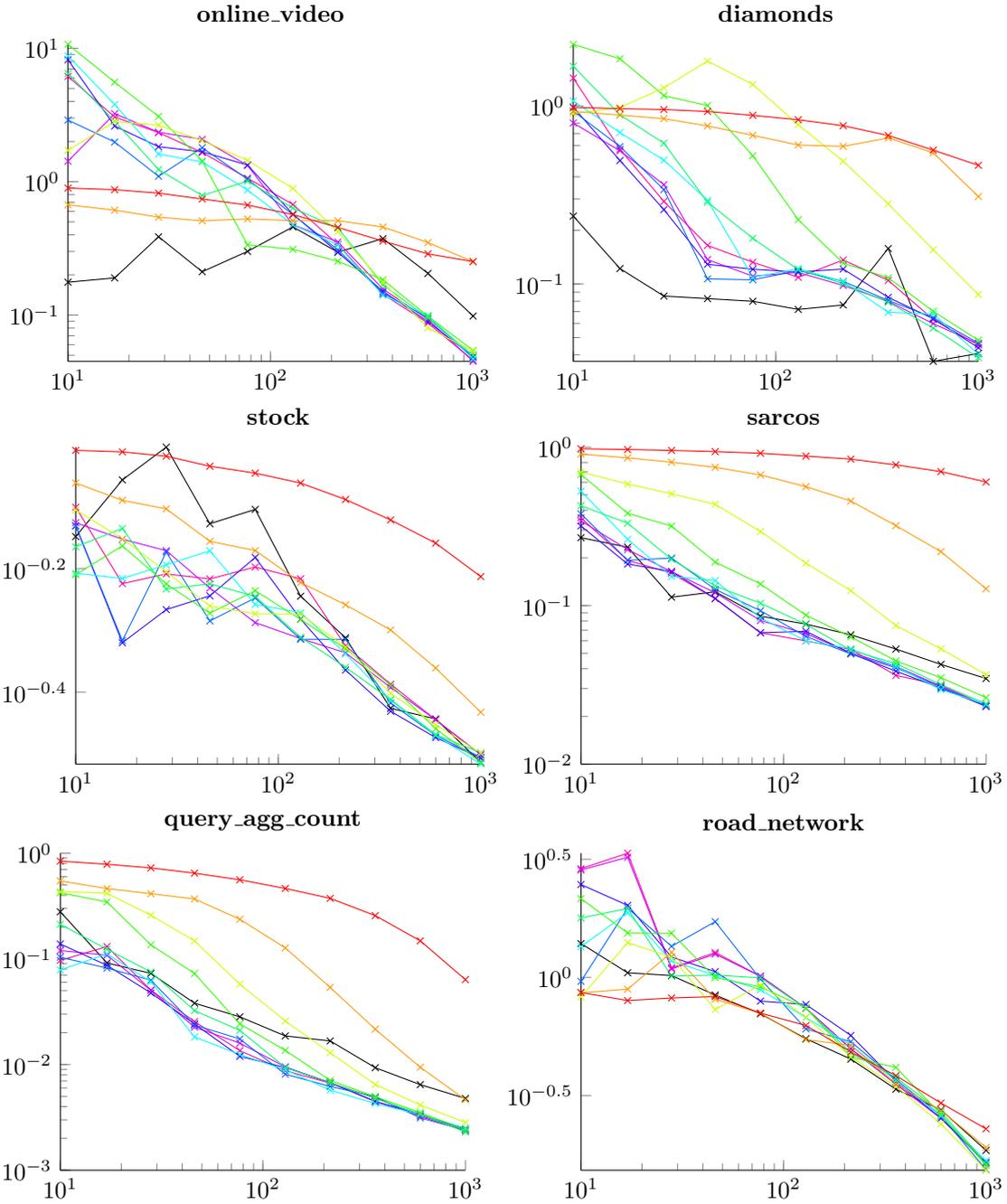

\begin{figure}[h]
\centering
\setlength\fwidth{.7\textwidth}
\ifplots
\input{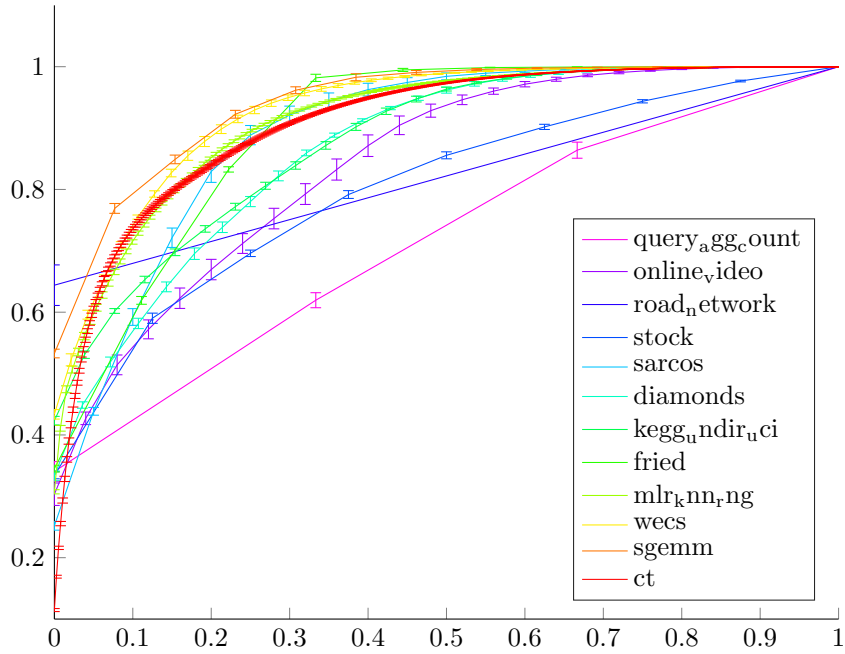}
\else
\fi
\caption{Visualization of the ratio $\sum_{i=1}^n |s_i(\mA_{\bs{\theta}})| / \sum_{i=1}^d |s_i(\mA_{\bs{\theta}})|$ (mean and standard deviation over 5 reruns) ($y$-axis) over the ratio $n/d$ of considered singular values ($x$-axis).
In conjunction with Figure \ref{fig:ml_datasets_results_1} and \ref{fig:ml_datasets_results_2} one can see that the 2L-VKOGA approach works better when the ratio of the singular approaches the value one more quickly.}
\label{fig:overview_singularvalues}
\end{figure}

\begin{table}[]
\centering 
\begin{tabular}{l|ll}
Short name & Number of examples & Number of features \\ \hline
fried & 40768 & 10 \\
sarcos & 44484 & 21 \\
ct & 53500 & 379 \\
diamonds & 53940 & 29 \\
stock & 59049 & 9 \\
kegg\_undir\_uci & 64608 & 27 \\
online\_video & 68784 & 26 \\
wecs & 72000 & 48 \\
mlr\_knn\_rng & 111753 & 132 \\
query\_agg\_count & 200000 & 4 \\
sgemm & 241600 & 14 \\
road\_network & 434874 & 2 \\
\end{tabular}
\caption{Overview on the machine learning data sets which were used.}
\label{tab:tabular_datasets}
\end{table}

\subsection{Stability of the kernel optimization} \label{subsec:num_stability}

In order to justify that the 2L-VKOGA approach indeed uses a stable data-driven kernel optimization, we employ different criteria for the kernel optimization. While Subsection \ref{subsec:num_ex_ml_data} relied on leave-one-out cross validation (i.e.\ $k$-fold cross validation for $k=64$, which was the batch size) for the optimization of the kernel, here we rerun the same experiments and make use of different values of $k$ for our cross validation, leveraging the efficient implementation described in Subsection \ref{subsec:loss_func_for_optim}. This is motivated by the fact that it is a priori unclear whether leave-one-out cross validation is a suitable choice for $k$-fold cross validation, or one should rather use a different value.

In Figure \ref{fig:stability_alignment} the results for three out of the twelve datasets are depicted. 
In the left column, the decay of the singular values is shown.
In the right column, the principal angle between subspaces according to \cite{knyazev2002principal} are plotted.
In detail, we depicted the principal angle between the subspace spanned by the first $n$ right singular vectors of the optimized matrix $\mA_{\bs{\theta}}$ for some method with the corresponding subspace of a benchmark method (for which we here used the leave-one-out cross validation method, i.e.\ $k=64$).
For a precise definition of the principal angle, see \cite[Section 1]{knyazev2002principal}.
For comparison reasons, also the corresponding quantities for the standard approach are computed, i.e.\ all the singular values are equal to one as we have $\mA_{\bs{\theta}} = \mathsf{I}_d$.
\begin{enumerate}
\item From the plots of the singular value decays we can infer that the singular value distribution is always very close to each other, in particular for large singular values.
Those large singular values are more important, as data in direction of singular vectors to larger singular values are stretched, while those in direction of smaller singular values are compressed. 
Furthermore one can clearly observe the difference of the classical approach, where all the singular values have the same value 1, i.e.\ all the directions within the dataset have equal importance.
\item From the plots of the overlap we can observe that the right singular vectors to the largest singular values of the matrix $\mA_{\boldsymbol{\theta}}$ are quite aligned (i.e.\ small angles) for all the optimization methods.
Only for larger subspaces, which however also correspond to less important singular values, the angle starts to grow.
\end{enumerate}
We remark that the plots for other datasets are qualitatively similar and showing them does not provide further insights.

Since both the (large) singular values and the corresponding singular vectors are similar we infer that the optimized matrices $\mA_{\boldsymbol{\theta}}$ are quite close to each other in the sense of detecting the same features and applying the same scalings to those features.
The differences within the smaller singular values and corresponding singular vectors actually do not negatively impact the performance of the kernel model, as the $f$-greedy selection criterion can accommodate for this. 
In fact the test error decays (i.e.\ Figure \ref{fig:ml_datasets_results_1}, \ref{fig:ml_datasets_results_2}) when using those slightly different matrices almost perfectly match (not shown here). 
We note that the optimized matrices $\mA_{\boldsymbol{\theta}}$ are not necessarily close to each other, which is related to the irrelevance of the right singular vectors for the overall kernel model as elaborated in Subsection \ref{subsubsec:analysis_first_layer}.

All in all these experiments using different $k$-fold cross validation parameters for our kernel optimization showed that the used optimization procedure is stable. 
This was desired, because we strived to obtain a data adapted kernel. 

For the experiments in Subsection \ref{subsec:num_ex_ml_data} we simply used $k=64$ (batch size) for the optimization as it is slightly more time-efficient than using another value of $k$.
However also the use of another $k$ fold cross validation parameter is not detrimental, as the time consumption for the kernel optimization is always smaller than the time required for the greedy selection, see also Table \ref{tab:tabular_results}.
This was however only possible by making use of the efficient implementation due the extended Rippa's algorithm, see Eq.\ \eqref{eq:rippa_loss}.

\begin{figure}[h]
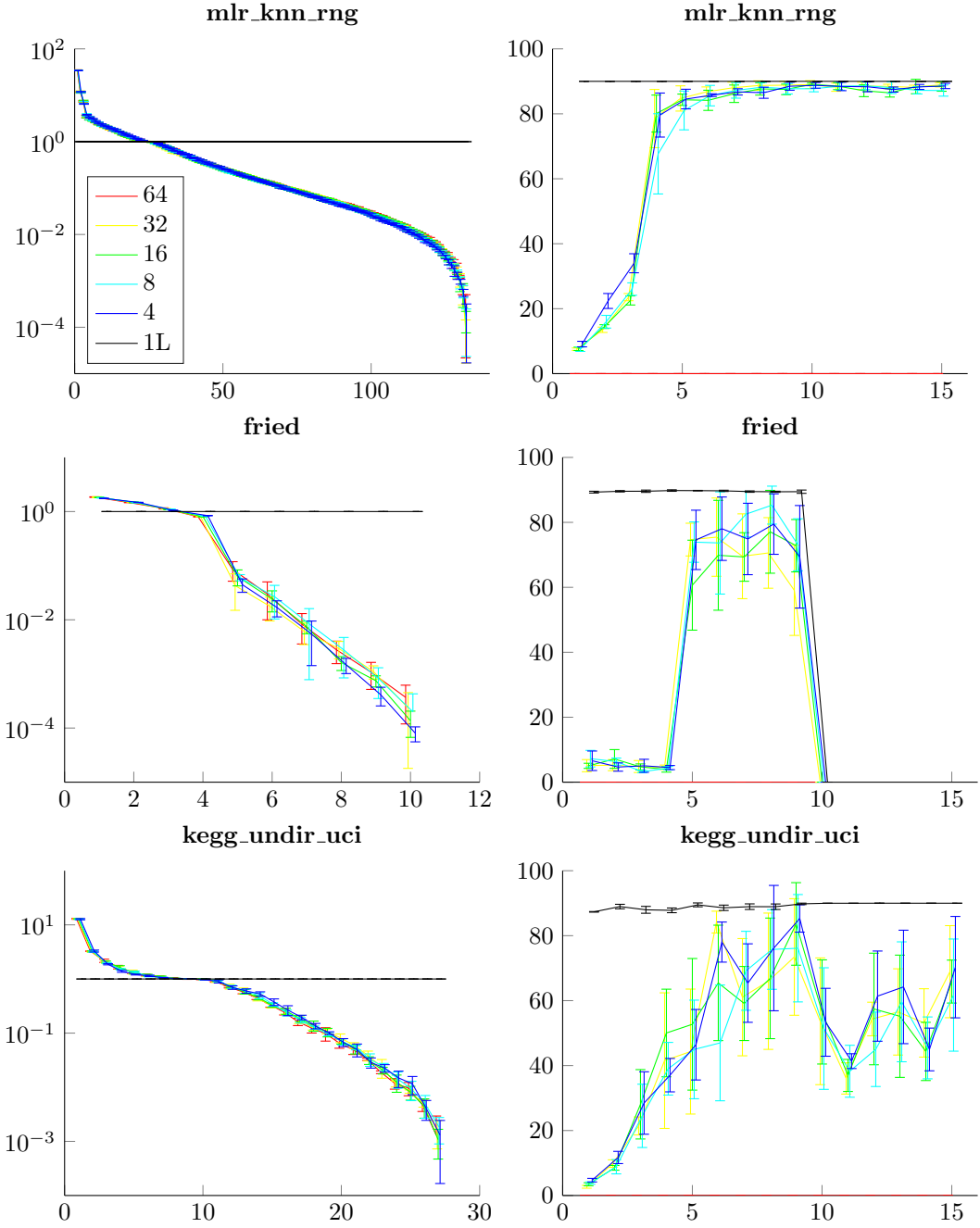

\centering
\setlength\fwidth{.4\textwidth}
\ifplots
\input{Figures/singvals_mlr_knn_rng.tex}
%
%
\definecolor{mycolor1}{rgb}{1.00000,1.00000,0.00000}%
\definecolor{mycolor2}{rgb}{0.00000,1.00000,1.00000}%
\begin{tikzpicture}

\begin{axis}[%
width=0.951\fwidth,
height=0.75\fwidth,
at={(0\fwidth,0\fwidth)},
scale only axis,
xmin=0,
xmax=16,
ymin=0,
ymax=100,
axis background/.style={fill=white},
title style={font=\bfseries},
title={mlr\_knn\_rng},
axis x line*=bottom,
axis y line*=left,
legend style={at={(0.97,0.03)}, anchor=south east, legend cell align=left, align=left, draw=white!15!black}
]
\addplot [color=red]
 plot [error bars/.cd, y dir = both, y explicit]
 table[row sep=crcr, y error plus index=2, y error minus index=3]{%
0.86	5.32496051164344e-06	4.92794106321526e-06	4.92794106321526e-06\\
1.86	1.51345257108915e-05	1.01744672065252e-05	1.01744672065252e-05\\
2.86	1.57169324666029e-05	1.11199233288062e-05	1.11199233288062e-05\\
3.86	1.84977961907862e-05	6.28324141871417e-06	6.28324141871417e-06\\
4.86	1.98347806872334e-05	4.34955563832773e-06	4.34955563832773e-06\\
5.86	2.91193955490598e-05	6.00157682129066e-06	6.00157682129066e-06\\
6.86	2.99840721709188e-05	9.31471367948689e-06	9.31471367948689e-06\\
7.86	2.64020000031451e-05	4.71247085442883e-06	4.71247085442883e-06\\
8.86	2.88328046735842e-05	4.64734648630838e-06	4.64734648630838e-06\\
9.86	3.25568180414848e-05	7.66345237934729e-06	7.66345237934729e-06\\
10.86	3.60374433512334e-05	1.15309912871453e-05	1.15309912871453e-05\\
11.86	3.7266181607265e-05	5.9245758166071e-06	5.9245758166071e-06\\
12.86	3.80630735890009e-05	4.62621892438619e-06	4.62621892438619e-06\\
13.86	4.01168945245445e-05	7.16747035767185e-06	7.16747035767185e-06\\
14.86	4.01223624066915e-05	5.07938239024952e-06	5.07938239024952e-06\\
};

\addplot [color=mycolor1]
 plot [error bars/.cd, y dir = both, y explicit]
 table[row sep=crcr, y error plus index=2, y error minus index=3]{%
0.93	7.53238582611084	0.374998241662979	0.374998241662979\\
1.93	13.5014600753784	0.815271556377411	0.815271556377411\\
2.93	23.5557041168213	1.09599018096924	1.09599018096924\\
3.93	78.541015625	8.95968437194824	8.95968437194824\\
4.93	84.7316665649414	3.88419485092163	3.88419485092163\\
5.93	86.7635879516602	1.51552736759186	1.51552736759186\\
6.93	87.9840545654297	1.76581466197968	1.76581466197968\\
7.93	88.7124786376953	1.44642794132233	1.44642794132233\\
8.93	88.8143463134766	1.18578505516052	1.18578505516052\\
9.93	89.005615234375	1.20655119419098	1.20655119419098\\
10.93	88.1957550048828	1.26319372653961	1.26319372653961\\
11.93	88.6000823974609	1.11692774295807	1.11692774295807\\
12.93	88.4247207641602	0.674286603927612	0.674286603927612\\
13.93	88.152587890625	0.656468212604523	0.656468212604523\\
14.93	88.9222640991211	0.639569640159607	0.639569640159607\\
};

\addplot [color=green]
 plot [error bars/.cd, y dir = both, y explicit]
 table[row sep=crcr, y error plus index=2, y error minus index=3]{%
1	7.63374805450439	0.528828859329224	0.528828859329224\\
2	14.4796276092529	0.575391590595245	0.575391590595245\\
3	22.5352993011475	1.39741516113281	1.39741516113281\\
4	80.0746383666992	5.69393873214722	5.69393873214722\\
5	84.3341064453125	1.7427750825882	1.7427750825882\\
6	84.1588745117188	3.10691738128662	3.10691738128662\\
7	86.1787261962891	2.68668246269226	2.68668246269226\\
8	87.661491394043	1.89548850059509	1.89548850059509\\
9	87.5260620117188	1.61534738540649	1.61534738540649\\
10	89.0672988891602	0.410112023353577	0.410112023353577\\
11	88.5258407592773	1.07805192470551	1.07805192470551\\
12	87.1090698242188	1.82009267807007	1.82009267807007\\
13	86.4640960693359	1.2996279001236	1.2996279001236\\
14	88.4047241210938	2.22595238685608	2.22595238685608\\
15	88.4383773803711	1.49258708953857	1.49258708953857\\
};

\addplot [color=mycolor2]
 plot [error bars/.cd, y dir = both, y explicit]
 table[row sep=crcr, y error plus index=2, y error minus index=3]{%
1.07	7.52194976806641	0.679533183574677	0.679533183574677\\
2.07	15.9859828948975	1.9549241065979	1.9549241065979\\
3.07	26.0387382507324	1.91136467456818	1.91136467456818\\
4.07	67.6805114746094	12.3578863143921	12.3578863143921\\
5.07	81.0373153686523	5.97173023223877	5.97173023223877\\
6.07	85.5986557006836	3.15336894989014	3.15336894989014\\
7.07	87.2156295776367	2.393070936203	2.393070936203\\
8.07	88.1892929077148	1.46403586864471	1.46403586864471\\
9.07	87.7133255004883	1.60764682292938	1.60764682292938\\
10.07	87.7313232421875	0.953594088554382	0.953594088554382\\
11.07	88.5068359375	1.23686146736145	1.23686146736145\\
12.07	88.6141891479492	1.14743995666504	1.14743995666504\\
13.07	88.0611343383789	1.81118762493134	1.81118762493134\\
14.07	87.2361221313477	1.12035059928894	1.12035059928894\\
15.07	87.1526565551758	1.71133685112	1.71133685112\\
};

\addplot [color=blue]
 plot [error bars/.cd, y dir = both, y explicit]
 table[row sep=crcr, y error plus index=2, y error minus index=3]{%
1.14	9.10523319244385	0.85183572769165	0.85183572769165\\
2.14	22.3721656799316	2.30081295967102	2.30081295967102\\
3.14	33.9875679016113	2.92663621902466	2.92663621902466\\
4.14	79.6183395385742	6.77360582351685	6.77360582351685\\
5.14	84.5434341430664	3.02357983589172	3.02357983589172\\
6.14	85.7456359863281	0.423344045877457	0.423344045877457\\
7.14	86.7536239624023	0.961765110492706	0.961765110492706\\
8.14	86.4914398193359	1.70848596096039	1.70848596096039\\
9.14	88.4914016723633	1.22741830348969	1.22741830348969\\
10.14	88.7572784423828	0.867274761199951	0.867274761199951\\
11.14	88.4263229370117	1.29987001419067	1.29987001419067\\
12.14	88.3199157714844	1.54217112064362	1.54217112064362\\
13.14	87.4690246582031	0.788071632385254	0.788071632385254\\
14.14	88.2828521728516	0.767221331596375	0.767221331596375\\
15.14	88.5755920410156	0.814777612686157	0.814777612686157\\
};

\addplot [color=black]
 plot [error bars/.cd, y dir = both, y explicit]
 table[row sep=crcr, y error plus index=2, y error minus index=3]{%
1.21	90	0	0\\
2.21	90	0	0\\
3.21	90	0	0\\
4.21	90	3.41196891895379e-06	3.41196891895379e-06\\
5.21	90	0	0\\
6.21	90	0	0\\
7.21	90	0	0\\
8.21	90	3.41196891895379e-06	3.41196891895379e-06\\
9.21	90	3.41196891895379e-06	3.41196891895379e-06\\
10.21	90	3.41196891895379e-06	3.41196891895379e-06\\
11.21	90	0	0\\
12.21	90	3.41196891895379e-06	3.41196891895379e-06\\
13.21	90	0	0\\
14.21	90	0	0\\
15.21	90	0	0\\
};

\end{axis}
\end{tikzpicture}%
%
%
\definecolor{mycolor1}{rgb}{1.00000,1.00000,0.00000}%
\definecolor{mycolor2}{rgb}{0.00000,1.00000,1.00000}%
\begin{tikzpicture}

\begin{axis}[%
width=0.951\fwidth,
height=0.75\fwidth,
at={(0\fwidth,0\fwidth)},
scale only axis,
xmin=0,
xmax=12,
ymode=log,
ymin=1e-05,
ymax=10,
yminorticks=true,
axis background/.style={fill=white},
title style={font=\bfseries},
title={fried},
axis x line*=bottom,
axis y line*=left,
legend style={at={(0.03,0.03)}, anchor=south west, legend cell align=left, align=left, draw=white!15!black}
]
\addplot [color=red]
 plot [error bars/.cd, y dir = both, y explicit]
 table[row sep=crcr, y error plus index=2, y error minus index=3]{%
0.86	1.84910273551941	0.0252321511507034	0.0252321511507034\\
1.86	1.46812796592712	0.0128739168867469	0.0128739168867469\\
2.86	1.14365649223328	0.0280170384794474	0.0280170384794474\\
3.86	0.802081406116486	0.00646516727283597	0.00646516727283597\\
4.86	0.0850364044308662	0.033416859805584	0.033416859805584\\
5.86	0.0300779156386852	0.0200807061046362	0.0200807061046362\\
6.86	0.0082957586273551	0.00472257798537612	0.00472257798537612\\
7.86	0.00279023544862866	0.00124174158554524	0.00124174158554524\\
8.86	0.00107577512972057	0.000558281608391553	0.000558281608391553\\
9.86	0.000369307730579749	0.000250082521233708	0.000250082521233708\\
};

\addplot [color=mycolor1]
 plot [error bars/.cd, y dir = both, y explicit]
 table[row sep=crcr, y error plus index=2, y error minus index=3]{%
0.93	1.83664703369141	0.0405893437564373	0.0405893437564373\\
1.93	1.49518728256226	0.030299324542284	0.030299324542284\\
2.93	1.12139022350311	0.0193845611065626	0.0193845611065626\\
3.93	0.804196357727051	0.00550323352217674	0.00550323352217674\\
4.93	0.0428102649748325	0.0278447009623051	0.0278447009623051\\
5.93	0.0178587101399899	0.0083374073728919	0.0083374073728919\\
6.93	0.00609118537977338	0.00261076889000833	0.00261076889000833\\
7.93	0.00294121308252215	0.00104304659180343	0.00104304659180343\\
8.93	0.00110579375177622	0.000304274290101603	0.000304274290101603\\
9.93	0.000234534221817739	0.000216525309951976	0.000216525309951976\\
};

\addplot [color=green]
 plot [error bars/.cd, y dir = both, y explicit]
 table[row sep=crcr, y error plus index=2, y error minus index=3]{%
1	1.82322525978088	0.0233384035527706	0.0233384035527706\\
2	1.48410594463348	0.0280719231814146	0.0280719231814146\\
3	1.09038209915161	0.00844744686037302	0.00844744686037302\\
4	0.805469870567322	0.0118987038731575	0.0118987038731575\\
5	0.0628900155425072	0.0205649752169847	0.0205649752169847\\
6	0.0241405609995127	0.0101705510169268	0.0101705510169268\\
7	0.00748347956687212	0.00195485143922269	0.00195485143922269\\
8	0.00162104947958142	0.000464589567855	0.000464589567855\\
9	0.00074857915751636	0.000174229804542847	0.000174229804542847\\
10	0.000136328948428854	6.90848464728333e-05	6.90848464728333e-05\\
};

\addplot [color=mycolor2]
 plot [error bars/.cd, y dir = both, y explicit]
 table[row sep=crcr, y error plus index=2, y error minus index=3]{%
1.07	1.79155850410461	0.0234783105552197	0.0234783105552197\\
2.07	1.46955931186676	0.0239106994122267	0.0239106994122267\\
3.07	1.0947573184967	0.0187467355281115	0.0187467355281115\\
4.07	0.824438393115997	0.00507907895371318	0.00507907895371318\\
5.07	0.0604236945509911	0.00784828793257475	0.00784828793257475\\
6.07	0.0267944987863302	0.0164422262459993	0.0164422262459993\\
7.07	0.0084362281486392	0.00765310367569327	0.00765310367569327\\
8.07	0.00278073409572244	0.00193646864499897	0.00193646864499897\\
9.07	0.000825408089440316	0.000474546948680654	0.000474546948680654\\
10.07	0.000199519723537378	0.000229540557484142	0.000229540557484142\\
};

\addplot [color=blue]
 plot [error bars/.cd, y dir = both, y explicit]
 table[row sep=crcr, y error plus index=2, y error minus index=3]{%
1.14	1.75127279758453	0.0212541036307812	0.0212541036307812\\
2.14	1.47409343719482	0.00882690586149693	0.00882690586149693\\
3.14	1.04942071437836	0.0175731088966131	0.0175731088966131\\
4.14	0.840126991271973	0.00610914500430226	0.00610914500430226\\
5.14	0.0445030704140663	0.0123560708016157	0.0123560708016157\\
6.14	0.0168748553842306	0.00555225322023034	0.00555225322023034\\
7.14	0.00545098399743438	0.00402723252773285	0.00402723252773285\\
8.14	0.00147716782521456	0.000462035997770727	0.000462035997770727\\
9.14	0.000411617191275582	0.00015523012552876	0.00015523012552876\\
10.14	8.03243601694703e-05	2.47026837314479e-05	2.47026837314479e-05\\
};

\addplot [color=black]
 plot [error bars/.cd, y dir = both, y explicit]
 table[row sep=crcr, y error plus index=2, y error minus index=3]{%
1.21	1	0	0\\
2.21	1	0	0\\
3.21	1	0	0\\
4.21	1	0	0\\
5.21	1	0	0\\
6.21	1	0	0\\
7.21	1	0	0\\
8.21	1	0	0\\
9.21	1	0	0\\
10.21	1	0	0\\
};

\end{axis}
\end{tikzpicture}%
%
%
\definecolor{mycolor1}{rgb}{1.00000,1.00000,0.00000}%
\definecolor{mycolor2}{rgb}{0.00000,1.00000,1.00000}%
\begin{tikzpicture}

\begin{axis}[%
width=0.951\fwidth,
height=0.75\fwidth,
at={(0\fwidth,0\fwidth)},
scale only axis,
xmin=0,
xmax=16,
ymin=0,
ymax=100,
axis background/.style={fill=white},
title style={font=\bfseries},
title={fried},
axis x line*=bottom,
axis y line*=left,
legend style={at={(0.97,0.03)}, anchor=south east, legend cell align=left, align=left, draw=white!15!black}
]
\addplot [color=red]
 plot [error bars/.cd, y dir = both, y explicit]
 table[row sep=crcr, y error plus index=2, y error minus index=3]{%
0.86	3.874818503391e-06	1.93741766452149e-06	1.93741766452149e-06\\
1.86	9.38317953114165e-06	5.8766850088432e-06	5.8766850088432e-06\\
2.86	1.69393861142453e-05	9.53336711972952e-06	9.53336711972952e-06\\
3.86	1.79643575393129e-05	9.2183836386539e-06	9.2183836386539e-06\\
4.86	2.47582847805461e-05	7.35946150598465e-06	7.35946150598465e-06\\
5.86	1.92610641533975e-05	3.887694219884e-06	3.887694219884e-06\\
6.86	2.29084344027797e-05	6.92205503582954e-06	6.92205503582954e-06\\
7.86	2.2245181753533e-05	7.03465002516168e-06	7.03465002516168e-06\\
8.86	2.09799472941086e-05	5.62569721296313e-06	5.62569721296313e-06\\
9.86	2.63297806668561e-05	8.41110613691853e-06	8.41110613691853e-06\\
};

\addplot [color=mycolor1]
 plot [error bars/.cd, y dir = both, y explicit]
 table[row sep=crcr, y error plus index=2, y error minus index=3]{%
0.93	5.08681774139404	1.86833477020264	1.86833477020264\\
1.93	5.25029277801514	1.69958090782166	1.69958090782166\\
2.93	4.81788730621338	1.84037327766418	1.84037327766418\\
3.93	4.57057857513428	0.697074055671692	0.697074055671692\\
4.93	74.6946868896484	5.03474950790405	5.03474950790405\\
5.93	75.4451065063477	12.0491104125977	12.0491104125977\\
6.93	69.5624465942383	13.05504322052	13.05504322052\\
7.93	70.5836410522461	10.8660707473755	10.8660707473755\\
8.93	59.0627212524414	13.8768253326416	13.8768253326416\\
9.93	2.56154125963803e-05	6.92689400239033e-06	6.92689400239033e-06\\
};

\addplot [color=green]
 plot [error bars/.cd, y dir = both, y explicit]
 table[row sep=crcr, y error plus index=2, y error minus index=3]{%
1	5.03328227996826	0.786871016025543	0.786871016025543\\
2	7.14795780181885	2.88931655883789	2.88931655883789\\
3	4.70194911956787	0.918432712554932	0.918432712554932\\
4	3.93019819259644	0.828965961933136	0.828965961933136\\
5	60.6638870239258	13.8468341827393	13.8468341827393\\
6	69.8641052246094	16.9180335998535	16.9180335998535\\
7	69.3444366455078	7.47795867919922	7.47795867919922\\
8	77.0991058349609	12.7553262710571	12.7553262710571\\
9	72.8096694946289	8.06084823608398	8.06084823608398\\
10	2.97695860353997e-05	1.10407436295645e-05	1.10407436295645e-05\\
};

\addplot [color=mycolor2]
 plot [error bars/.cd, y dir = both, y explicit]
 table[row sep=crcr, y error plus index=2, y error minus index=3]{%
1.07	7.24149465560913	2.54523205757141	2.54523205757141\\
2.07	6.32692527770996	1.09702908992767	1.09702908992767\\
3.07	3.23139190673828	0.339772820472717	0.339772820472717\\
4.07	4.32500219345093	0.505148231983185	0.505148231983185\\
5.07	73.9172897338867	6.21500635147095	6.21500635147095\\
6.07	73.6553192138672	15.7363557815552	15.7363557815552\\
7.07	82.5291442871094	7.03040647506714	7.03040647506714\\
8.07	85.3044586181641	5.84578037261963	5.84578037261963\\
9.07	73.0973281860352	8.00978851318359	8.00978851318359\\
10.07	2.91750220640097e-05	8.80195693753194e-06	8.80195693753194e-06\\
};

\addplot [color=blue]
 plot [error bars/.cd, y dir = both, y explicit]
 table[row sep=crcr, y error plus index=2, y error minus index=3]{%
1.14	6.59375476837158	3.00182962417603	3.00182962417603\\
2.14	4.66383266448975	1.27753508090973	1.27753508090973\\
3.14	5.04787349700928	2.00338196754456	2.00338196754456\\
4.14	4.49463367462158	0.6351717710495	0.6351717710495\\
5.14	74.6074142456055	9.13517475128174	9.13517475128174\\
6.14	78.0448303222656	9.76110744476318	9.76110744476318\\
7.14	74.8841934204102	10.9972801208496	10.9972801208496\\
8.14	79.4724044799805	9.30877590179443	9.30877590179443\\
9.14	69.3948287963867	15.7861309051514	15.7861309051514\\
10.14	2.44075145019451e-05	4.3506688598427e-06	4.3506688598427e-06\\
};

\addplot [color=black]
 plot [error bars/.cd, y dir = both, y explicit]
 table[row sep=crcr, y error plus index=2, y error minus index=3]{%
1.21	89.3086471557617	0.274709701538086	0.274709701538086\\
2.21	89.571662902832	0.213149428367615	0.213149428367615\\
3.21	89.5870361328125	0.357576668262482	0.357576668262482\\
4.21	89.7988739013672	0.23363932967186	0.23363932967186\\
5.21	89.7676315307617	0.0860228762030602	0.0860228762030602\\
6.21	89.7053451538086	0.195801243185997	0.195801243185997\\
7.21	89.4793014526367	0.21904593706131	0.21904593706131\\
8.21	89.4217300415039	0.177726224064827	0.177726224064827\\
9.21	89.4444351196289	0.507743060588837	0.507743060588837\\
10.21	0	0	0\\
};

\end{axis}
\end{tikzpicture}%
\input{Figures/singvals_kegg_undir_uci.tex}
%
%
\definecolor{mycolor1}{rgb}{1.00000,1.00000,0.00000}%
\definecolor{mycolor2}{rgb}{0.00000,1.00000,1.00000}%
\begin{tikzpicture}

\begin{axis}[%
width=0.951\fwidth,
height=0.75\fwidth,
at={(0\fwidth,0\fwidth)},
scale only axis,
xmin=0,
xmax=16,
ymin=-1.21306402434129e-06,
ymax=100,
axis background/.style={fill=white},
title style={font=\bfseries},
title={kegg\_undir\_uci},
axis x line*=bottom,
axis y line*=left,
legend style={at={(0.97,0.03)}, anchor=south east, legend cell align=left, align=left, draw=white!15!black}
]
\addplot [color=red]
 plot [error bars/.cd, y dir = both, y explicit]
 table[row sep=crcr, y error plus index=2, y error minus index=3]{%
0.86	3.80997380489134e-06	5.02303782923263e-06	5.02303782923263e-06\\
1.86	9.23410243558465e-06	6.258556368266e-06	6.258556368266e-06\\
2.86	1.31922761283931e-05	5.34570835952763e-06	5.34570835952763e-06\\
3.86	1.71877709362889e-05	4.81237339045038e-06	4.81237339045038e-06\\
4.86	2.13309249375015e-05	9.70732617133763e-06	9.70732617133763e-06\\
5.86	2.8715734515572e-05	1.09586853795918e-05	1.09586853795918e-05\\
6.86	2.78076677204808e-05	8.92918978934176e-06	8.92918978934176e-06\\
7.86	2.91930227831472e-05	1.14034955913667e-05	1.14034955913667e-05\\
8.86	3.25638502545189e-05	8.79898107086774e-06	8.79898107086774e-06\\
9.86	3.64444094884675e-05	1.04986665974138e-05	1.04986665974138e-05\\
10.86	3.57903409167193e-05	4.87496208734228e-06	4.87496208734228e-06\\
11.86	3.10045943479054e-05	2.16000262298621e-06	2.16000262298621e-06\\
12.86	3.44891814165749e-05	4.44319675807492e-06	4.44319675807492e-06\\
13.86	4.13101479352918e-05	1.57642011799908e-06	1.57642011799908e-06\\
14.86	4.33250243077055e-05	9.05851447896566e-06	9.05851447896566e-06\\
};

\addplot [color=mycolor1]
 plot [error bars/.cd, y dir = both, y explicit]
 table[row sep=crcr, y error plus index=2, y error minus index=3]{%
0.93	2.68240022659302	0.428594887256622	0.428594887256622\\
1.93	9.82885551452637	1.15725231170654	1.15725231170654\\
2.93	21.4236087799072	2.75950074195862	2.75950074195862\\
3.93	41.5295486450195	20.8764247894287	20.8764247894287\\
4.93	44.3186149597168	19.2391777038574	19.2391777038574\\
5.93	84.1846542358398	3.35659885406494	3.35659885406494\\
6.93	61.0652885437012	18.0480823516846	18.0480823516846\\
7.93	65.9743804931641	20.9858131408691	20.9858131408691\\
8.93	73.4394073486328	17.9488410949707	17.9488410949707\\
9.93	53.639232635498	19.5103759765625	19.5103759765625\\
10.93	36.0563011169434	4.82502937316895	4.82502937316895\\
11.93	54.340648651123	5.17525482177734	5.17525482177734\\
12.93	56.5232734680176	13.2808208465576	13.2808208465576\\
13.93	53.1350479125977	10.5260715484619	10.5260715484619\\
14.93	68.9146957397461	14.176833152771	14.176833152771\\
};

\addplot [color=green]
 plot [error bars/.cd, y dir = both, y explicit]
 table[row sep=crcr, y error plus index=2, y error minus index=3]{%
1	3.46845555305481	0.465590864419937	0.465590864419937\\
2	8.52420711517334	0.7235227227211	0.7235227227211\\
3	28.0813598632812	10.6801118850708	10.6801118850708\\
4	50.0366630554199	13.4903926849365	13.4903926849365\\
5	52.7099914550781	20.2581729888916	20.2581729888916\\
6	65.4633026123047	17.7689552307129	17.7689552307129\\
7	59.1166458129883	11.3995876312256	11.3995876312256\\
8	66.8602294921875	18.5657844543457	18.5657844543457\\
9	83.6110916137695	12.7138452529907	12.7138452529907\\
10	56.5157470703125	16.0346965789795	16.0346965789795\\
11	36.9627456665039	4.93602085113525	4.93602085113525\\
12	57.4296760559082	17.1767272949219	17.1767272949219\\
13	55.1896781921387	18.8203029632568	18.8203029632568\\
14	44.3585433959961	8.96016216278076	8.96016216278076\\
15	65.8652191162109	6.61202669143677	6.61202669143677\\
};

\addplot [color=mycolor2]
 plot [error bars/.cd, y dir = both, y explicit]
 table[row sep=crcr, y error plus index=2, y error minus index=3]{%
1.07	3.81540489196777	0.355655044317245	0.355655044317245\\
2.07	8.85348033905029	2.19184160232544	2.19184160232544\\
3.07	24.5256462097168	9.74860191345215	9.74860191345215\\
4.07	39.0293502807617	8.10015487670898	8.10015487670898\\
5.07	44.9717178344727	15.1814775466919	15.1814775466919\\
6.07	46.9867172241211	17.8246917724609	17.8246917724609\\
7.07	69.1805267333984	12.1631269454956	12.1631269454956\\
8.07	75.7536315917969	12.2668428421021	12.2668428421021\\
9.07	76.1378402709961	16.530969619751	16.530969619751\\
10.07	51.3991470336914	18.7178611755371	18.7178611755371\\
11.07	38.2515029907227	7.99645090103149	7.99645090103149\\
12.07	45.0165634155273	11.4596538543701	11.4596538543701\\
13.07	59.6386642456055	18.4623508453369	18.4623508453369\\
14.07	45.450065612793	9.52725028991699	9.52725028991699\\
15.07	61.7100639343262	17.2658767700195	17.2658767700195\\
};

\addplot [color=blue]
 plot [error bars/.cd, y dir = both, y explicit]
 table[row sep=crcr, y error plus index=2, y error minus index=3]{%
1.14	4.72839593887329	0.549733221530914	0.549733221530914\\
2.14	11.7067050933838	1.88411843776703	1.88411843776703\\
3.14	28.4827003479004	9.60907077789307	9.60907077789307\\
4.14	37.0527191162109	5.14223718643188	5.14223718643188\\
5.14	46.4264297485352	10.854549407959	10.854549407959\\
6.14	78.0285186767578	6.17288303375244	6.17288303375244\\
7.14	65.4087524414062	12.0846319198608	12.0846319198608\\
8.14	76.1768646240234	19.2973289489746	19.2973289489746\\
9.14	85.2950592041016	4.20898628234863	4.20898628234863\\
10.14	53.3351058959961	10.4520816802979	10.4520816802979\\
11.14	41.3694915771484	2.28042483329773	2.28042483329773\\
12.14	61.388256072998	13.8853435516357	13.8853435516357\\
13.14	64.2191467285156	17.4674186706543	17.4674186706543\\
14.14	44.9677696228027	6.57700729370117	6.57700729370117\\
15.14	70.3096466064453	15.6116561889648	15.6116561889648\\
};

\addplot [color=black]
 plot [error bars/.cd, y dir = both, y explicit]
 table[row sep=crcr, y error plus index=2, y error minus index=3]{%
1.21	87.3410797119141	0.094890721142292	0.094890721142292\\
2.21	88.9645767211914	0.695136785507202	0.695136785507202\\
3.21	87.962516784668	1.05625915527344	1.05625915527344\\
4.21	87.7980499267578	0.72417289018631	0.72417289018631\\
5.21	89.4856948852539	0.593363106250763	0.593363106250763\\
6.21	88.5842132568359	0.830021023750305	0.830021023750305\\
7.21	88.9318618774414	0.861953854560852	0.861953854560852\\
8.21	88.9153900146484	0.846308648586273	0.846308648586273\\
9.21	89.7990875244141	0.225253611803055	0.225253611803055\\
10.21	89.9863510131836	0.0190186444669962	0.0190186444669962\\
11.21	89.9847106933594	0.0198125615715981	0.0198125615715981\\
12.21	90	0	0\\
13.21	89.9960479736328	0.00791626237332821	0.00791626237332821\\
14.21	89.9944000244141	0.0111907962709665	0.0111907962709665\\
15.21	89.9960479736328	0.00791626237332821	0.00791626237332821\\
};

\end{axis}
\end{tikzpicture}%
\else
\fi
\caption{Analysis of the optimized matrices $\mA_{\boldsymbol{\theta}}$ for three datasets. 
Left: Visualization of the singular values of the matrix $\mA_{\boldsymbol{\theta}}$. 
Right: Visualization of the principal angle ($y$-axis) between subspaces spanned by the first $n$ right singular vectors ($x$-axis). 
We set the number of folds as $n_\text{fold} = 64$ as reference and compare to $n_\text{fold} \in \{64, 32, 16, 8, 4 \}$ and no optimization (i.e. the classical VKOGA).
Errorbars are used to visualize the mean and standard deviation for 5 reruns. 
}
\label{fig:stability_alignment}
\end{figure}

\section{Conclusion and outlook}
\label{sec:conclusion}

In this paper we introduced a machine learning way of choosing kernel hyperparameters, which is done by a stochastic gradient descent optimization.
For the optimization, we leveraged a recent efficient way to compute $k$-fold cross validation errors and compared their performances.
Especially those hyperparameter optimized kernels can be seen as a two-layered kernel machine, i.e.\ a deep kernel. 
The method was used in conjunction with greedy methods to select proper data points in order to obtain sparse models. 
Fundamental analysis on the proposed method was provided as well as experiments on synthetic and real world data, 
which highlight the benefits of the approach. 

Future work aims at combining the kernel optimization with the greedy selection procedure and generalizing the first layer map. Instead of using only linear kernels which give rise to linear mappings, the use of nonlinear kernels seems appealing.

\textbf{Acknowledgements:} The first author acknowledges the funding of the project by the Deutsche Forschungsgemeinschaft (DFG, German Research Foundation) under
Germany's Excellence Strategy - EXC 2075 - 390740016 and funding by the BMBF under contract 05M20VSA. 
The first author thanks Gabriele Santin and Bernard Haasdonk for discussions. 
The second author acknowledges the financial support of the Programma Operativo Nazionale (PON) "Ricerca e Innovazione" 2014 - 2020.
The research of the second and third authors has been accomplished within the Italian Network on Approximation (RITA), 
the thematic group on \lq\lq Approximation Theory and Applications\rq\rq \hskip 0.1cm (TAA) of the Italian Mathematical Union (UMI) and partially funded by the GNCS-IN$\delta$AM.

\bibliography{references}
\bibliographystyle{abbrv}

\end{document}